\documentclass[english, 12pt]{smfart}
\usepackage{amsmath,amsthm,amssymb}
\usepackage{dsfont}
\usepackage{mathrsfs} 
\usepackage[all,cmtip]{xy}
\usepackage{patchcmd}
\input{diagxy}
\usepackage{url}
\usepackage{bm}
\usepackage{mathtools}
\usepackage{enumerate}
\usepackage{pbox}
\usepackage{color}
\usepackage{stmaryrd}

 \theoremstyle{plain}
 \newtheorem{thm}{Theorem}[section]
 \newtheorem{cor}[thm]{Corollary}
 \newtheorem{lem}[thm]{Lemma}
 
 \newtheorem{prop}[thm]{Proposition}

\theoremstyle{definition}
 \newtheorem{defn}[thm]{Definition}

\theoremstyle{remark}
 \newtheorem{rem}[thm]{Remark}

 \newtheorem{ter}[thm]{Terminology}
 \newtheorem{nota}[thm]{Notation}
 \newtheorem{conv}[thm]{Convention}
 \newtheorem{exam}[thm]{Example}

 \numberwithin{equation}{section}
 \renewcommand{\theequation}{\thesection.\arabic{equation}}

\DeclareMathOperator{\VF}{VF}

\DeclareMathOperator{\RV}{RV}
\DeclareMathOperator{\DC}{DC}
\DeclareMathOperator{\MM}{\mathcal{M}}

\DeclareMathOperator{\OO}{\mathcal{O}}

\DeclareMathOperator{\UU}{\mathcal{U}}

 \DeclareMathOperator{\ran}{ran}
 \DeclareMathOperator{\dom}{dom}

 \DeclareMathOperator{\id}{id}

 \DeclareMathOperator{\lh}{lh}

 \DeclareMathOperator{\dcl}{dcl}
 \DeclareMathOperator{\pr}{pr}

 \DeclareMathOperator{\mgl}{GL}

\DeclareMathOperator{\jcb}{Jcb}

\DeclareMathOperator{\K}{\Bbbk}

\def\Xint#1{\mathchoice
{\XXint\displaystyle\textstyle{#1}}%
{\XXint\textstyle\scriptstyle{#1}}%
{\XXint\scriptstyle\scriptscriptstyle{#1}}%
{\XXint\scriptscriptstyle\scriptscriptstyle{#1}}%
\!\int}
\def\XXint#1#2#3{{\setbox0=\hbox{$#1{#2#3}{\int}$}
\vcenter{\hbox{$#2#3$}}\kern-.5\wd0}}

\newcommand{\Z}{\mathds{Z}}

\newcommand{\Q}{\mathds{Q}}
\newcommand{\N}{\mathds{N}}

\newcommand{\R}{\mathds{R}}

\newcommand{\omin}{$o$\nobreakdash}

\newcommand{\cmin}{$C$\nobreakdash}

\newcommand{\gB}{\mathfrak{B}}
\newcommand{\gC}{\mathfrak{C}}

\newcommand{\ga}{\mathfrak{a}}
\newcommand{\gb}{\mathfrak{b}}
\newcommand{\gc}{\mathfrak{c}}
\newcommand{\gd}{\mathfrak{d}}

\newcommand{\gh}{\mathfrak{h}}
\newcommand{\gi}{\mathfrak{i}}

\newcommand{\gm}{\mathfrak{m}}
\newcommand{\gn}{\mathfrak{n}}
\newcommand{\go}{\mathfrak{o}}
\newcommand{\gp}{\mathfrak{p}}
\newcommand{\gq}{\mathfrak{q}}

\newcommand{\0}{\emptyset}

\DeclareMathAlphabet{\mathpzc}{OT1}{pzc}{m}{it}

 \newcommand{\set}[1]{\left\{#1\right\}}

 \newcommand{\wt}[1]{\widetilde{#1}}
 \newcommand{\wh}[1]{\widehat{#1}}
 \newcommand{\usub}[2]{#1_{\textup{#2}}}
 
 \newcommand{\dlbr}{( \! (}
 \newcommand{\drbr}{) \! )}

 \newcommand{\lan}[3]{\mathcal{L}_{#1 \textup{#2} #3}}

\newcommand{\mdl}[1]{\mathcal{#1}}  
\newcommand{\bb}[1]{\mathbb{#1}}

\newcommand{\limplies}{\rightarrow}

\newcommand{\ex}[1]{\exists #1 \;} 

\newcommand{\rest}{\upharpoonright}

\newcommand{\fun}{\longrightarrow}
\newcommand{\efun}{\longmapsto}
\newcommand{\sub}{\subseteq}

\newcommand{\mi}{\smallsetminus}

\newcommand{\colim}[1]{\underset{#1}{\text{colim}}\,}  

\newcommand{\la}{\langle}
\newcommand{\ra}{\rangle}

\DeclareMathOperator{\rv}{rv}

\DeclareMathOperator{\vv}{val}

\DeclareMathOperator{\RES}{RES}

\DeclareMathOperator{\gsk}{\mathbf{K}_+}
\DeclareMathOperator{\ggk}{\mathbf{K}}
\DeclareMathOperator{\sgsk}{!\mathbf{K}_+}

\DeclareMathOperator{\ob}{Ob}
\DeclareMathOperator{\fn}{FN}
\DeclareMathOperator{\fib}{fib}

\DeclareMathOperator{\vol}{vol}
\DeclareMathOperator{\isp}{I_{sp}}

\DeclareMathOperator{\rad}{rad}

\DeclareMathOperator{\vrv}{vrv}

\DeclareMathOperator{\RVH}{RVH}

\DeclareMathOperator{\can}{\mathbf{c}}

\DeclareMathOperator{\pvf}{pvf}
\DeclareMathOperator{\prv}{prv}

\DeclareMathOperator{\ito}{int}
\DeclareMathOperator{\cl}{cl}

\DeclareMathOperator{\sgn}{sgn}
\DeclareMathOperator{\reg}{reg}

\DeclareMathOperator{\tor}{Tor}
\DeclareMathOperator{\res}{res}
\DeclareMathOperator{\aff}{Aff}

\newcommand{\ol}[1]{\overline{#1}}

\DeclareMathOperator{\TCVF}{TCVF}

\DeclareMathOperator{\db}{db}
\DeclareMathOperator{\ipm}{I_{pm}}
\DeclareMathOperator{\dsgsk}{!!\mathbf{K}_+}
\DeclareMathOperator{\dsggk}{!!\mathbf{K}}
\DeclareMathOperator{\sbe}{!\bb E}

\author[Yimu Yin]{Yimu Yin}
\address{Institut Math\'{e}matique de Jussieu \\ Universit\'e Pierre et Marie Curie \\ 4 place Jussieu \\ 75252 Paris Cedex 05 \\ France}
\email{yyin@math.jussieu.fr}
\title[Additive invariants in $\TCVF$]{Additive invariants in o-minimal valued fields}

\begin{document}

\begin{abstract}
We develop a theory of Hrushovski-Kazhdan style motivic integration for certain type of non-archimedean \omin-minimal fields, namely polynomial-bounded $T$-convex valued fields. The structure of valued fields is expressed through a two-sorted first-order language $\lan{}{TRV}{}$. We establish canonical homomorphisms between the Grothendieck semirings of various categories of definable sets that are associated with the $\VF$-sort and the $\RV$-sort of $\lan{}{TRV}{}$. The groupifications of some of these homomorphisms may be described explicitly and are understood as generalized Euler characteristics. In the end, following the Hrushovski-Loeser method, we construct topological zeta functions associated with (germs of) definable continuous functions in an arbitrary polynomial-bounded \omin-minimal field and show that they are rational. The overall construction is closely modeled on that of the original Hrushovski-Kazhdan construction, as reproduced in the series of papers by the present author.

\end{abstract}

\thanks{The research reported in this paper has been fully supported by the ERC Advanced Grant NMNAG}

\maketitle

\tableofcontents

\section{Introduction}

Towards the end of the introduction of \cite{hrushovski:kazhdan:integration:vf} three hopes for the future of the theory of motivic integration are mentioned. We investigate one of them in this paper: integration, or rather, since we will not consider general volume forms, additive invariants, in \omin-minimal valued fields. The prototype of such valued fields is $\R \dlbr t^{\Q} \drbr$, the power series field over $\R$ with exponents in $\Q$. One of the cornerstones of the methodology of \cite{hrushovski:kazhdan:integration:vf} is \cmin-minimality, which is the right analogue of \omin-minimality for algebraically closed valued fields and other closely related structures that epitomizes the behavior of definable subsets of the affine line. It, of course, fails in an \omin-minimal valued field, mainly due to the presence of a total ordering. The construction of additive invariants in this paper is thus carried out in a different framework, which affords a similar type of normal forms for definable subsets of the affine line, a special kind of weak \omin-minimality; this framework is van den Dries and Lewenberg's theory of $T$-convex valued fields \cite{DriesLew95, Dries:tcon:97}.

For a description of the ideas and the main results of the Hrushovski-Kazhdan style integration theory, we refer the reader to the original introduction in \cite{hrushovski:kazhdan:integration:vf} and also the introductions in \cite{Yin:int:acvf, Yin:int:expan:acvf}. There is also a quite comprehensive introduction to the same materials in \cite{hru:loe:lef} and, more importantly, a specialized version that relates the Hrushovski-Kazhdan style integration to the geometry and topology of Milnor fibers over the complex field. The method expounded there will be featured in this paper as well. In fact, since much of the work below is closely modeled on that in \cite{hrushovski:kazhdan:integration:vf, Yin:special:trans, Yin:int:acvf, hru:loe:lef}, the reader may simply substitute the term ``theory of polynomial-bounded $T$-convex valued fields'' for ``theory of algebraically closed valued fields'' or more generally ``$V$-minimal theories'' in those introductions and thereby acquire a quite good grip on what the results of this paper look like. For the reader's convenience, however, we shall repeat some of the key points, perhaps with minor changes here and there.

Let $(K, \vv : K \fun \Gamma)$ be a valued field, where $\vv$ is the valuation map, and $\OO$, $\MM$, $\K$ the corresponding valuation
ring, its maximal ideal, and the residue field. Let
\[
\RV(K) = K^{\times} / (1 + \MM)
\]
and $\rv : K^{\times} \fun \RV(K)$ be the quotient map. Note that, for each $a \in K$, $\vv$ is constant on the subset $a + a\MM$ and hence there is a naturally induced map $\vrv$ from $\RV(K)$ onto the value group $\Gamma$. The situation is illustrated in the following commutative diagram
\begin{equation*}
\bfig
 \square(0,0)/^{ (}->`->>`->>`^{ (}->/<600, 400>[\OO \mi \MM`K^{\times}`\K^{\times}`
\RV(K);`\text{quotient}`\rv`]
 \morphism(600,0)/->>/<600,0>[\RV(K)`\Gamma;\vrv]
 \morphism(600,400)/->>/<600,-400>[K^{\times}`\Gamma;\vv]
\efig
\end{equation*}
where the bottom sequence is exact. This structure may be expressed by a two-sorted first-order language $\lan{}{TRV}{}$, where $K$ is referred to as the $\VF$-sort and $\RV$ is taken as a new sort. On the other hand, for the main construction in this paper, $K$ could carry any extra structure that amounts to a polynomial-bounded \omin-minimal expansion $T$ of the theory of real closed fields (henceforth abbreviated as $\usub{\textup{RCF}}{}$); this is what the letter ``$T$'' stands for in $\lan{}{TRV}{}$. In fact, there is essentially no loss of generality if we take $K = \R \dlbr t^{\Q} \drbr$, which we shall do in the remainder of this introduction (see Example~\ref{exam:RtQ} for further details).

Let $\VF_*$ and $\RV[*]$ be two categories of definable sets that are respectively associated with the $\VF$-sort and the $\RV$-sort. In $\VF_*$, the objects are definable subsets of products of the form $\VF^n \times \RV^m$ and the morphisms are definable bijections. On the other hand, for technical reasons (particularly for keeping track of ambient dimensions), $\RV[*]$ is formulated in a somewhat complicated way and is hence equipped with a gradation by ambient dimension (see Definition~\ref{defn:c:RV:cat}). The main construction of the Hrushovski-Kazhdan theory is a canonical homomorphism from the Grothendieck semiring $\gsk \VF_*$ to the Grothendieck semiring $\gsk \RV[*]$ modulo a semiring congruence relation $\isp$ on the latter. In fact, it turns out to be an isomorphism. This construction has three main steps.
\begin{enumerate}[{Step} 1.]
 \item First we define a lifting map $\bb L$ from the set of objects in $\RV[*]$ into the set of objects in $\VF_*$; see Definition~\ref{def:L}. Next we single out a subclass of isomorphisms in $\VF_*$, which are called special bijections; see
Definition~\ref{defn:special:bijection}. Then we show that for any object $A$ in $\VF_*$ there is a special bijection
$T$ on $A$ and an object $\bm U$ in $\RV[*]$ such that $T(A)$ is isomorphic to $\bb L (\bm U)$. This implies that $\bb L$
hits every isomorphism class of $\VF_*$. Of course, for this result alone we do not have to limit our means to special
bijections. However, in Step~3 below, special bijections become an essential ingredient in computing the semiring congruence
relation $\isp$.

 \item For any two isomorphic objects $\bm U_1$, $\bm U_2$ in $\RV[*]$, their lifts $\bb L(\bm U_1), \bb L(\bm U_2)$ in
$\VF_*$ are isomorphic as well. This shows that $\bb L$ induces a semiring homomorphism from $\gsk \RV[*]$ into $\gsk \VF_*$, which is also denoted by $\bb L$.

 \item A number of classical properties of integration can already be (perhaps only partially) verified for the inversion of the homomorphism $\bb L$ and hence, morally, this third step is not necessary. For applications, however, it is much more satisfying to have a precise description of the semiring congruence relation induced by $\bb L$. The basic notion used in the description is that of a blowup of an object in $\RV[*]$, which is essentially a restatement of the trivial fact that there is an additive translation from $1 + \MM$ onto $\MM$. We then show that, for any objects $\bm U_1$, $\bm U_2$ in $\RV[*]$, there are isomorphic blowups $\bm U_1^{\sharp}$, $\bm U_2^{\sharp}$ of them if and only if $\bb L(\bm U_1)$, $\bb L(\bm U_2)$ are isomorphic. The ``if'' direction essentially contains a form of Fubini's Theorem and is the most technically involved part of the construction.
\end{enumerate}
The inverse of $\bb L$ thus obtained is called a Grothendieck homomorphism. If the Jacobian transformation preserves integrals, that is, the change of variables formula holds, then it may be called a motivic integration; we shall only consider a very primitive case of this notion in this paper. When the semirings are formally groupified, this Grothendieck homomorphism is recast as a ring homomorphism, which is denoted by $\int$.

The Grothendieck ring $\ggk \RV[*]$ may be expressed as a tensor product of two other Grothendieck rings $\ggk \RES[*]$ and $\ggk \Gamma[*]$, where $\RES[*]$ is essentially the category of definable sets over $\R$ (as a model of the theory $T$) and $\Gamma[*]$ is essentially the category of definable sets over $\Q$ (as an \omin-minimal group), and both are graded by ambient dimension. This results in various retractions from $\ggk \RV[*]$ into $\ggk \RES[*]$ or $\ggk \Gamma[*]$ and, when combined with the canonical homomorphism $\int$ (see Theorem~\ref{thm:ring}), yields various (generalized) Euler characteristics
\[
\Xint{\textup{G}}  :  \ggk \VF_* \to^{\sim} \Z^{(2)} \coloneqq \Z[X]/(X+X^2),
\]
which is actually an isomorphism, and
\[
\Xint{\textup{R}}^g, \Xint{\textup{R}}^b: \ggk \VF_* \fun \Z.
\]

For the construction of topological zeta functions we shall need to introduce the simplest volume form, namely the constant $\Gamma$-volume form $1$, into the various categories above. This modification has no bearing on the collection of objects in these categories, but does trim down the collection of morphisms. The resulting categories are denoted by $\vol \VF[*]$, $\vol \RV[*]$, etc. For example, given $(a, a')$ and $(b, b')$ in $\Q^2$, the subsets $\vv^{-1}(a, a')$ and $\vv^{-1}(b, b')$ of $\R \dlbr t^{\Q} \drbr^2$ are isomorphic in $\VF_*$ but are not isomorphic in $\vol \VF[*]$ unless $a + a' = b + b'$. Another change in $\vol \VF[*]$ is that morphisms may ignore a subset whose dimension is smaller than the ambient dimension. Thus $\vol \VF[*]$ is also graded by ambient dimension; this is why we have changed the position of ``$*$'' in the notation. The semiring congruence relation $\isp$ is now homogeneous and we have a canonical isomorphism of \emph{graded} Grothendieck rings
\[
\int : \ggk \vol \VF[*] \fun \ggk \vol \RV[*] /  \isp;
\]
see Theorem~\ref{main:k:dag}. Moreover, if we do restrict our attention to a special type of objects, namely those objects whose images under $\vv$ are (in effect) bounded from both sides, then there are two natural homomorphisms of graded rings
\[
\Xint{\textup{R}}^{\pm}: \ggk \vol \VF^{\diamond}[\ast] \fun \Z[X];
\]
see Theorem~\ref{thm:poin}.

Now let $f : \R^n \fun \R$ be a definable non-constant continuous function sending $0$ to $0$, or more generally a germ at $0$ of such functions. For example, $f$ could be a polynomial function or a subanalytic function if it is allowed by the theory $T$. Unlike in the complex case, there is no Milnor fibration of $f$ (and hence there is no consensus on what the monodromy of $f$ should be). Nevertheless we may still define the positive and the negative \emph{Milnor fibers} of $f$ at $0$:
\[
M_+ = B(0, \epsilon) \cap f^{-1}(\delta) \quad \text{and} \quad M_- = B(0, \epsilon) \cap f^{-1}(- \delta),
\]
where $0 < \delta \ll \epsilon \ll 1$ and $B(0, \epsilon)$ is the ball in $\R^n$ centered at $0$ with radius $\epsilon$. By \omin-minimal trivialization (see \cite[\S 9]{dries:1998}), the (embedded) definable homeomorphism types of $M_+$ and $M_-$ are well-defined (of course $M_+$ and $M_-$ are not necessarily homeomorphic, definably or not; indeed $M_-$ may be empty while $M_+$ is not). By $T$-convexity, $f$ may be lifted in a unique way to a definable continuous function $f^{\uparrow} : \OO^n \fun \OO$. The Milnor fibers of $f$ with (thickened) formal arcs attached to each point are defined as
\begin{align*}
\wt M_+  &= \{ a \in \MM^n : \rv (f^{\uparrow}(a)) = \rv(t)\}, \\
\wt M_-  &= \{ a \in \MM^n : \rv (f^{\uparrow}(a)) = -\rv(t)\}.
\end{align*}
Following \cite{denefloeser:arc, denef:loeser:1992:caract}, we attach topological (or motivic) zeta functions $Z^{\pm}(\wt M_{\pm})(Y)$ to $f$ (two to each one of $\wt M_+$ and $\wt M_-$, due to the lack of a canonical identification of certain graded ring with $\Z[X]$), which are power series in $\Z[X] \llbracket Y \rrbracket$ whose coefficients are integrands of truncated (thickened) formal arcs. In Theorem~\ref{zeta:rat}, we show that these zeta functions are rational and their denominators are products of terms of the form $1 - (-X)^{a} Y^b$, where $b \geq 1$. Consequently, $Z^{\pm}(\wt M_{\pm})(Y)$ attain limits $e^{\pm}(\wt M_{\pm}) \in \Z$ as $Y \limplies \infty$ and we have the equality
\[
e^{\pm}(\wt M_{\pm}) X^n = - \Xint{\textup{R}}^{\pm} [\wt M_{\pm}].
\]

For certain purposes, the difference between model theory and algebraic geometry is somewhat easier to bridge if one works over the complex field, as is demonstrated in \cite{hru:loe:lef}; however, over the real field, although they do overlap significantly, the two worlds seem to diverge in their methods and ideas. Our results should be understood in the context of ``\omin-minimal geometry'' \cite{dries:1998, DrMi96}. This is reflected in our preference of the terminology ``topological zeta function'' since, in the literature of real algebraic geometry, ``motivic zeta function'' has already been constructed (see, for example, \cite{Comte:fichou}), which is a much finer invariant as there are much less morphisms in the background. In general, the various Grothendieck rings considered in real algebraic geometry bring about lesser collapse of ``algebraic data'' and hence are more faithful in this regard, although the flip side of the story is that they are computationally intractable (especially when resolution of singularities is involved) and specializations are often needed in practice. For instance, the Grothendieck ring of real algebraic varieties may be specialized to $\Z[X]$, which is called the virtual Poincar\'e polynomial (see \cite{mccrory:paru:virtual:poin}). Still, our method does not seem to be suited for recovering invariants at this level, at least not directly (that the homomorphism $\Xint{\textup{R}}^{\pm}$ has $\Z[X]$ as its codomain is merely a coincidence and is not an essential feature of the construction).

Similar constructions are available for other (closely related) categories of definable sets, in particular, for such categories with general volume forms, which we have not included in this paper for the sake of simplicity and brevity. For those constructions, one needs to add a section from the $\RV$-sort into the $\VF$-sort or at least a standard part map, that is, a section from the residue field into the $\VF$-sort, since it is not conceptually correct to use the ``counting measure'' on the residue field anymore. We shall elaborate on this in a sequel.

The role of $T$-convexity in this paper cannot be overemphasized. However, it does not quite work if the exponential function is included in the theory $T$. It remains a worthy challenge to find a suitable framework in which the construction of this paper may be extended to that case. Among many other things, such a construction will enable us to define motivic Mellin transforms, the lack of which is a major obstacle to further development and application of the theory of motivic integration.

\section{Preliminaries}

A large part of the notations, terminologies, and conventions in \cite{DriesLew95, Dries:tcon:97, Yin:special:trans,Yin:int:acvf, Yin:int:expan:acvf} are directly applicable in the current setting, more or less verbatim; some of these will be recalled as we go along.

Throughout this paper, let $T$ be a complete polynomial-bounded \omin-minimal $\lan{T}{}{}$-theory extending $\usub{\textup{RCF}}{}$. We assume that $T$ admits quantifier elimination and is universally axiomatizable in $\lan{T}{}{}$. Without loss of generality, we may also assume $T = T^{\textup{df}}$ (see \cite[\S2.4]{DriesLew95}). Thus every substructure of a model of $T$ is actually a model of $T$ and, as such, is an elementary substructure.

\begin{defn}
The language $\lan{}{TRV}{}$ has the following sorts and symbols:
\begin{itemize}
 \item A sort $\VF$, in which we use the language $\lan{T}{}{}$, together with a new constant symbol $\imath$.
 \item A sort $\RV$, whose basic language is that of groups, written multiplicatively as $\{1, \times \}$, together with a constant symbol $\infty$. We shall denote $\RV \mi \{\infty\}$ by $\RV^{\times}$.
 \item A unary predicate $\K^{\times}$ in the $\RV$-sort. The union $\K^{\times} \cup \{\infty\}$ is denoted by $\K$, which is more conveniently thought of as a sort and, as such, employs the language $\lan{T}{}{}$ as well, where the constant symbol $1$ is shared with the $\RV$-sort and $\infty$ serves as $0$.
 \item A binary relation symbol $\leq$ in the $\RV$-sort.
 \item A function symbol $\rv : \VF \fun \RV$.
\end{itemize}
\end{defn}

\begin{defn}\label{defn:tcf}
The theory $\usub{\textup{TCVF}}{}$ of \emph{$T$-convex valued fields} in the language $\lan{}{TRV}{}$ states the following:
\begin{enumerate}[({Ax.} 1.)]
 \item The $\lan{T}{}{}$-reduct of the $\VF$-sort is a model of $T$. Let $\VF^+ \sub \VF$ be the subset of positive elements and $\VF^- \sub \VF$ the subset of negative elements. Let
     \[
     |\cdot| : \VF^{\times} \fun \VF^+
      \]
      be the canonical group homomorphism, that is, the absolute value map. The element $\imath$ is infinitesimal.

 \item $(\RV^{\times}, 1, \times)$ is a divisible abelian group, where multiplication $\times$ is augmented by $t \times \infty = \infty$ for all $t \in \RV^{\times}$, and $\rv : \VF^{\times} \fun \RV^{\times}$ is a surjective group homomorphism augmented by $\rv(0) = \infty$.

  \item The relation $\leq$ is a total ordering on $\RV$ such that, for all $t, t' \in \RV$, $t < t'$ if and only if $\rv^{-1}(t) < \rv^{-1}(t')$. Let $\RV^+ \sub \RV$ be the subset of positive elements and $\RV^- \sub \RV$ the subset of negative elements. Obviously $\rv(\VF^+) = \RV^+$ and $\rv(\VF^-) = \RV^-$. The absolute value map on $\RV$ induced by that on $\VF$ is still denoted by $|\cdot|$. Note that, depending on the context, the element $\infty \in \RV$, now more aptly referred to as the \emph{middle element} of $\RV$, is sometimes more suggestively denoted by $0$. The subsets $\RV^+ \cup \{\infty\}$ and $\RV^- \cup \{\infty\}$ are denoted by $\RV^+_{\infty}$ and $\RV^-_{\infty}$, respectively.

   \item $\K^{\times}$ is a subgroup of $\RV^{\times}$ that does not contain $\rv(\imath)$ and $\K^{+} \coloneqq \K^{\times} \cap \RV^+$ is a convex subgroup of $\RV^+$. The quotient groups $\RV^{\times} / \K^+$, $\RV^{+} / \K^+$ are denoted by $\Gamma$, $\Gamma^+$, their quotient maps by $\vrv$, $\vrv^+$, and the sets $\Gamma \cup \{\infty\}$, $\Gamma^+ \cup \{\infty\}$ by $\Gamma_{\infty}$, $\Gamma^+_{\infty}$, respectively. Note that group actions in $\Gamma$, $\Gamma^+$ will be written \emph{multiplicatively} as $*$ (or omitted if there is no danger of confusion). With the ordering $\leq$ induced by $\vrv^+$, $\Gamma^+$ is actually an ordered abelian group. Let $\leq^{-1}$ be the ordering on $\Gamma^+_{\infty}$ inverse to $\leq$, with $\infty$ the top element, and
       \[
       |\Gamma_{\infty}| \coloneqq (\Gamma^+_{\infty}, +, \leq^{-1})
        \]
        the resulting \emph{additively} written ordered abelian group. The composition
      \[
      |\vv| : \VF \to^{\rv} \RV \to^{|\cdot|} \RV^+_{\infty} \to^{\vrv^+}  |\Gamma_{\infty}|
       \]
       is a (nontrivial) valuation with valuation ring $\OO = \rv^{-1}(\RV^{\circ})$ and maximal ideal $\MM = \rv^{-1}(\RV^{\circ\circ})$, where
     \begin{align*}
     \RV^{\circ} &= \{t \in \RV: 1 \geq \vrv^+(|t|) \} \\
     \RV^{\circ \circ} &= \{t \in \RV: 1 > \vrv^+(|t|)\}.
     \end{align*}
     It follows that $\imath \in \MM$. The multiplicative group $\OO^{\times} \coloneqq \OO \mi \MM$ of units of $\OO$ is sometimes denoted by $\UU$.

 \item The $\K$-sort is a model of $T$ and, as a field, is the residue field of the valued field $(\VF, \OO)$. The canonical quotient map $\OO \fun \K$ is denoted by $\res$. For notational convenience, we extend the domain of $\res$ to $\VF$ by setting $\res(a) = 0$ for all $a \in \VF \mi \OO$. The following function is also denoted by $\res$:
     \[
     \RV \to^{\rv^{-1}} \VF \to^{\res} \K. \label{ax:t:model}
     \]

 \item ($T$-convexity). Let $f : \VF \fun \VF$ be a continuous function defined by an $\lan{T}{}{}$-formula. Then $f(\OO) \sub \OO$. \label{ax:tcon}

\item Suppose that $\phi$ is an $\lan{T}{}{}$-formula that defines a continuous function $f : \VF^m \fun \VF$. Then $\phi$ also defines a continuous function $\ol f : \K^m \fun \K$. For all $a \in \OO^m$, we have $\res(f(a)) = \ol f(\res(a))$.  \label{ax:match}
\end{enumerate}
\end{defn}

Although the behavior of the valuation map $|\vv|$ in the traditional sense is coded in $\TCVF$, it is technically more correct to work with the more natural \emph{signed} valuation map
\[
\vv : \VF \to^{\rv} \RV \to^{\vrv} \Gamma_{\infty} \coloneqq (\Gamma_{\infty}, *, \leq),
\]
where the ordering $\leq$ no longer needs to be inverted. The axioms above guarantee that the quotient group $\K^{\times} / \K^+$ has exactly two elements, which is simply written as $\pm 1$. Thus $|\vv|$ may be understood as the quotient $\vv / \pm 1$, that is, the absolutization of $\vv$; similarly for $|\vrv|$. Note that we shall slightly abuse the notation and denote the ordering in $|\Gamma_{\infty}|$ simply by $\leq$; this should not cause confusion since the ordering in $\Gamma_{\infty}$ will rarely be used.

Let $T_{\textup{convex}}$ be the $\lan{}{convex}{}$-theory of pairs $(\mdl R, \OO)$ with $\mdl R \models T$ and $\OO$ a \emph{proper} $T$-convex subring, as described in \cite{DriesLew95}. We also add the new constant symbol $\imath$ to $\lan{}{convex}{}$ and the axiom ``$\imath$ is in the maximal ideal'' to $T_{\textup{convex}}$ so that $T_{\textup{convex}}$ may be formulated as a universal theory. Thus every substructure of a model of $T_{\textup{convex}}$ is a model of $T_{\textup{convex}}$. We shall view $T_{\textup{convex}}$ as the $\lan{}{convex}{}$-reduct of $\TCVF$.

\begin{thm}\label{tcon:qe}
The theory $T_{\textup{convex}}$ admits quantifier elimination and is complete; moreover, it has definable Skolem functions given by $\lan{T}{}{}(\imath)$-terms.
\end{thm}
\begin{proof}
The first assertion is contained in \cite[Theorem~3.10, Corollary~3.13]{DriesLew95}. The second assertion is an easy consequence of our assumption on $T$, quantifier elimination in $T_{\textup{convex}}$, and universality of $T_{\textup{convex}}$, as in \cite[Corollary~2.15]{DMM94}.
\end{proof}

\begin{exam}\label{exam:RtQ}
Let $T = \usub{\textup{RCF}}{an}$ be the \omin-minimal expansion of $\usub{\textup{RCF}}{}$ with restricted analytic functions and $K = \R \dlbr t^{\Q} \drbr$ the power series field (or the Hahn field) over $\R$ with exponents in $\Q$ (see \cite[\S1.2]{DMM94}). According to \cite[\S2.8]{DMM94}, $K$ models $T$ in a natural way. Let $\OO(K) = \R \llbracket t^{\Q} \rrbracket$, the subring of $K$ that consists of those power series whose leading exponents are nonnegative, and $\MM(K) \sub \OO(K)$, the $\OO(K)$-module that consists of those power series whose leading exponents are positive. It is well-known that $(K, \OO(K))$ is a henselian valued field whose residue field and value group are manifestly isomorphic to $\R$ and $\Q$, respectively. Let
\[
\RV(K) = K^{\times} / (1 + \MM(K)),
\]
which is canonically isomorphic to $\R \oplus \Q$, and $\rv: K^{\times} \fun \R \oplus \Q$ be the leading term map augmented by $\rv(0) = (0, \infty)$, that is, $\rv(x) = (a_{\xi}, \xi)$ if and only if $x$ is of the form $a_{\xi} t^{\xi} + \ldots$ and $a_{\xi} \neq 0$. Let $\R^{+}$ be the multiplicative group of positive reals and $\RV^{+}(K) = \R^+ \oplus \Q$. The quotient
\[
\Gamma(K) \coloneqq (\R \oplus \Q) / \R^{+}
\]
is canonically isomorphic to the subgroup
\[
\pm e^{\Q} \coloneqq e^{\Q} \cup - e^{\Q}
\]
of $\R^{\times}$. Now it is routine to interpret $K$ as an $\lan{}{TRV}{}$-structure. It is also a model of $\TCVF$; all the axioms are more or less derived from the valued field structure, except (Ax.~\ref{ax:tcon}), which holds since $\usub{\textup{RCF}}{an}$ is polynomial-bounded, and (Ax.~\ref{ax:match}), which follows from \cite[Proposition~2.20]{DriesLew95}.
\end{exam}

\begin{rem}
By our assumption on $T$ and \cite[Proposition~2.20]{DriesLew95}, the $T$-model that the residue field carries, as dictated by (Ax.~\ref{ax:t:model}) of Definition~\ref{defn:tcf}, is precisely the canonical one described in \cite[Remark~2.16]{DriesLew95}. This is clear if one works with a concrete \omin-minimal theory such as $\usub{\textup{RCF}}{}$ or $\usub{\textup{RCF}}{an}$. The general case follows from the fact that we may translate $T$ into a different language in which \emph{all} primitives, except the ordering $\leq$, define a \emph{total continuous} function in all models of $T$.

Thus every $T_{\textup{convex}}$-model interprets, or rather expands to, a unique $\TCVF$-model. In fact, every embedding between two $T_{\textup{convex}}$-models expands uniquely to an embedding between two $\TCVF$-models. It follows that $\TCVF$ is complete. From now on, we shall work in a sufficiently saturated model $\gC \models \TCVF$ and write $\VF(\gC)$, $\RV(\gC)$, etc., simply as $\VF$, $\RV$, etc. A subset in $\gC$ may simply be referred to as a subset.

It is routine to check that, except surjectivity of the homomorphism $\rv$, $\TCVF$ is also universally axiomatized. Let $\mdl S \sub \gC$ be a substructure. Thus $\mdl S$ is indeed a model of $\TCVF$ if it is $\VF$-generated, that is, if $\RV(\mdl S) = \rv(\VF(\mdl S))$. At any rate, $\VF(\mdl S)$, $\res(\VF(\mdl S))$, and $\K(\mdl S)$ are all models of $T$. Due to the presence of the constant $\imath$, $\vv(\VF(\mdl S))$ is never trivial.

For any $A \sub \VF \cup \RV$, the substructure generated by $A$ over $\mdl S$ is denoted by $\la \mdl S , A \ra$ or $\mdl S \la A \ra$. If $A \sub \VF$ then the $T$-model generated by $A$ over $\VF(\mdl S)$ is denoted by $\la \mdl S , A \ra_T$ or $\mdl S \la A \ra_T$, and it is easy to see that $\VF(\la \mdl S , A \ra) = \la \mdl S , A \ra_T$.
\end{rem}

The default topologies of $\VF$, $\RV$, $\Gamma$, $\K$, etc., are of course the ones induced by the total orderings. The topological operators for interior, closure, and boundary are denoted by $\ito$, $\cl$, and $\partial$, respectively. A product of open intervals $(x_i, y_i)$ (or an \emph{open box} with \emph{sides} $(x_i, y_i)$) is sometimes denoted by $(x, y)$, where $x = (x_1, \ldots, x_n)$ and $y = (y_1, \ldots, y_n)$; similarly for other cases.

\begin{thm}\label{theos:qe}
The theory $\TCVF$ admits quantifier elimination.
\end{thm}
\begin{proof}
We shall run the usual Shoenfield test for quantifier elimination. Thus let $\mdl M$ be a model of $\TCVF$, $\mdl S$ a substructure of $\mdl M$, and $\sigma : \mdl S \fun \mdl \gC$ a monomorphism. All we need to do is to extend $\sigma$ to an embedding $\wh \sigma : \mdl M \fun \gC$. The construction is more or less a variation of that in the proof of \cite[Theorem~3.10]{Yin:QE:ACVF:min}. The strategy is to reduce the situation to Theorem~\ref{tcon:qe}. In the process of doing so, instead of the dimension inequality of the general theory of valued fields, the Wilkie inequality \cite[Corollary~5.6]{Dries:tcon:97} is used (see \cite[\S3.2]{DriesLew95} for the notion of ranks of $T$-models).

Let $\mdl S_* = \la \VF(\mdl S) \ra$ and $t \in \RV(\mdl S) \mi \RV(\mdl S_*)$. Note that if such a $t$ does not exist then $\mdl S = \mdl S_*$ is a model of $\TCVF$ and, since its $\lan{}{convex}{}$-reduct is also a model of $T_{\textup{convex}}$, an embedding as desired can be easily obtained by applying Theorem~\ref{tcon:qe}. Let $a \in \VF(\mdl M)$ with $\rv(a) = t$ and $b \in \VF$ with $\rv(b) = \sigma(t)$. Observe that, according to $\sigma$, $a$ and $b$ must make the same cut in $\VF(\mdl S)$ and $\VF(\sigma(\mdl S))$, respectively, and hence there is an $\lan{T}{}{}$-isomorphism
\[
\bar \sigma : \la \mdl S_*, a \ra_T \fun \la \sigma(\mdl S_*), b \ra_T
\]
with $\bar \sigma(a) = b$ and $\bar \sigma \rest \VF(\mdl S) = \sigma \rest \VF(\mdl S)$. We shall show that $\bar \sigma$ expands to an isomorphism between $\la \mdl S_*, a \ra$ and $\la \sigma(\mdl S_*), b \ra$ that is compatible with $\sigma$.

Case (1): There is an $a_1 \in \la \mdl S_*, a \ra_T$ such that
\[
| \OO(\mdl S_*) | < a_1 < | \VF(\mdl S_*) \mi \OO(\mdl S_*) |.
\]
Set $\Gamma(\mdl S_*) = G$. Since $\OO(\la\mdl S_*, a \ra)$ is $T$-convex, by \cite[Lemma~5.4]{Dries:tcon:97} and \cite[Remark~3.8]{DriesLew95},
\begin{itemize}
  \item either $a_1 \in \OO(\la\mdl S_*, a \ra)$ and $\Gamma(\la \mdl S_*, a \ra) = G$ or
  \item $a_1 \notin \OO(\la\mdl S_*, a \ra)$ and $\Gamma(\la \mdl S_*, a \ra) = G \oplus \Q$.
\end{itemize}
By the Wilkie inequality, if $\Gamma(\la \mdl S_*, a \ra) = G \oplus \Q$ then $\K(\la \mdl S_*, a \ra) = \K(\mdl S_*)$ and hence $\vrv(t) \notin G$, which implies $\vrv(\sigma(t)) \notin \sigma(G)$; conversely, if
\[
\Gamma(\la \sigma(\mdl S_*), b \ra) = \sigma(G) \oplus \Q
\]
then $\vrv(t) \notin G$. Therefore
\[
\Gamma(\la \mdl S_*, a \ra) = G \oplus \Q \quad \text{if and only if} \quad \Gamma(\la \sigma(\mdl S_*), b \ra) = \sigma(G) \oplus \Q,
\]
which, by \cite[Remark~3.8]{DriesLew95}, is equivalent to saying that $a_1 \in \OO(\la\mdl S_*, a \ra)$ if and only if $\bar \sigma(a_1) \in \OO(\la \mdl \sigma(\mdl S_*), b \ra)$.

Subcase (1a): $a_1 \in \OO(\la \mdl S_*, a \ra)$. Subcase~(1a) of the proof of \cite[Theorem~3.10]{DriesLew95} shows that $\bar \sigma$ expands to an $\lan{}{convex}{}$-isomorphism and hence to an $\lan{}{TRV}{}$-isomorphism, which is also denoted by $\bar \sigma$. Since $\Gamma(\la \mdl S_*, a \ra) = G$, we may assume $a \in \K(\mdl M)$. By the Wilkie inequality, $\K(\la \mdl S_*, a \ra)$ is precisely the $T$-model generated by $t$ over $\K(\mdl S_*)$ and hence
\[
\RV(\la \mdl S_*, a \ra) = \la \RV(\mdl S_*), t \ra.
\]
It follows that
\[
\bar \sigma \rest \RV(\la \mdl S_*, a \ra) = \sigma \rest \RV(\la \mdl S_*, a \ra).
\]

Subcase (1b): $a_1 \notin \OO(\la \mdl S_*, a \ra)$. As above, Subcase~(1b) of the proof of \cite[Theorem~3.10]{DriesLew95} shows that $\bar \sigma$ expands to an $\lan{}{TRV}{}$-isomorphism and this time $\K(\la \mdl S_*, a \ra) = \K(\mdl S_*)$. It is clear that
\[
\bar \sigma \rest \RV(\la \mdl S_*, a \ra) = \sigma \rest \RV(\la \mdl S_*, a \ra).
\]

Case (2): Case (1) fails. Then there is also no $b_1 \in \la \sigma(\mdl S_*), b \ra_T$ such that
\[
| \OO(\sigma(\mdl S_*)) | < b_1 < | \VF(\sigma(\mdl S_*)) \mi \OO(\sigma(\mdl S_*)) |.
\]
Using Case~(2) of the proof of \cite[Theorem~3.10]{DriesLew95}, compatibility between $\bar \sigma$ and $\sigma$ may be deduced as in Case (1) above.
\end{proof}

Therefore, for all $A \sub \VF$, $\la A \ra$ is an elementary substructure of $\gC$.

\begin{rem}
We may assume that, in any $\lan{}{TRV}{}$-formula, all the $\lan{T}{}{}$-terms occur in the form $\rv(F(X))$. For example, if $F(X)$ and $G(X)$ are $\lan{T}{}{}$-terms then the formula $F(X) < G(X)$ is equivalent to $\rv(F(X) - G(X)) < 0$.
\end{rem}

\begin{cor}
Every parametrically $\lan{}{TRV}{}$-definable subset of $\VF^n$ is parametrically $\lan{}{convex}{}$-definable.
\end{cor}

This corollary enables us to transfer results in the theory of $T$-convex valued fields \cite{DriesLew95, Dries:tcon:97} into our setting, which we shall do without further explanation.

\begin{nota}\label{indexing}
Coordinate projections, their inverses, and other related operations are ubiquitous in this paper. It is more efficient to fix a shorthand for them as soon as possible. Unless otherwise specified, by writing $a \in A$ we shall mean that $a$ is a finite tuple of elements (or ``points'') of $A$, whose length, denoted by $\lh(a)$, is not always indicated. If $a = (a_1, \ldots, a_n)$ then for all $1 \leq i \leq n$ the tuple
\[
(a_1, \ldots, a_{i-1}, a_{i+1}, \ldots, a_n)
\]
is sometimes denoted by $\wh a_i$. We say that $B$ is a subset \emph{in} $A$ if $B \sub A^n$ for some $n$. If the coordinates of $A$ range in the sorts $\VF$, $\RV$, $\Gamma$, etc., then it makes sense to speak of the positive and the negative parts of $A$, which are denoted by $A^+$ and $A^-$, respectively. The absolute value map $A \fun A^+$ is always denoted by $| \cdot |$.

For each $n \in \N$, let $[n] = \set{1, \ldots, n}$. Let $A$ be a subset of $\VF^n \times \RV^m$. As a general rule, the coordinates of $A$ are indexed by $[n+m] = [n] \uplus [m]$. Let $E \sub [n+m]$ and $\wt E = [n+m] \mi E$. If $E$ is a singleton $\{i\}$ then we always write $E$ as $i$ and $\wt E$ as $\wt i$. We write $\pr_E(A)$ for the projection of $A$ to the coordinates in $E$. For $a \in \pr_{\wt E} (A)$, the fiber $\{b : ( b, a) \in A \}$ is denoted by $\fib(A, a)$ or, if there is no danger of confusion, by $A_a$. Note that the distinction between the two subsets $A_a$ and $A_a \times \{ a \}$ is often immaterial and hence they will be tacitly identified. Also, it is more convenient to use simple descriptions
as subscripts. For example, if $E = \{1, \ldots, k\}$, etc., then we may write $\pr_{\leq k}$, etc. If $E$ contains exactly the
$\VF$-indices (respectively $\RV$-indices) then $\pr_E$ is written as $\pvf$ (respectively $\prv$). If $\pr_{E}(A) \sub B$ and $B$ has been clearly understood in the context then it is more informative to write $\pr_B \rest A$.

Given a function $f : A \fun B$, we shall often write $A_b$ for the fiber over $b \in B$ under $f$. In particular, given a subset $A$, we may write $A_{x}$ for the fiber over $x$ under a function of the forms $\rv \rest A$, $\vv \rest A$, $\vrv \rest A$, etc. Of course which function is being considered should always be clear in context.
\end{nota}

\begin{conv}
We shall work with a fixed small substructure $\mdl S$. Note that $\mdl S$ is regarded as a part of the theory now and hence, contrary to the usual convention in the model-theoretic literature, ``$\0$-definable'' or ``definable'' only means ``$\mdl S$-definable'' instead of ``parametrically definable'' if no other qualifications are given. To simplify the notation, we shall not mention $\mdl S$ and its extensions in context if no confusion can arise. For example, the definable closure operator $\dcl_{\mdl S}$, etc., will simply be written as $\dcl$, etc.

Semantically we shall treat the value group $\Gamma$ as an imaginary sort and write $\RV_{\Gamma}$ for $\RV \cup \, \Gamma$. However, syntactically any reference to $\Gamma$ may be eliminated in the usual way and we can still work with $\lan{}{TRV}{}$-formulas for the same purpose.
\end{conv}

\begin{defn}
Let $\mdl M$, $\mdl N$ be substructures and $\sigma : \mdl M \fun \mdl N$ an $\lan{}{TRV}{}$-isomorphism. We say that $\sigma$ is an \emph{immediate isomorphism} if $\sigma(t) = t$ for all $t \in \RV(\mdl M)$.
\end{defn}

\begin{lem}\label{imm:iso}
Let $\mdl M$, $\mdl N$ be $\VF$-generated substructures and $\sigma : \mdl M \fun \mdl N$ an immediate isomorphism. Let $a, b \in \VF \mi (\VF(\mdl M) \cup \VF(\mdl N))$ such that $\rv(a - c) = \rv(b -\sigma(c))$ for all $c \in \VF(\mdl M)$. Then $\sigma$ may be extended to an immediate isomorphism $\bar \sigma : \la \mdl M, a \ra \fun \la \mdl N, b \ra$ with $\bar \sigma(a) = b$.
\end{lem}
\begin{proof}
This is completely analogous to the proof of Theorem~\ref{theos:qe} and is left to the reader.
\end{proof}

\begin{lem}\label{imm:ext}
Every immediate isomorphism $\sigma : \mdl M \fun \mdl N$ may be extended to an immediate automorphism of $\gC$.
\end{lem}
\begin{proof}
Let $\mdl M_* = \la \VF(\mdl M) \ra$ and $\mdl N_* = \la \VF(\mdl N) \ra$. Let $t \notin \RV(\mdl M_*)$ and $a \in \rv^{-1}(t)$. Clearly $\rv(a - c) = \rv(a -\sigma(c))$ for all $c \in \VF(\mdl M_*)$. By Lemma~\ref{imm:iso}, $\sigma$ may be extended to an immediate isomorphism $\la \mdl M, a \ra \fun \la \mdl N, a \ra$. Iterating this procedure, the assertion follows.
\end{proof}

Since $\TCVF$ is a weakly \omin-minimal theory (see \cite[Corollary~3.14]{DriesLew95}), we can use the dimension theory of \cite[\S4]{mac:mar:ste:weako} in $\gC$.

\begin{defn}
The \emph{$\VF$-dimension} of a definable subset $A$, which is denoted by $\dim_{\VF}(A)$, is the largest natural number $k$ such that, possibly after re-indexing of the $\VF$-coordinates, $\pr_{\leq k}(A_t)$ has non-empty interior for some $t \in \prv(A)$.
\end{defn}

By the Wilkie inequality, the exchange principle holds in $\VF$. Therefore, by \cite[\S4.12]{mac:mar:ste:weako}, we may equivalently define $\dim_{\VF}(A)$ to be the maximum of the algebraic dimensions of the fibers $A_t$. Yet another way to define this notion of $\VF$-dimension is to imitate \cite[Definiton~4.1]{Yin:special:trans}, since we have:

\begin{lem}
If $\dim_{\VF}(A) = k$ then $k$ is the smallest number such that there is a definable injection $f: A \fun \VF^k \times \RV^l$.
\end{lem}
\begin{proof}
This is immediate by a routine argument combining the exchange principle, Lemma~\ref{RV:no:point} below, and compactness.

Alternatively, we may quote \cite[Theorem~4.11]{mac:mar:ste:weako}.
\end{proof}

The proof of \cite[Lemma~4.2]{Yin:special:trans} requires the crucial though straightforward \cite[Lemma~4.10]{Yin:QE:ACVF:min}. The latter is trivialized in the current setting since there is a total ordering.

A property holds \emph{almost everywhere} on $A$ or \emph{for almost every element} in $A$ if it holds away from a definable subset of $A$ of a smaller $\VF$-dimension. This terminology will also be used when other notions of dimension are involved.

Recall \cite[Theorem~A]{Dries:tcon:97}: The structure of the definable subsets in the $\K$-sort is precisely that given by the theory $T$.

Recall \cite[Theorem~B]{Dries:tcon:97}: The structure of the definable subsets in the (imaginary) $\Gamma$-sort is precisely that given by the theory of nontrivially ordered vector spaces over $\Q$. There are two ways of treating an element $\gamma \in \Gamma$: as a point (when we study $\Gamma$ as an independent structure) or a subset of $\gC$ (when we need to remain in the realm of definable subsets of $\gC$). The former perspective simplifies the notation but is of course dispensable. We shall write $\vrv^{-1}(\gamma)$ when we want to emphasize that $\gamma \in \Gamma$ is a subset of $\gC$.

In a nutshell, both the $\K$-sort and the $\Gamma$-sort are stably embedded.

\begin{rem}\label{rem:RV:weako}
It is easy to check that the axioms that only concern the $\RV$-sort also form a weakly \omin-minimal theory and the exchange principle holds with respect to this theory. Therefore we can use the dimension theory of \cite[\S4]{mac:mar:ste:weako} directly in the $\RV$-sort as well. We call it the $\RV$-dimension and the corresponding operator is denoted by $\dim_{\RV}$. Note that $\dim_{\RV}$ does not depend on parameters (see \cite[\S4.12]{mac:mar:ste:weako}) and agrees with the \omin-minimal dimension in the $\K$-sort (see \cite[\S4.1]{dries:1998}) whenever both are applicable.

Similarly we shall use \omin-minimal dimension in the $\Gamma$-sort and call it the $\Gamma$-dimension. The corresponding operator is denoted by $\dim_{\Gamma}$.

Very often we may and shall apply the theory of \omin-minimality, in particular its terminologies and notions, to a subset $U \sub \RV^n$ such that $\vrv(U)$ is a singleton or, more generally, is finite. For example, we shall say that $U$ is a \emph{cell} if the translation $U / u \sub (\K^+)^n$ of $U$ with respect to some $u \in U$ is an \omin-minimal cell (see \cite[\S3]{dries:1998}); this definition clearly does not depend on $t$. Similarly, the \emph{Euler characteristic} $\chi(U)$ of $U$ is the Euler characteristic of $U / u$ (see \cite[\S4.2]{dries:1998}). We may extend this definition to the disjoint union of any finite number of (not necessarily disjoint) subsets $U_i \sub \RV^n \times \Gamma^m$ such that each $\vrv(U_i)$ is finite.
\end{rem}

\begin{lem}
If $\dim_{\RV}(U) = k$ then $k$ is the smallest number such that there is a definable injection $f: U \fun \RV^k$.
\end{lem}
\begin{proof}
This is immediate by \cite[Theorem~4.11]{mac:mar:ste:weako}.
\end{proof}

\begin{defn}\label{defn:disc}
A subset $\gb$ of $\VF$ is an \emph{open disc} if there is a $\gamma \in |\Gamma|$ and a $b \in \gb$ such that $a \in \gb$ if and
only if $|\vv|(a - b) > \gamma$; it is a \emph{closed disc} if $a \in \gb$ if and only if $|\vv|(a - b) \geq \gamma$; it is an
\emph{$\RV$-disc} if $\gb = \rv^{-1}(t)$ for some $t \in \RV$. The value $\gamma$ is the \emph{valuative radius} of $\gb$, which is denoted by $\rad (\gb)$. Each point in $\VF$ is a closed disc of valuative radius $\infty$ and $\VF$ is a clopen disc of radius $- \infty$. If $|\vv|$ (or $\vv$) is constant on $\gb$ --- that is, $\gb$ is contained in an $\RV$-disc --- then $|\vv|(\gb)$ is the \emph{valuative center} of $\gb$; if $|\vv|$ is not constant on $\gb$, that is, $0 \in \gb$, then the \emph{valuative center} of
$\gb$ is $\infty$.

A closed disc with a maximal open subdisc removed is called a \emph{thin annulus}.

A subset $\gp \sub \VF^n \times \RV^m$ is an (\emph{open, closed, $\RV$-}) \emph{polydisc} if it is of the form $(\prod_{i \leq n} \gb_i) \times \{t \}$, where each $\gb_i$ is an (open, closed, $\RV$-) disc. The \emph{radii} of $\gp$, denoted by $\rad(\gp)$, means the tuple $((\rad(\gb_1), \ldots, (\rad(\gb_n))$, while the \emph{radius} of $\gp$ means $\min \rad(\gp)$. The open and closed polydiscs centered at $a = (a_1, \ldots, a_n) \in \VF^n$ with radii $\gamma = (\gamma_1, \ldots, \gamma_n) \in |\Gamma|^n$ are denoted by $\go(a, \gamma)$ and $\gc(a, \gamma)$, respectively.

An $\RV$-polydisc $\rv^{-1}(t_1, \ldots, t_n) \times \{ s \}$ is \emph{degenerate} if $t_i = \infty$ for some $i$.

The \emph{$\RV$-hull} of a subset $A$, denoted by $\RVH(A)$, is the union of all the $\RV$-polydiscs whose intersections with $A$ are nonempty. If $A$ equals $\RVH(A)$ then $A$ is called an \emph{$\RV$-pullback}.
\end{defn}

At times it will be more convenient to work in the traditional expansion $\gC^{\textup{eq}}$ of $\gC$. However, a much simpler expansion $\gC^{\bullet}$ suffices: it has only one additional sort $\DC$ that contains, as elements, all the open and closed discs (we do think of $\VF$ as a subset of $\DC$). Heuristically, we may think of a disc that is properly contained in an $\RV$-disc as a ``thickened'' point of certain stature in $\VF$. When we work in $\gC^{\bullet}$, the underlying substructure $\mdl S$ may contain discs with valuative radii in $\Gamma(\mdl S)$ (or more appropriately, in $|\Gamma(\mdl S)|$); whether parameters in $\DC$ are used or not shall be indicated explicitly, if it is necessary. Note that it is redundant to include in $\DC$ discs centered at $0$, since they may be identified with their valuative radii. This expansion can help reduce the technical complexity of our discussion. However, as is the case with the imaginary $\Gamma$-sort, it is conceptually inessential since, for the purpose of this paper, all allusions to discs as (imaginary) elements may be eliminated in favor of objects already definable in $\gC$.

For a disc $\ga \sub \VF$, the corresponding imaginary element in $\DC$ is denoted by $\dot \ga$ when notational distinction makes the discussion more streamlined. Conversely, a subset $D \sub \DC$ is often identified with the set $\{\ga : \dot \ga \in D\}$, in which case $\bigcup D$ denotes a subset of $\VF$.

\begin{defn}
Let $\ga$, $\gb$ be two discs. Note that possibly $\gb \sub \ga$ or $\ga \sub \gb$. The subset
\[
\{a \in \VF : \ga < a < \gb \}
\]
is called a \emph{proper $\vv$-interval} and is denoted by $(\ga, \gb)$, while the subset
\[
\{a \in \VF : \ex{x \in \ga, y \in \gb} ( x \leq a \leq y) \}
\]
is called an \emph{improper $\vv$-interval} and is denoted by $[\ga, \gb]$. The $\vv$-intervals $[\ga, \gb)$, $(-\infty, \gb]$, etc., are defined in the obvious way, where $(-\infty, \gb]$ is a half $\vv$-interval that is unbounded from below. Let $A$ be such a $\vv$-interval. The discs $\ga$, $\gb$ are called the \emph{end-discs} of $A$. If $\ga$, $\gb$ are both points in $\VF$ then of course we just say that $A$ is an interval and if $\ga$, $\gb$ are both $\RV$-discs then we say that $A$ is an $\RV$-interval. If $A = (\ga, \gb]$, where $\ga$ is an open disc and $\gb$ is the smallest closed disc containing $\ga$, then $A$ is called a \emph{half thin annulus}.

Obviously an $\vv$-interval is definable if and only if its end-discs are definable. Two $\vv$-intervals are \emph{properly disconnected} if their union is not a $\vv$-interval.

More generally, a subset $A$ is a \emph{$\vv$-box} if it is a product $\prod_i A_i$ of $\vv$-intervals. In that case, $A$ is \emph{degenerate} if some $A_i$ is a point in $\VF$ and $A$ is simply a \emph{box} if every $A_i$ is an interval.
\end{defn}

\begin{prop}[Valuation property]
Let $(\mdl R, \OO) \prec (\mdl R', \OO')$ be models of $T_{\textup{convex}}$, where $\mdl R'$ is generated by an $a \in \VF$ over $\mdl R$ and $\vv(\mdl R) \neq \vv(\mdl R')$. Then there is a $d \in \mdl R$ such that $\vv(a - d) \notin \vv(\mdl R)$.
\end{prop}
\begin{proof}
See~\cite[Proposition~9.2]{DriesSpei:2000} and the remark thereafter.
\end{proof}

\begin{rem}\label{rem:HNF}
By the valuation property and \cite[Proposition~7.6]{Dries:tcon:97}, we have an important tool called \emph{Holly normal form} \cite[Theorem~4.8]{holly:can:1995} (henceforth abbreviated as HNF); that is, every definable subset $A \sub \VF$ is a unique union of finitely many pairwise definable properly disconnected $\vv$-intervals.
\end{rem}

\begin{conv}\label{conv:can}
We reiterate \cite[Convention~4.20]{Yin:QE:ACVF:min} here, with a different terminology, since this trivial-looking convention is actually quite crucial for understanding the whole construction, especially the parts that involve special bijections. For any subset $A$, let
\[
\can(A) = \{(a, \rv(a), t) : (a, t) \in A \text{ and } a \in \pvf(A)\}.
\]
The natural bijection $\can : A \fun \can(A)$ is called the \emph{regularization} of $A$. The convention is that we shall tacitly substitute $\can(A)$ for $A$ in the discussion if it is necessary or is just more convenient. Whether this substitution has been performed or not should be clear in context.
\end{conv}

\begin{ter}
Let $\gamma \in |\Gamma|$. A subset of $|\Gamma_{\infty}|^n$ is \emph{$\gamma$-bounded} if it is contained in the box $[\gamma, \infty]^n$ and is \emph{doubly $\gamma$-bounded} if it is contained in the box $[-\gamma, \gamma]^n$. More generally, let $A$ be a subset of $\VF^n \times \RV^m \times \Gamma_{\infty}^l$ and
\[
|A|_{\Gamma} = \{(|\vv|(a), |\vrv|(t), |\gamma|) : (a, t, \gamma) \in A\} \sub |\Gamma_{\infty}|^{n+m+l};
\]
then we say that $A$ is \emph{$\gamma$-bounded} (resp.\ \emph{doubly $\gamma$-bounded}) if $|A|_{\Gamma}$ is $\gamma$-bounded (resp.\ doubly $\gamma$-bounded).

\end{ter}

\section{Concerning the landscape of definable subsets in $\VF$}

We say that a definable function $f$ is \emph{quasi-$\lan{T}{}{}$-definable} if it is a restriction of an $\lan{T}{}{}$-definable function (with parameters in $\VF(\mdl S)$, of course).

\begin{lem}\label{fun:suba:fun}
Every definable function $f : \VF^n \fun \VF$ is piecewise quasi-$\lan{T}{}{}$-definable; that is, there are a definable finite partition $A_i$ of $\VF^n$ and $\lan{T}{}{}$-definable functions $f_i: \VF^n \fun \VF$ such that $f \rest A_i = f_i \rest A_i$ for all $i$.
\end{lem}
\begin{proof}
By induction on $n$ and compactness, this is immediately reduced to the case $n = 1$. In that case, let $\phi(X, Y)$ be a quantifier-free formula that defines $f$. Let $F_i(X, Y)$ enumerate the occurring $\lan{T}{}{}$-terms in $\phi(X, Y)$. For each $a \in \VF$ and each $F_i(a, Y)$ let $B_{a, i} \sub \VF$ be the finite subset given by \omin-minimal monotonicity (see \cite[\S1.3]{DriesLew95}). It is not hard to see that if $f(a) \notin \bigcup_i B_{a, i}$ then there would be a $b \neq f(a)$ such that $\phi(a, b)$ holds, contradicting the assumption that $f$ is a function. The assertion follows.
\end{proof}

\begin{cor}[Monotonicity]\label{mono}
Let $A \sub \VF$ and $f : A \fun \VF$ be a definable function. Then there is a definable finite partition of $A$ into $\vv$-intervals $A_i$ such that every $f \rest A_i$ is quasi-$\lan{T}{}{}$-definable, continuous, and monotone (constant or strictly increasing or strictly decreasing). Consequently, each $f(A_i)$ is a $\vv$-interval.
\end{cor}
\begin{proof}
This is immediate by Lemma~\ref{fun:suba:fun}, \omin-minimal monotonicity, and HNF.
\end{proof}

\begin{cor}\label{uni:fun:decom}
For the function $f$ in Corollary~\ref{mono}, there is a definable function $\pi : A \fun\RV$ such that, for each $t \in \ran(\pi)$, $f \rest \pi^{-1}(t)$ is either constant or injective.
\end{cor}
\begin{proof}
This is immediate by monotonicity. In fact, the proof of \cite[Lemma~4.11]{Yin:QE:ACVF:min} still works.
\end{proof}

Let $p : A \fun |\Gamma|$ be a definable function, where $A \sub \VF^n$. We say that $p$ is an \emph{$\go$-partition} of $A$ if for any $a \in A$ the function $p$ is constant on $\go(a, p(a)) \cap A$.

\begin{lem}\label{vol:par:bounded}
If $p$ is an $\go$-partition and $A$ is closed and bounded then $p(A)$ is doubly bounded.
\end{lem}
\begin{proof}
Since $A$ is bounded, $p(A)$ must be bounded as well. We first consider the case $n=1$. Suppose for contradiction that $p(A)$ is not bounded from the other direction. For each $\gamma \in |\Gamma|$ let $A_{\gamma} = \set{a \in A : p(a) > \gamma}$. For all $c \in \MM \mi \{0\}$, since $\mdl S \la c \ra_T$ is a model of $\TCVF$, we have $A_{|\vv|(c)} \cap \mdl S \la c \ra_T \neq \0$. By compactness, there is a definable function $f : \MM \mi \{0\} \fun A$ such that $p(f(c)) > |\vv|(c)$ for every $c \in \MM \mi \{0\}$. By monotonicity,
\[
a^+ \coloneqq \lim_{x \rightarrow 0^+} f \quad \text{and} \quad a^- \coloneqq \lim_{x \rightarrow 0^-} f
\]
exist and, since $A$ is closed, they are contained in $A$. Therefore, there is a positive $c \in \MM \mi \{0\}$ such that $|\vv|(c) > p(a^+)$ and $f(c) \in \go(a^+, p(a^+))$. Since $p$ is an $\go$-partition, this implies $p(f(c)) = p(a^+)$, which is a contradiction.

Suppose that $n > 1$. For each $a \in \pr_{<n}(A)$ let $p_a : A_a \fun |\Gamma|$ be the $\go$-partition induced by $p$ and $\beta_a \in |\Gamma|$ an $a$-definable upper bound of $p_a(A_a)$. Observe that, for all $a' \in \go(a, \beta_a) \cap \pr_{<n}(A)$, $p_a(A_a) = p_{a'}(A_{a'})$ and hence, by \omin-minimality in the $\Gamma$-sort, we may assume that the function $p' : \pr_{<n}(A) \fun |\Gamma|$ given by $a \efun \beta_a$ is an $\go$-partition. Now the assertion follows from a routine induction on $n$.
\end{proof}

This lemma is crucial for the good behavior of motivic Fourier transform (see \cite[\S11]{hrushovski:kazhdan:integration:vf} and \cite{yin:hk:part:3}); in this paper we shall use it only once in \S6.

\begin{lem}\label{RV:no:point}
For $t = (t_1, \ldots, t_n) \in \RV$, if $a \in \VF$ is $t$-definable then $a$ is definable. Similarly, for $\gamma = (\gamma_1, \ldots, \gamma_n) \in \Gamma$, if $t \in \RV$ is $\gamma$-definable then $t$ is definable.
\end{lem}
\begin{proof}
The first assertion is easily seen through an induction on $n$ with the trivial base case $n=0$. For any $b \in \rv^{-1}(t_n)$, by the inductive hypothesis, we have $a \in \VF(\la b \ra)$. If $a$ were not definable then we would have $b \in \VF(\la a \ra)$ and hence $\rv^{-1}(t_n) \sub \VF(\la a \ra)$, which is impossible. The second assertion is similar, using the exchange principle in the $\RV$-sort (see Remark~\ref{rem:RV:weako}).
\end{proof}

\begin{cor}\label{function:rv:to:vf:finite:image}
Let $U \sub \RV^m$ be a definable subset and $f : U \fun \VF^n$ a definable function. Then $f(U)$ is finite.
\end{cor}
\begin{proof}
We may assume $n=1$. Then this is immediate by Lemma~\ref{RV:no:point} and compactness.
\end{proof}

There is a slightly more general version of Lemma~\ref{RV:no:point} that involves parameters in the $\DC$-sort:

\begin{lem}\label{ima:par:red}
Let $\dot \ga = (\dot \ga_1, \ldots, \dot \ga_n) \in \DC^n$. If $a \in \VF$ is $\dot \ga$-definable then $a$ is definable.
\end{lem}
\begin{proof}
We do induction on $n$. Let $b \in \ga_n$ and $t \in \RV$ such that $\vrv(t) = \rad(\ga_n)$. Then $a$ is $(\dot \ga_1, \ldots, \dot \ga_{n-1}, t, b)$-definable. By the inductive hypothesis and Lemma~\ref{RV:no:point}, we have $a \in \VF(\la b \ra)$. If $a$ were not definable then we would have $b \in \VF(\la a \ra)$ and hence $\ga_n \sub \VF(\la a \ra)$, which is impossible.
\end{proof}

\begin{defn}
Let $D$ be a subset. We say that a (not necessarily definable) nonempty subset $A$ \emph{generates a (complete) $D$-type} if, for every $D$-definable subset $B$, either $A \sub B$ or $A \cap B = \0$. In that case, $A$ is \emph{$D$-type-definable} if no subset properly contains $A$ and also generates a $D$-type. If $A$ is $D$-definable and generates a $D$-type, or equivalently, if $A$ is $D$-definable and $D$-type-definable then we say that $A$ is \emph{$D$-atomic} or \emph{atomic over $D$}.
\end{defn}

\begin{rem}\label{rem:type:atin}
It is easy to see that, by HNF, if $\gi \sub \VF$ is atomic then $\gi$ must be a $\vv$-interval; in fact, there are only four possibilities for $\gi$: a point, an open disc, a closed disc, and a half thin annulus.
\end{rem}

\begin{lem}\label{atom:gam}
Let $\ga$ be an atomic subset. Then $\ga$ remains $\gamma$-atomic for all $\gamma \in \Gamma$; if $\ga \sub \VF^n$ is an open polydisc then it remains $\dot \ga$-atomic.
\end{lem}
\begin{proof}
The first assertion is a direct consequence of definable choice in the $\Gamma$-sort. For the second assertion, let $\gamma = \rad(\ga)$. If $\ga$ were not $\dot \ga$-atomic then, by compactness, there would be a $\gamma$-definable subset $A \sub \VF^n$ such that $A \cap \ga$ is a nonempty proper subset of $\ga$, which contradicts the first assertion that $\ga$ is $\gamma$-atomic.
\end{proof}

Recall from \cite[Definition~4.5]{mac:mar:ste:weako} the notion of a cell in a weakly \omin-minimal structure; in our setting, for the $\VF$-sort, we may require that the images of the bounding functions $f_1$, $f_2$ are contained in $\DC$; then cell decomposition \cite[Theorem~4.6]{mac:mar:ste:weako} holds accordingly. Note that cells are in general not invariant under coordinate permutations; however, an atomic subset of $\VF^n$ must be a cell and must remain so under coordinate permutations.

\begin{nota}\label{nota:tor}
For each $\gamma \in |\Gamma|$ with $0 \leq \gamma < \infty$ let $\MM_\gamma$ and $\OO_{\gamma}$ be the open disc and the closed disc around $0$ with valuative radius $\gamma$, respectively. Let $\RV_{\gamma} = \VF^{\times} / (1 + \MM_\gamma)$, which is a subset of $\DC$ and is also an ordered abelian group. The canonical projection $\VF^{\times} \fun \RV_{\gamma}$ is denoted by $\rv_{\gamma}$ and is augmented by $\rv_{\gamma}(0) = \infty$. If $\dot \gb \in \DC$ and $\rad(\gb) \leq \gamma$ then $\gb$ is a union of open discs of the form $\rv_{\gamma}^{-1}(\dot \ga)$. In this case, we shall abuse the notation slightly and write $\dot \ga \in \gb$, $\gb \sub \RV_{\gamma}$, etc. For each $\dot \ga \in \RV_{\gamma}$ let $\tor (\dot \ga) \sub \RV_{\gamma}$ be the $\dot \ga$-definable subset such that $\rv^{-1}_{\gamma}(\tor (\dot \ga))$ is the smallest closed disc containing $\ga$. Set
\begin{gather*}
\tor^{\times}(\dot \ga) = \tor (\dot \ga) \mi \{\dot \ga\},\\
\tor^+(\dot \ga) = \{t \in \tor(\dot \ga):  t > \dot \ga\}, \quad \tor^-(\dot \ga) = \{t \in \tor(\dot \ga):  t < \dot \ga\}.
\end{gather*}
If $\dot \ga = (\dot \ga_1, \ldots, \dot \ga_n)$ with $\dot \ga_i \in \RV_{\gamma_i}$ then $\prod_i \tor(\dot \ga_i)$ is simply written as $\tor(\dot \ga)$; similarly for $\tor^{\times}(\dot \ga)$, $\tor^+(\dot \ga)$, etc.
\end{nota}

\begin{rem}\label{rem:K:aff}
Let $\dot \ga \in \RV_{\gamma}$. Since, via additive translation by $\dot \ga$, there is a canonical $\dot \ga$-definable order-preserving bijection \[
\aff_{\dot \ga} :\tor(\dot \ga) \fun \vrv^{-1}(\pm \gamma) \cup \{0\} \sub \RV,
\]
we see that $\dot \ga$-definable subsets of $\tor(\dot \ga)^n$ naturally correspond to those of
\[
(\vrv^{-1}(\pm \gamma)\cup \{0\})^n.
\]
If there is an $\dot \ga$-definable $t \in \vrv^{-1}(\pm \gamma)$ then, via multiplicative translation by $t$, this correspondence may be extended to $\dot \ga$-definable subsets of $\K^n$. More generally, for any $t \in \vrv^{-1}(\pm \gamma)$, the induced bijection $\tor(\dot \ga) \fun \K$ is denoted by $\aff_{\dot \ga, t}$. Consequently, $\tor(\dot \ga)$ may be viewed as a $\K$-torsor and, as such, is equipped with much of the structure of $\K$.

For example, if $\dot \gb \in \RV_{\beta}$ and $f : \tor(\dot \ga) \fun \tor(\dot \gb)$ is a function then we may define the derivative $\frac{d}{dx} f$ of $f$ at any point $\dot \gd \in \tor(\dot \ga)$ as follows. Choose any $t \in \vrv^{-1}(\pm \gamma)$ and any $s \in \vrv^{-1}(\pm \beta)$. Set $r = \aff_{\dot \ga, t}(\dot \gd)$ and consider the function
\[
f_{\dot \ga, \dot \gb, t,s} : \K \to^{\aff^{-1}_{\dot \ga, t}} \tor(\dot \ga) \to^f \tor(\dot \gb) \to^{\aff_{\dot \gb, s}} \K.
\]
Suppose that $\frac{d}{dx} f_{\dot \ga, \dot \gb, t,s}(r) \in \K$ exists. Then we set
\[
\tfrac{d}{d x} f(\dot \gd) \coloneqq s t^{-1} \tfrac{d}{d x} f_{\dot \ga, \dot \gb, t,s}(r) \in \vrv^{-1}(\pm \beta \gamma^{-1}) \cup \{0\}.
\]
It is routine to check that this construction does not depend on the choice of $\dot \ga$, $\dot \gb$, $t$, and $s$.
\end{rem}

Observe that if $\gamma$ is a definable point in $\Gamma$ then $\vrv^{-1}(\gamma)$ must contain a definable point in $\RV$.

\begin{lem}\label{gk:ortho}
If $f : \Gamma \fun \K$ is a definable function then $f(\K)$ is finite. Similarly, if $g : \K \fun \Gamma$ is a definable function then $g(\Gamma)$ is finite.
\end{lem}
\begin{proof}
See \cite[Proposition~5.8]{Dries:tcon:97}.
\end{proof}

In other words, the $\K$-sort and the $\Gamma$-sort are orthogonal to each other.

\begin{lem}\label{open:K:con}
In $\gC^{\bullet}$, let $\ga \sub \VF$ be an open disc and $f : \ga \fun \K$ a definable nonconstant function. Then there is a definable proper subdisc $\gb \sub \ga$ such that $f \rest (\ga \mi \gb)$ is constant.
\end{lem}
\begin{proof}
If $\gb_1$ and $\gb_2$ are two proper subdiscs of $\ga$ such that $f \rest (\ga \mi \gb_1)$ and $f \rest (\ga \mi \gb_2)$ are both constant then $\gb_1$ and $\gb_2$ must be concentric, for otherwise $f$ would be constant. Therefore, it is enough to show that $f \rest (\ga \mi \gb)$ is constant for some proper subdisc $\gb \sub \ga$. To that end, without loss of generality, we may assume that $\ga$ is centered at $0$. For each $\gamma \in |\vv|(\ga) \sub |\Gamma|$, by \cite[Theorem~A]{Dries:tcon:97} and \omin-minimality, $f(\vv^{-1}(\pm \gamma))$ contains a $\gamma$-definable element $t_{\gamma}$. By weak \omin-minimality, $f(\vv^{-1}(\pm \gamma)) = \{t_{\gamma}\}$ for all but finitely many $\gamma \in |\vv|(\ga)$. Let $g : |\vv(\ga)| \fun \K$ be the definable function given by $\gamma \fun t_{\gamma}$. By Lemma~\ref{gk:ortho}, the image of $g$ is finite. The assertion follows.

Alternatively, we may simply quote \cite[Theorem~1.2]{jana:omin:res}.
\end{proof}

\begin{lem}\label{open:rv:cons}
In $\gC^{\bullet}$, let $\ga \sub \VF^n$ be an atomic open polydisc and $f : \ga \fun \VF$ a definable function. If $f$ is not constant then $f(\ga)$ is an (atomic) open disc; in particular, $\rv \rest f(\ga)$ is always constant.
\end{lem}
\begin{proof}
Note that, by atomicity, $f(\ga)$ must be an atomic $\vv$-interval. We proceed by induction on $n$. For the base case $n=1$, it is easy to see that, by monotonicity and \cite[Proposition~4.2]{Dries:tcon:97}, $f(\ga)$ must be either a point or an open disc. For the case $n > 1$, suppose for contradiction that $f(\ga)$ is a closed disc (other than a point) or a half thin annulus. By the inductive hypothesis, for every $a \in \pr_{1}(\ga)$ there is a maximal open subdisc $\gb_a \sub f(\ga)$ that contains $f(\ga_a)$. Applying (the proof of) Lemma~\ref{open:K:con} to the $\dot \gh$-definable function on $\pr_{1}(\ga)$ given by $a \efun \aff_{\dot \gh}(\dot{\gb}_a)$, where $\gh$ is a maximal open subdisc of $f(\ga)$,  we see that some closed proper subdisc of $\ga$ is definable (the parameter $\dot \gh$ is not needed). This yields the desired contradiction.
\end{proof}

\begin{cor}\label{poly:open:cons}
Let $f : \VF^n \fun \VF$ be a definable function and $\ga \sub \VF^n$ an open polydisc. If $(\rv \circ f) \rest \ga$ is not constant then there is an $\dot \ga$-definable nonempty proper subset of $\ga$.
\end{cor}

Here is a strengthening of Lemma~\ref{atom:gam}:

\begin{lem}\label{atom:self}
Let $\ga = \ga_1 \times \ldots \times \ga_n \sub \VF^n$ be an open polydisc that generates a type. Then for all $a = (a_1, \ldots, a_n)$ and $b = (b_1, \ldots, b_n)$ in $\ga$ there is an immediate automorphism $\sigma$ of $\gC$ such that $\sigma(a) = b$. Consequently, $\ga$ is $(\dot \ga, t)$-atomic for all $t \in \RV$.
\end{lem}
\begin{proof}
To see that the first assertion implies the second, suppose for contradiction that there is an $(\dot \ga, t)$-definable nonempty proper subset $A \sub \ga$. Let $a \in A$, $b \in \ga \mi A$, and $\sigma$ be an immediate automorphism of $\gC$ such that $\sigma(a) = b$. It follows that $\sigma$ is an immediate automorphism of $\gC$ over $\la \dot \ga, t \ra$, contradicting the assumption that $A$ is $(\dot \ga, t)$-definable.

For the first assertion, by Lemma~\ref{imm:ext}, it is enough to show that there is an immediate isomorphism $\sigma : \la a \ra \fun \la b \ra$ sending $a$ to $b$. Let $\ga' = \ga_1 \times \ldots \times \ga_{n-1}$, $a' =  (a_1, \ldots, a_{n-1})$, and $b' =  (b_1, \ldots, b_{n-1})$. Then, by induction on $n$ and Lemma~\ref{imm:iso}, it is enough to show that, for any immediate isomorphism $\sigma' : \la a' \ra \fun \la b' \ra$ sending $a'$ to $b'$ and any definable function $f : \VF^{n-1} \fun \VF$,
\[
\rv(a_n - f(a')) = \rv(b_n - \sigma'(f(a'))).
\]
This is clear for the base case $n=1$. For the case $n > 1$, by the inductive hypothesis and Lemma~\ref{open:rv:cons}, $f(\ga') = \sigma'(f(\ga'))$ is either a point or an open disc that is disjoint from $\ga_n$ and hence the desired condition is satisfied.
\end{proof}

\begin{cor}\label{part:rv:cons}
Let $A \sub \VF^n$ and $f : A \fun \VF$ be a definable function. Then there is a definable finite partition $A_i$ of $A$ such that, for all $i$, if $\ga \sub A_i$ is an open polydisc then $\rv \rest f(\ga)$ is constant.
\end{cor}
\begin{proof}
Suppose for contradiction that the assertion fails. For each $a \in A$, let $D_a \sub A$ be the type-definable subset containing $a$. By Lemma~\ref{atom:self}, every open polydisc $\ga \sub D_a$ is $\dot \ga$-atomic and hence, by Lemma~\ref{open:rv:cons}, $\rv \rest f(\ga)$ is constant. By compactness, the assertion must hold in a definable subset $A_a \sub A$ that contains $a$; by compactness again, it holds in finitely many definable subsets $A_1, \ldots, A_m$ of $A$ with $\bigcup_i A_i = A$. Then the partition of $A$ generated by $A_1, \ldots, A_m$ is as desired.
\end{proof}

If $\mdl S$ is $\VF$-generated then it is an elementary substructure and hence every definable subset contains a definable point. This of course fails if $\mdl S$ carries extra $\RV$-data. However, we do have:

\begin{lem}\label{clo:disc:bary}
Every definable closed disc $\gb$ contains a definable point.
\end{lem}
\begin{proof}
Suppose for contradiction that $\gb$ does not contain a definable point. Since $\gC$ is sufficiently saturated, there is an open disc $\ga$ that is disjoint from $\VF(\mdl S)$ and properly contains $\gb$. Let $a \in \ga \mi \gb$ and $b \in \gb$. Clearly $\rv(c - b) = \rv(c - a)$ for all $c \in \VF(\mdl S)$. As in the proof of Lemma~\ref{atom:self}, there is an immediate automorphism $\sigma$ of $\gC$ such that $\sigma(a) = b$. This means that $\gb$ is not definable, which is a contradiction.
\end{proof}

Notice that the argument above does not work if $\gb$ is an open disc.

\begin{cor}\label{open:disc:def:point}
Let $\ga \sub \VF$ be an open disc and $A \sub \VF$ a definable subset. If $\ga \cap A$ is a nonempty proper subset of $\ga$ then $\ga$ contains a definable point.
\end{cor}
\begin{proof}
It is not hard to see that, by HNF, if $\ga \cap A$ is a nonempty proper subset of $\ga$ then $\ga$ contains a definable closed disc and hence the claim is immediate by Lemma~\ref{clo:disc:bary}.
\end{proof}

\begin{defn}\label{defn:otop}
Let $A$, $B$ be two subsets of $\VF$ and $f : A \fun B$ a bijection. We say that $f$ is \emph{concentric} if, for all open disc $\ga \sub A$, $f(\ga)$ is also an open disc; if both $f$ and $f^{-1}$ are concentric then $f$ has the \emph{open-to-open property} (henceforth abbreviated as ``otop''). Note that if $f$ is definable and has otop then it also has the closed-to-closed property, that is, for all closed disc $\gc \sub A$, $f(\gc)$ is also a closed polydisc, and vice versa.

More generally, let $f : A \fun B$ be a bijection between two subsets $A$ and $B$, each with exactly one $\VF$-coordinate. For each $(t, s) \in \pr_{> 2}(f)$, let
\[
f_{t, s} = f \cap (\VF^2 \times \{(t, s)\}),
\]
which is called an \emph{$\RV$-fiber} of $f$. We say that $f$ has \emph{otop} if every $\RV$-fiber of $f$ has otop.
\end{defn}

\begin{lem}\label{open:pro}
Let $f : A \fun B$ be a definable bijection between two subsets $A$ and $B$, each with exactly one $\VF$-coordinate. Then there is a definable finite partition $A_i$ of $A$ such that each $f \rest A_i$ has otop.
\end{lem}
\begin{proof}
By compactness, we may simply assume that $A$ and $B$ are subsets of $\VF$. Then we may proceed exactly as in the proof of Corollary~\ref{part:rv:cons}, using Lemmas~\ref{open:rv:cons} and~\ref{atom:self}.
\end{proof}

\begin{defn}
Let $A$ be a subset of $\VF^n$. The \emph{$\RV$-boundary} of $A$, denoted by $\partial_{\RV}A$, is the definable subset of $\rv(A)$ such that $t \in \partial_{\RV} A$ if and only if $\rv^{-1}(t) \cap A$ is a proper subset of both $\rv^{-1}(t)$ and $A$. The definable subset
\[
\ito_{\RV}(A) \coloneqq \rv(A) \mi \partial_{\RV}A
\]
is called the \emph{$\RV$-interior} of $A$.
\end{defn}

Obviously if $A \sub \VF^n$ is not an $\RV$-pullback and $\rv(A)$ is not a singleton then $\partial_{\RV} A$ is not empty. Note that $\partial_{\RV}A$ is in general different from the topological boundary of $\rv(A)$ in $\RV^n$.

\begin{lem}\label{RV:bou:dim}
Let $A$ be a definable subset of $\VF^n$. Then $\dim_{\RV}(\partial_{\RV} A) < n$.
\end{lem}
\begin{proof}
We do induction on $n$. Since the base case $n=1$ follows immediately from weak \omin-minimality (in both the $\VF$-sort and the $\RV$-sort), we proceed directly to the inductive step. Since $\partial_{\RV} A_a$ is finite for every $i \in [n]$ and every $a \in \pr_{\wt i}(A)$, by Corollary~\ref{open:disc:def:point} and compactness, there are a definable finite partition $A_{ij}$ of $\pr_{\wt i}(A)$ and, for each $A_{ij}$, finitely many definable functions $f_{ijk} : A_{ij} \fun \VF$ such that $\bigcup_k \rv(f_{ijk}(a)) = \partial_{\RV} A_a$ for all $a \in A_{ij}$. By Corollary~\ref{part:rv:cons}, we may assume that if $\rv^{-1}(t) \sub A_{ij}$ then $\rv \rest f_{ijk}(\rv^{-1}(t))$ is constant. Hence each $f_{ijk}$ induces a definable function
\[
C_{ijk} : (\rv(A_{ij}) \mi  \partial_{\RV} A_{ij}) \fun \RV.
\]
Let $C = \bigcup_{i, j, k} C_{ijk}$ and $B = \bigcup_{i,j} \bigcup_{t \in \partial_{\RV} A_{ij}} \rv(A)_t$. Obviously $\dim_{\RV}(C) < n$. By the inductive hypothesis, for all $A_{ij}$ we have $\dim_{\RV}(\partial_{\RV} A_{ij}) < n-1$. Thus $\dim_{\RV}(B) < n$. Since $\partial_{\RV} A \sub B \cup C$, the claim follows.
\end{proof}

\begin{defn}\label{defn:corr:cont}
Let $f : A \fun B$ be a function. We say that $f$ is \emph{contractible} if there is a (necessarily unique) function $f_{\downarrow} : \rv(A) \fun \rv(B)$, called the \emph{contraction} of $f$, such that
\[
(\rv \rest B) \circ f = f_{\downarrow} \circ (\rv \rest A).
\]

We say that $f$ is \emph{$\res$-contractible} if there is a (necessarily unique) function $f_{\downarrow} : \res(A) \fun \res(B)$, called the \emph{$\res$-contraction} of $f$, such that
\[
(\res \rest B) \circ f = f_{\downarrow} \circ (\res \rest A).
\]
We say that $f$ is \emph{$\Gamma$-contractible} if the same holds with respect to $\vv$ or $\vrv$, depending on the coordinates, instead of $\res$.

The subscripts in these contractions will be written as $\downarrow_{\rv}$, $\downarrow_{\res}$, $\downarrow_{\Gamma}$ if they occur in the same context and therefore need to be distinguished from one another notationally.
\end{defn}

\begin{lem}\label{fn:alm:cont}
Let $f : \VF^n \fun \VF$ be a definable function. Then there is a definable subset $U \sub \RV^n$ such that $\dim_{\RV}(U) < n$ and $f \rest (\VF^n \mi  \rv^{-1}(U))$ is contractible.
\end{lem}
\begin{proof}
By Corollary~\ref{poly:open:cons}, for any $t \in \RV^n$, if $\rv(f(\rv^{-1}(t)))$ is not a singleton then $\rv^{-1}(t)$ has a $t$-definable proper subset. By compactness there is a definable subset $A \sub \VF^n$ such that $t \in \partial_{\RV} A$ if and only if $\rv(f(\rv^{-1}(t)))$ is not a singleton. So the assertion follows from Lemma~\ref{RV:bou:dim}.
\end{proof}

Let $f : \VF^n \fun \VF^m$ be a definable function. By Lemma~\ref{fun:suba:fun} and \omin-minimal differentiability, $f$ is $C^p$ almost everywhere for all $p$ (see \cite[\S7.3]{dries:1998}). For each $p$ let $\reg^p(f) \sub \VF^n$ be the definable subset of regular $C^p$-points of $f$. If $p=0$ then we write $\reg(f)$, which is simply the subset of regular points of $f$. Obviously if $a \in \reg(f)$ and $f$ is $C^1$ in a neighborhood of $a$ then $\reg^1(f)$ contains a neighborhood of $a$ on which the sign of the Jacobian of $f$, which is denoted by $\jcb_{\VF} f$, is constant. If $n=m$ and $f$ is locally injective on a definable open subset $A \sub \VF^n$ then $f$ is regular almost everywhere on $A$ and hence
\[
\dim_{\VF}(A \mi \reg^p(f)) < n
\]
for all $p$. By \cite[Theorem~A]{Dries:tcon:97}, the situation is the same if $f$ is a definable function of the form
\[
(\vrv^{-1}(\pm \alpha)\cup \{0\})^n \fun (\vrv^{-1}(\pm \beta)\cup \{0\})^m,
\]
in particular, from $\K^n$ into $\K^m$, or more generally, from $\tor(u)$ into $\tor(v)$, where $u \in \RV^n_{\alpha}$ and $v \in \RV^m_{\beta}$ (see Remark~\ref{rem:K:aff}).

Suppose that $f$ is a definable function from $\UU$ into $\OO$. By monotonicity, there are a definable finite subset $B \sub \UU$ and a definable finite partition of $A \coloneqq \UU \mi B$ into infinite $\vv$-intervals $A_i$ such that $f$ and $\frac{d}{d x} f$ are quasi-$\lan{T}{}{}$-definable, continuous, and monotone on each $A_i$. If $\rv(A_i)$ is not a singleton then let $U_i \sub \K$ be the largest open interval contained in $\rv(A_i)$. Let $U = \bigcup_i U_i$ and $f^* = f \rest \rv^{-1}(U)$. By Corollary~\ref{part:rv:cons}, we may refine the partition such that both $f^*$ and $\frac{d}{d x} f^*$ are contractible. By Lemma~\ref{gk:ortho},
\[
\vv \rest f^*(\rv^{-1}(U_i)) \quad \text{and} \quad \vv \rest \tfrac{d}{d x} f^*(\rv^{-1}(U_i))
\]
must be constant, say $\alpha_i$ and $\beta_i$, respectively. So it makes sense to speak of $\frac{d}{d x} f^*_{\downarrow}$ on each $U_i$. Deleting finitely many points from $U$ if necessary, we may assume that $f^*_{\downarrow}$, $(\frac{d}{d x} f^*)_{\downarrow}$, and $\frac{d}{d x} f^*_{\downarrow}$ are all continuous monotone functions on each $U_i$. Note that
\begin{itemize}
  \item $|\beta_i| \geq |\alpha_i|$,
  \item either $f^*_{\downarrow} \rest U_i$ is constant or $| \vv(\frac{d}{d \! x} f^*_{\downarrow}(U_i)) | = \{|\alpha_i|\}$.
\end{itemize}
A moment of reflection shows that we must have $|\beta_i| = |\alpha_i|$ unless $f^*_{\downarrow} \rest U_i$ is constant, for otherwise $f^* \rest \rv^{-1}(U_i)$ would increase or decrease too slowly to make $f^*_{\downarrow}(U_i)$ contain more than one point; in fact, since $f^* \rest \rv^{-1}(U_i)$ cannot increase or decrease too fast either, we see that $|\beta_i| = |\alpha_i| < \infty$ if and only if $f^*_{\downarrow} \rest U_i$ is not constant; by a standard estimate argument, if $|\beta_i| = |\alpha_i|$ then $(\frac{d}{d x} f^*)_{\downarrow} = \frac{d}{d x} f^*_{\downarrow}$ on each $U_i$. More generally:

\begin{lem}\label{univar:der:contr}
Let $A  \sub \UU^n$ be a definable $\RV$-pullback with $\dim_{\RV}(\rv(A)) = n$ and $f : A \fun \OO$ a definable function. Let $p \in \N^n$ be a multi-index of order $|p| = d$ and $k \in \N$ with $k \gg d$. Suppose that $f$ is $C^k$ and, for all $q \leq p$, $\frac{\partial^q}{\partial x^q} f$ is contractible and its contraction $(\frac{\partial^q}{\partial x^q} f)_{\downarrow}$ is also $C^k$. Then there is a definable subset $V \sub \rv(A)$ with $\dim_{\RV}(V) < n$ and $U \coloneqq \rv(A) \mi V$ open such that, for all $a \in \rv^{-1}(U)$ and all $q' < q \leq p$ with $|q'| + 1 = q$, exactly one of the following two conditions holds:
\begin{itemize}
 \item either $\frac{\partial^{q}}{\partial x^{q}} f(a) = 0$ or $| \vv(\frac{\partial^{q'}}{\partial x^{q'}} f(a)) | < | \vv(\frac{\partial^{q}}{\partial x^{q}} f(a)) |$,
 \item $(\frac{\partial^{q - q'}}{\partial x^{q - q'}} \frac{\partial^{q'}}{\partial x^{q'}} f)_{\downarrow}(\rv (a)) = \frac{\partial^{q - q'}}{\partial x^{q - q'}}(\frac{\partial^{q'}}{\partial x^{q'}} f)_{\downarrow}(\rv( a)) \neq 0$;
\end{itemize}
if the first condition never occurs then, for all $q \leq p$, we actually have $(\frac{\partial^q}{\partial x^q} f )_{\downarrow} = \frac{\partial^{q}}{\partial x^{q}} f_{\downarrow}$ on $U$; at any rate, for all $q \leq p$, we have $(\frac{\partial^q}{\partial x^q} f )_{\downarrow_{\res}} = \frac{\partial^{q}}{\partial x^{q}} f_{\downarrow_{\res}}$ on $U$.
\end{lem}
\begin{proof}
Observe that, by induction on $d$, it is enough to consider the case $d =1$. In that case, for each $a \in \pr_{<n}(A)$, by the discussion above, there is an $a$-definable finite subset $V_{a} \sub \fib(\rv(A), \rv(a))$ such that the assertion holds for the restriction $f \rest (A_a \mi \rv^{-1}(V_{a}))$. Let
\[
A^* = \bigcup_{a \in \pr_{<n}(A)} \rv^{-1}(V_{a}).
\]
By Lemma~\ref{RV:bou:dim}, $\dim_{\RV}(\partial_{\RV} A^*) < n$ and hence $\dim_{\RV}(\rv(A^*)) < n$. Therefore, by Lemma~\ref{fn:alm:cont}, there is a definable open subset $U \sub \ito(\rv(A) \mi \rv(A^*))$ that is as desired.
\end{proof}

Let $f = (f_1, \ldots, f_m) : A \fun \OO$ be a sequence of definable $\res$-contractible functions, where $A$ is as in Lemma~\ref{univar:der:contr}. Let $P(X_1, \ldots, X_m)$ be a partial differential operator with definable $\res$-contractible coefficients $a_i : A \fun \OO$ and $P_{\downarrow}(X_1, \ldots, X_m)$ the corresponding operator with $\res$-contracted coefficients $a_{i\downarrow} : \res(A) \fun \K$. Note that both $P(f) : A \fun \OO$ and $P_{\downarrow}(f_{\downarrow}) : \res(A) \fun \K$ are defined almost everywhere. By Lemma~\ref{univar:der:contr}, such an operator $P$ almost commutes with $\res$:

\begin{cor}\label{rv:op:comm}
For almost all $t \in \rv(A)$ and almost all $a \in \rv^{-1}(t)$,
\[
\res(P(f)(a)) = P_{\downarrow}(f_{\downarrow})(\res(a)).
\]
\end{cor}

The following corollary will not be needed below.

\begin{cor}
Let $f : \rv^{-1}(U) \fun \rv^{-1}(V)$ be a definable contractible function, where $U$, $V$ are definably connected subsets of $(\K^{\times})^n$. Suppose that $f_{\downarrow}$ is continuous and locally injective. Then there is a definable subset $U^* \sub U$ of $\RV$-dimension $< n$ such that the sign of $\jcb_{\VF} f$ is constant on $\rv^{-1}(U \mi U^*)$.
\end{cor}
\begin{proof}
This follows immediately from Corollary~\ref{rv:op:comm} and \cite[Theorem~3.2]{pet:star:otop}.
\end{proof}

\begin{lem}\label{atom:type}
In $\gC^{\bullet}$, let $\ga \sub \VF$ be an atomic subset and $f : \ga \fun \VF$ a definable injection. Then $\ga$ and $f(\ga)$ must be of the same type (see Remark~\ref{rem:type:atin}).
\end{lem}
\begin{proof}
This is trivial if $\ga$ is a point. The case of $\ga$ being an open disc is covered by Lemma~\ref{open:rv:cons}. So we only need to show that if $\ga$ is a closed disc then $f(\ga)$ cannot be a half thin annulus. Suppose for contradiction that $\dot \ga = \tor(\dot \gm) \sub \RV_{\gamma}$ and $\dot{f(\ga)} = \tor^+(\dot \gn) \sub \RV_{\delta}$. By Lemma~\ref{open:pro} and monotonicity, $f$ induces an increasing (or decreasing, which can be handled similarly) bijection $f_{\downarrow} : \tor(\dot \gm) \fun \tor^+(\dot \gn)$. In fact, it is not hard to see that, for all $p$,
\[
\tfrac{d^p}{d x^p} f_{\downarrow} : \tor(\dot \gm) \fun \vrv^{-1}(\delta \gamma^{-p})
\]
cannot be constant and hence must be continuous and increasing. This means that there is a parametrically definable function $\K \fun \K^+$ that is not polynomial-bounded, which is a contradiction.
\end{proof}

\begin{defn}
Let $\ga$ be an open disc and $f : \ga \fun \VF$ an injection. We say that $f$ is \emph{$\Gamma$-linear} if there is a (necessarily unique) $\gamma \in \Gamma$ such that, for all $a, a' \in \ga$,
\[
\vv(f(a) - f(a')) = \gamma + \vv(a - a').
\]
We say that $f$ is \emph{$\rv$-linear} if there is a (necessarily unique) $t \in \RV$ such that, for all $a, a' \in \ga$,
\[
\rv(f(a) - f(a')) = t \rv(a - a').
\]
\end{defn}

Obviously $\rv$-linearity implies $\Gamma$-linearity. With the extra structure afforded by the total ordering, we can reproduce \cite[Lemma~3.18]{Yin:int:acvf} with a somewhat simpler proof:

\begin{lem}\label{rv:lin}
In $\gC^{\bullet}$, let $f : \ga \fun \gb$ be a definable bijection between two atomic open discs. Then $f$ is $\rv$-linear and hence $\Gamma$-linear with respect to $\rad(\gb) - \rad(\ga)$.
\end{lem}
\begin{proof}
Since $f$ has otop by Lemma~\ref{open:pro}, for all $\rad(\ga) < \gamma \leq \infty$ and all $\gd \coloneqq \tor(\dot \gc) \sub \rv_{\gamma}(\ga)$, it induces a $\dot \gd$-definable $C^1$ function $f_{\dot \gd} : \gd \fun \tor(\dot{f(\gc)})$. The codomain of its derivative $\frac{d}{d x} f_{\dot \gd}$ can be narrowed down to either $\vrv^{-1}(\delta \gamma^{-1})$ or $\vrv^{-1}(- \delta \gamma^{-1})$, where $\delta = \rad(f(\gc))$. In fact, by (the proof of) Lemma~\ref{atom:type}, $\frac{d}{d x} f_{\dot \gd}$ must be constant; but this is not actually needed.

By Lemma~\ref{open:rv:cons}, there is a $t \in \RV^{\times}$ such that $\frac{d}{d x} f(\ga) \sub \rv^{-1}(t)$. By (the proof of) Lemma~\ref{univar:der:contr}, for all $\gd$, all $\dot \gc \in \gd$, and all $a \in \gc$, $\frac{d}{d x} f_{\dot \gd}(\dot \gc) = \rv(\frac{d}{d x} f(a)) = t$ and hence
\[
\aff_{\dot{f(\gc)}} \circ f_{\dot \gd} \circ \aff^{-1}_{\dot \gc} : \vrv^{-1}(\pm \gamma) \cup \{0\} \fun \vrv^{-1}(\pm \delta) \cup \{0\}
\]
is a linear function (see Remark~\ref{rem:K:aff}). It follows that, for
\begin{itemize}
  \item $a$ and $a'$ in $\ga$,
  \item $\gd$ the smallest closed disc containing $a$ and $a'$,
  \item $\gc$ and $\gc'$ the maximal open subdiscs of $\gd$ containing $a$ and $a'$, respectively,
\end{itemize}
we have
\[
\rv(f(a) - f(a')) = \rv(f(\gc) - f(\gc')) = t \rv(\gc - \gc') = t \rv(a - a').
\]
That is, $f$ is $\rv$-linear. It is clear from the open-to-open property that $\vrv(t) = \rad(\gb) - \rad(\ga)$.
\end{proof}

\section{Grothendieck semirings of $\RV$ and $\Gamma$}

The main purpose of this section is to express the Grothendieck semiring of the $\RV$-category $\RV[*]$ as a tensor product of the Grothendieck semirings of the $\RES$-category $\RES[*]$ and the $\Gamma$-category $\Gamma[*]$, which will be defined below.

\begin{lem}\label{gam:red:K}
Let $D$ be a definable subset of $\Gamma^n$ with $\dim_{\Gamma}(D) = k$. Then $\vrv^{-1}(D)$ is definably bijective to a disjoint union of finitely many subsets of the form $(\K^+)^{n-k} \times \vrv^{-1}(D')$, where $D' \sub \Gamma^k$.
\end{lem}
\begin{proof}
Over a definable finite partition of $D$, we may assume that $D$ is a subset of $(\Gamma^+)^n$ and $\pr_{\leq k} \rest D$ is injective. By \cite[Theorem~4.4]{Dries:tcon:97}, the definable function $\pr_{\leq k}(D) \fun \pr_{>k}(D)$ is piecewise $\Q$-linear. The assertion follows.
\end{proof}

\begin{lem}\label{gam:pulback:mono}
Let $D$, $E$ be subsets of $\Gamma^n$ and $g : D \fun E$ a definable bijection. Then $g$ is definably a piecewise $\mgl_n(\Q)$-transformation (with constant terms).
\end{lem}
\begin{proof}
This is immediate by \cite[Theorem~4.4]{Dries:tcon:97} and \cite[Lemma~2.29]{Yin:int:expan:acvf}.
\end{proof}

\begin{lem}\label{resg:decom}
Let $A \sub \RV^k \times \Gamma^l$ be a definable subset. Set $\pr_{\leq k}(A) = U$ and suppose that $\vrv(U)$ is finite. Then there is a finite definable partition $U_i$ of $U$ such that, for each $i$ and all $t, t' \in U_i$, we have $A_t = A_{t'}$.
\end{lem}
\begin{proof}
By stable embeddedness, for every $t \in U$, $A_t$ is $\vrv(t)$-definable in the $\Gamma$-sort alone. Since $\vrv(U)$ is finite, the assertion simply follows from compactness.
\end{proof}

\begin{lem}\label{gam:tup:red}
Let $\beta$, $\gamma = (\gamma_1, \ldots, \gamma_m)$ be tuples in $\Gamma$. If there is a $\beta$-definable nonempty proper subset of $\vrv^{-1}(\gamma)$ then, for some $\gamma_i$, $\vrv^{-1}(\gamma_i)$ contains a $t$-definable point for every $t \in \vrv^{-1}(\wh \gamma_i)$. Consequently, if $U$ is a $\beta$-definable subset of $\vrv^{-1}(\gamma)$ then there is a (possibly empty) subtuple  $\gamma^* \sub  \gamma$ such that $\pr_{\gamma^*}(U) = \vrv^{-1}(\gamma^*)$, where $\pr_{\gamma^*}$ denotes the obvious coordinate projection, and there is a $\beta$-definable function from $\vrv^{-1}(\gamma^*)$ into $\vrv^{-1}(\gamma \mi \gamma^*)$.
\end{lem}
\begin{proof}
For the first assertion we do induction on $m$. The base case $m = 1$ simply follows from \omin-minimality in the $\K$-sort and Lemma~\ref{RV:no:point}. For the inductive step $m > 1$, let $U$ be a $\beta$-definable nonempty proper subset of $\vrv^{-1}(\gamma)$. By the inductive hypothesis, we may assume
\[
\{ t \in \pr_{>1}(U) : U_t \neq \vrv^{-1}(\gamma_1)\} = \vrv^{-1}(\wh \gamma_1).
\]
Then $\gamma_1$ is as desired.

The second assertion follows easily from the first.
\end{proof}

\begin{lem}\label{dim:cut:gam}
Let $U \sub \RV^n$ be a definable subset with $\dim_{\RV}(U) = k$. Then $\dim_{\RV}(U_{\gamma}) = k$ for some $\gamma \in \vrv(U)$.
\end{lem}
\begin{proof}
This follows from \cite[Theorem~4.11]{mac:mar:ste:weako} and \omin-minimality in the $\K$-sort.
\end{proof}

\begin{defn}[$\RV$-categories]\label{defn:c:RV:cat}
An object $U$ of the category $\RV_k$ is a definable subset in $\RV$ with $\dim_{\RV}(U) \leq k$. Any definable bijection between two such objects is a \emph{morphism} of $\RV_k$. Set $\RV_* = \bigcup_k \RV_k$.

An object of the category $\RV[k]$ is a definable pair $(U, f)$, where $U$ is a subset in $\RV$ and $f : U \fun (\RV^{\times})^k$ is a finite-to-one function. Given two such objects $(U, f)$ and $(V, g)$, any definable bijection $F : U \fun V$ is a \emph{morphism} of $\RV[k]$. Note that such a morphism $F$ induces a finite-to-finite correspondence between $f(U)$ and $g(V)$:
\[
F^{\rightleftharpoons} \coloneqq \{(t, s) \in f(U) \times g(V) : \ex{u \in U} (f(u) = t \wedge (g \circ F)(u) = s) \}.
\]
Set $\RV[\leq k] = \coprod_{i \leq k} \RV[i]$ and $\RV[*] = \coprod_k \RV[k]$.
\end{defn}

\begin{lem}\label{RV:decom:RES:G}
Let $U \sub \RV^m$ be a definable subset. Then there are finitely many objects $V_i \times \vrv^{-1}(D_i) \sub \K^{k_i} \times \vrv^{-1}(\Gamma)^{l_i}$ in $\RV_*$ such that $k_i + l_i = m$ for all $i$ and $[U] = \sum_i [V_i \times \vrv^{-1}(D_i)]$ in $\gsk \RV_*$.
\end{lem}
\begin{proof}
The case $m=1$ is an immediate consequence of HNF or weak \omin-minimality in the $\RV$-sort. For the case $m>1$, by Lemma~\ref{gam:tup:red}, compactness, and a routine induction on $m$, we may assume $U \sub \K^{k} \times \Gamma^{l}$ with $k + l = m$, that is, $U$ is a definable subset of $\K^{k} \times \RV^{l}$ that may be written as a union of subsets of the form $\{t\} \times \vrv^{-1}(\gamma)$, where $t \in \K^k$ and $\gamma \in \Gamma^l$. Then the assertion follows from Lemma~\ref{resg:decom}.
\end{proof}

Clearly for all $\bm U \coloneqq (U, f) \in \RV[k]$ there is a definable finite partition $\bm U_i \coloneqq (U_i, f_i)$ of $\bm U$ such that each $f_i$ is injective; in other words, $[\bm U] = \sum_i [\bm U_i]$ in $\gsk \RV[k]$. The forgetful map $\ob \RV[k] \fun \ob \RV_k$ given by $(U, f) \efun U$ induces a semigroup isomorphism $\gsk \RV[k] \simeq \gsk \RV_k$.

\begin{defn}[$\RES$-categories]\label{defn:RES:cat}
The category $\RES_k$ is the full subcategory of $\RV_k$ such that $U \in \RES_k$ if and only if $\vrv(U)$ is finite. Set $\RES_* = \bigcup_k \RES_k$.

The category $\RES[k]$ is the full subcategory of $\RV[k]$ such that $(U, f) \in \RES[k]$ if and only if $\vrv(U)$ is finite. Set $\RES[\leq k] = \coprod_{i \leq k} \RES[i]$ and $\RES[*] = \coprod_k \RES[k]$.
\end{defn}

Observe that the semiring $\gsk \RES_*$ is generated by isomorphism classes $[U]$ with $U$ a subset in $\K^+$. Since \cite[\S 8.2.11]{dries:1998} holds for any \omin-minimal expansion of an ordered field, we have the following explicit description of $\gsk \RES_*$: its underlying set is $(\{0\} \times \N) \cup (\N^+ \times \Z)$ and, for $(a, b), (c, d) \in \gsk \RES_*$,
\[
(a, b) + (c, d) = (\max\{a, c\}, b+d), \quad (a, b) \times (c, d) = (a + c, b \times d).
\]
By the computation in \cite{kage:fujita:2006}, the dimensional part is lost in the groupification of $\gsk \RES_*$, that is, $\ggk \RES_* = \Z$, which is of course much simpler than $\gsk \RES_*$. However, following the philosophy of \cite{hrushovski:kazhdan:integration:vf}, we shall work with Grothendieck semirings whenever possible.

By Lemma~\ref{gk:ortho}, if $(U, f) \in \RES[*]$ then $\vrv(f(U))$ is finite as well. Thus $\gsk \RES[*]$ is generated by isomorphism classes $[(U, f)]$ with $f$ a bijection between two subsets in $\K^+$. As above, each $\gsk \RES[k]$ may be described explicitly: $\gsk \RES[0]$ is canonically isomorphic to the semiring $\{(0, 0)\} \times \N$ and if $k > 0$ then its underlying set is $\bigcup_{0 \leq i \leq k}(\{(k, i)\} \times \Z)$ and its semigroup operation is given by
\[
(k, i, a) + (k, i', a') = (k, \max\{i, i'\}, a + a').
\]
Moreover, multiplication in $\gsk \RES[*]$ is given by
\[
(k, i, a) \times (l, j, b) = (k+l, i + j, a \times b).
\]

\begin{defn}[$\Gamma$-categories]\label{def:Ga:cat}
An object of the category $\Gamma[k]$ is simply a definable subset of $\Gamma^k$. Any definable bijection between two such objects is a \emph{morphism} of $\Gamma[k]$. The category $\Gamma^{c}[k]$ is the full subcategory of $\Gamma[k]$ such that $I \in \Gamma^{c}[k]$ if and only if $I$ is finite.

Set $\Gamma[\leq k] = \coprod_{i \leq k} \Gamma[i]$ and $\Gamma[*] = \coprod_k \Gamma[k]$; similarly for $\Gamma^c[\leq k]$ and $\Gamma^c[*]$.
\end{defn}

We clearly have $\gsk \Gamma^c[k] = \N$ for all $k$ and hence there is a canonical isomorphism $\gsk \Gamma^c[*] \simeq \N[X]$.

\begin{lem}\label{G:red}
For all $I \in \Gamma[k]$ there are finitely many definable subsets $H_i \sub \Gamma^{n_i}$ and tuples $\gamma_i \in \Gamma^{m_i}$ such that $\dim_{\Gamma}(H_i) = n_i$, $n_i + m_i = k$, and $[I] = \sum_i [H_i] \times [\{ \gamma_i \}]$ in $\gsk \Gamma[*]$.
\end{lem}
\begin{proof}
We do induction on $k$. The base case $k = 0$ is trivial. For the inductive step $k > 0$, the claim is also trivial if $\dim_{\Gamma}(I) = k$; so let us assume that $\dim_{\Gamma}(I) < k$ and, for simplicity, $I$ is a subset in $\Gamma^+$. We may partition $I$ into finitely many definable pieces $I_i$ such that each $I_i$ is contained in a hyperplane, that is, a subset defined by an equation of the form $\prod_{j=1}^k Z^{e_{ij}}_{ij} = \gamma_i$, where $\gamma_i$ is definable and the exponents $e_{ij} \in \Z$ are not all zero and are pairwise relatively prime. It is easy to see that for each $i$ we may find a matrix $M_i \in \mgl_k(\Z)$ such that the first row of $M_i$ is $(e_{i1}, \ldots, e_{ik})$. Since $M_i$ maps $I_i$ to a subset of the form $I'_i \times \{ \gamma_i \}$, where $I'_i \in \Gamma[k-1]$, and $[I_i] = [I'_i] \times [\{ \gamma_i \}]$ in $\gsk \Gamma[*]$, the claim simply follows from the inductive hypothesis.
\end{proof}

There is an obvious map $\ob \Gamma[*] \fun \ob \RV[*]$ given by $I \efun \bm I \coloneqq (\vrv^{-1}(I), \id)$, which is simply denoted by $\vrv^{-1}$. It follows from \cite[Theorem~A]{Dries:tcon:97} and \cite[Proposition~4.2.4]{dries:1998} (or Lemma~\ref{Gamma:lift} below) that this map $\vrv^{-1}$ induces a canonical injective homomorphism \[
\gsk \Gamma^{c}[*] \fun \gsk \RES[*]
\]
of graded semirings. Actually it also induces such a homomorphism $\gsk \Gamma[*] \fun \gsk \RV[*]$, which we shall establish later in this section.

Note that there is a similar semiring homomorphism $\gsk \Gamma^c[*] \fun \gsk \RES_*$, but it is not injective.

The following lemma is crucial for the tensor product expression we seek. It is easy to prove in algebraically closed valued fields or more generally in \cmin-minimal structures, but in the current context considerably more sophisticated argument is needed.

\begin{lem}\label{Gamma:lift}
Let $\bm D$ be a disjoint union of finitely many (not necessarily disjoint) definable subsets of $\Gamma^m$; similarly for $\bm E$. Set $A = \vrv^{-1}(\bm D)$ and $B = \vrv^{-1}(\bm E)$. Suppose that there is a definable bijection
\[
f : (\K^+)^{n} \times A \fun (\K^+)^{n} \times B.
\]
Then $f$ induces a definable bijection $e: \bm D \fun \bm E$ in the sense that if $e(\alpha) = \beta$ then there are $t \in (\K^+)^{n} \times \vrv^{-1}(\alpha)$ and $s \in (\K^+)^{n} \times \vrv^{-1}(\beta)$ such that $f(t) = s$.
\end{lem}
\begin{proof}
We do induction on $m$. The base case $m=0$ simply means that $\bm D$ and $\bm E$ are of the same finite size, which again follows from \cite[Theorem~A]{Dries:tcon:97} and \cite[Proposition~4.2.4]{dries:1998}. We proceed to the inductive step. We say that a tuple $\alpha = (\alpha_1, \ldots, \alpha_m) \in \bm D$ is \emph{reducible} if $\vrv^{-1}(\alpha)$ is $\alpha$-definably bijective to $\K^+ \times \vrv^{-1}(\wh \alpha_i)$ for some $i$; similarly for $\beta \in \bm E$.

Let $M \sub (\K^+)^n \times A$ be the union of the subsets $\{t\} \times \vrv^{-1}(\alpha)$ such that $f(\{t\} \times \vrv^{-1}(\alpha))$ is also of the form $\{s\} \times \vrv^{-1}(\beta)$. Let $\wh{\bm D} \sub \bm D$ consist of those $\alpha$ such that $\{t\} \times \vrv^{-1}(\alpha) \sub M$ for all $t \in (\K^+)^n$; the subset $\wh{\bm E} \sub \bm E$ is constructed similarly. Applying Lemma~\ref{resg:decom} to (the graph of) $f \rest M$, we obtain a definable finite partition
\[
f_i : U_i \times \vrv^{-1}(\bm D_i) \fun V_i \times \vrv^{-1}(\bm E_i)
\]
of $f \rest M$ such that each $f_i$ may be written in the form $(t, s) \efun (f'_i(t), f''_{i, t}(s))$, where
\begin{itemize}
  \item $f'_i : U_i \fun V_i$ is a definable bijection,
  \item $f''_{i, t} : \vrv^{-1}(\bm D_i) \fun \vrv^{-1}(\bm E_i)$ is a $t$-definable bijection that $\Gamma$-contracts to a definable bijection $\bm D_i \fun \bm E_i$.
\end{itemize}
In fact, by definable choice in the $\K$-sort, there is a definable $t \in U_i$ and hence each $f_i$ naturally yields a definable bijection between $\vrv^{-1}(\bm D_i)$ and $\vrv^{-1}(\bm E_i)$. We may assume that for no $i$, $j$ do we have
\[
f((\K^+)^n \times \vrv^{-1}(\bm D_i)) = (\K^+)^n \times \vrv^{-1}(\bm E_j),
\]
for otherwise they may simply be deleted. We may also assume that, for some $k > 1$,
\begin{itemize}
  \item $\bm D_1 = \ldots = \bm D_k \sub \wh{\bm D}$ and $\bm E_1, \ldots, \bm E_k \sub \wh{\bm E}$,
  \item $\bigcup_{1 \leq i \leq k} U_i = (\K^+)^n$.
\end{itemize}
The reason is as follows. If this does not happen to either the subsets $\bm D_i$ or the subsets $\bm E_j$ then, by definable choice in the $\K$-sort and Lemma~\ref{gam:tup:red}, all $\alpha \in \bm D$ and all $\beta \in \bf E$ are actually reducible. Hence, by compactness, upon further partitioning and re-indexing the coordinates, we may assume that $A$ and $B$ are respectively of the forms $(\K^+)^{n+1} \times \vrv^{-1}(\bm D')$ and $(\K^+)^{n+1} \times \vrv^{-1}(\bm E')$, where $\bm D'$ and $\bm E'$ are subsets of $\Gamma^{m-1}$. Then the inductive hypothesis may be applied.

For notational simplicity, let us further assume that $\bm E_1, \ldots, \bm E_k$ are pairwise distinct. By \cite[\S8.2.11]{dries:1998}, some $V_i$, say $V_1$, is a proper subset of $(\K^+)^n$. Without loss of generality, we may assume
\[
f^{-1}(((\K^+)^n \mi V_1) \times \vrv^{-1}(\bm E_1)) = U_p \times \vrv^{-1}(\bm D_p)
\]
for some $p > k$ and
\[
((\K^+)^n \mi U_p) \times \vrv^{-1}(\bm D_p) = U_q \times \vrv^{-1}(\bm D_q)
\]
for some $q > k$. Set $W = \biguplus_{1 \leq i \leq k} ((\K^+)^n \mi V_i)$ and suppose that $\dim_{\RV}(W) < n$. Let $K' \subsetneq K \sub V_1$ and $L' \subsetneq L \sub U_q$ be definable $n$-dimensional cells such that $K \mi K'$ and $L \mi L'$ are also $n$-dimensional. By \cite[\S8.2.11]{dries:1998} again, we have
\begin{itemize}
  \item $K'$ is definably bijective to $K$ and $L'$ is definably bijective to $L$,
  \item $K \mi K'$ is definably bijective to $L \mi L'$.
\end{itemize}
Therefore we may modify $f \rest M$ such that
\[
f((L \mi L') \times \vrv^{-1}(\bm D_q)) = (K \mi K') \times \vrv^{-1}(\bm E_1).
\]
Now, after such a modification of $f$ if necessary, we have
\[
\dim_{\RV} (W) = n \quad \text{and} \quad \chi(W) = (-1)^n(k-1).
\]
Applying \cite[\S8.2.11]{dries:1998} once more, we obtain a definable bijection $g$ from
\[
\bigcup_{1 \leq i \leq k} ((\K^+)^n \times \vrv^{-1}(\bm E_i))
\]
onto itself such that
\begin{itemize}
        \item $g(\bigcup_{1 \leq i \leq k} (V_i \times \vrv^{-1}(\bm E_i))) = (\K^+)^n \times \vrv^{-1}(\bm E_1)$,
        \item every $g(\{t\} \times \vrv^{-1}(\alpha))$ is of the form $\{s\} \times \vrv^{-1}(\beta)$.
\end{itemize}
Thus we have a commutative diagram
\[
\bfig
  \Vtriangle(0,0)/->`->`<-/<800,400>[M \mi ((\K^+)^n \times \vrv^{-1}(\bm D_1))`f(M) \mi ((\K^+)^n \times \vrv^{-1}(\bm E_1))`f(M) \mi \bigcup_{1 \leq i \leq k} (V_i \times \vrv^{-1}(\bm E_i)); f^*`f`g^*]
\efig
\]
where $g^*$ is the function naturally induced by $g$.

We may further partition the subsets $U_i$ for $i > 1$ if necessary, but leave the subsets $\bm D_i$ for $i > 1$ unchanged, so that the above procedure may be repeated with respect to $f^*$. Notice that the subsets $\bm E_i$ for $i > 1$ are also unchanged and the above procedure may be repeated with respect to $(f^*)^{-1}$ as well. Therefore, we will eventually reach a situation where the above procedure can no longer be carried out. As we have seen, this means that all $\alpha \in \bm D$ and all $\beta \in \bf E$ are reducible and the inductive hypothesis may be applied.

The additional requirement for the desired bijection is by extra bookkeeping. The details are left to the reader.
\end{proof}

\begin{rem}\label{rem:Ga:match:impr}
We have stated Lemma~\ref{Gamma:lift} with simplified assumptions so that the key points in the proof become more apparent. There are more general versions of this lemma. Here we describe one, which is only slightly more complicated but attains extra flexibility in application.

The first improvement is that, instead of $(\K^+)^{n}$, we may take any definable subset $W$ of $\K^{n}$ as the first factor in the two products, that is, the domain and the range of the bijection $f$, and the same proof works.

For the second improvement, we introduce the following terminology, which is based on the content of \cite[\S 8.2.11]{dries:1998}. We say that a subset $U \sub \RV^n$ is \emph{$\K$-homogeneous of $\K$-type $(k, l)$} if $\dim_{\RV}(U_{\gamma}) = k$ and $\chi(U_{\gamma}) = l$ for all $\gamma \in \vrv(U)$; two $\K$-homogeneous subsets are \emph{$\K$-isotypical} if they are of the same $\K$-type.

Let $\bm D$ (resp.\ $\bm E$) be a disjoint union of finitely many (not necessarily disjoint) definable subsets of $\Gamma^{m_1}$ (resp.\ $\Gamma^{m_2}$). Fix a pair of integers $k$ and $l$, where $0 \leq k \leq \min\{m_1, m_2\}$. Let $A$ be a definable subset of $W \times \vrv^{-1}(\bm D)$ such that, for all $t \in W$, $\vrv(A_t) = \bm D$ and $A_t$ is $\K$-homogeneous of $\K$-type $(k, l)$; similarly for $B \sub W \times \vrv^{-1}(\bm E)$. Note that $\vrv(W) \in \Gamma$ is written as $1$. Suppose that there is a definable bijection $f : A \fun B$. Then the conclusion of Lemma~\ref{Gamma:lift} holds. To see this, we recall \cite[\S\S 4.2.10-11]{dries:1998} and then observe that, by Lemma~\ref{gam:tup:red}, for each $ \gamma \in \bm D$ there is a (possibly empty) subtuple $ \gamma^* \sub  \gamma$ with $\lh( \gamma^*) \leq k$ such that
\begin{itemize}
  \item $\pr_{1,  \gamma^*}(A_{1,  \gamma}) = W \times \vrv^{-1}(\gamma^*)$, where $\pr_{1,  \gamma^*}$ is the obvious coordinate projection,
  \item there is a $\gamma^*$-definable function from $\pr_{1,  \gamma^*}(A_{1,  \gamma})$ into $\vrv^{-1}(\gamma \mi \gamma^*)$,
  \item for every $ t \in \pr_{1,  \gamma^*}(A_{1,  \gamma})$, $\fib(A_{ 1,  \gamma},  t)$ is of $\K$-type
  \[
  (k - \lh( \gamma^*), (-1)^{\lh( \gamma^*)} l);
  \]
\end{itemize}
similarly for each $\gamma \in \bm E$. By Lemma~\ref{gk:ortho}, the definable functions on $\bm D$ and $\bm E$ given by $\gamma \efun \gamma^*$ is finite to one; since $\bm D$ and $\bm E$ are already disjoint unions, upon further partitioning if necessary, we may assume that these functions are actually injective. Furthermore, by \cite[\S 8.2.11]{dries:1998}, we may adjust $A$, $B$ so that the lengths of the tuples of the form $\gamma^*$, $\gamma \in \bm D$ or $\gamma \in \bm E$, are all equal to the maximum length of such tuples before the adjustment. Now, each $A_{ 1,  \gamma}$ (resp.\ $B_{ 1,  \gamma}$) is $ \gamma$-definably bijective to a subset of the form $W \times Q \times \vrv^{-1}(\gamma^*)$, where
\[
Q \sub (\K^+ \cup \{0\})^{\lh(\gamma) - \lh(\gamma^*)}
\]
does not depend on $ \gamma$ and is of $\K$-type $(k - \lh( \gamma^*), (-1)^{\lh( \gamma^*)} l)$, and hence we are back in the situation that has been handled by the first improvement above.
\end{rem}

\begin{nota}\label{nota:RV:short}
Recall that $\Gamma$ is written multiplicatively and its ordering is inverse to that of $|\Gamma|$. We introduce the following shorthand for some distinguished elements in various Grothendieck semigroups and their groupifications (and closely related constructions):
\begin{gather*}
\bm 1_{\K} = [\{1\}] \in \gsk \RES[0], \quad [1] = [(\{1\}, \id)] \in \gsk \RES[1],\\
[\bm T] = [(\K^+, \id)] \in \gsk \RES[1], \quad [\bm A] = 2 [\bm T] +  [1] \in \gsk \RES[1],\\
\bm 1_{\Gamma} = [\Gamma^0] \in \gsk \Gamma[0], \quad [e] = [\{1\}] \in \gsk \Gamma[1], \quad [\bm H] = [(\infty, 1)] \in \gsk \Gamma[1],\\
[\bm J] = [(\RV^{\circ \circ} \mi \{\infty\}, \id)] -  [1] \in \ggk \RV[1].
\end{gather*}
As in~\cite{hrushovski:kazhdan:integration:vf}, the elements $[\bm J]$ and $\bm 1_{\K} + [\bm J]$ in $\ggk \RV[*]$ will play a special role in the discussion below (see Propositions~\ref{kernel:L:dag:coa} and \ref{kernel:volL:dag:coa} and the remarks thereafter).
\end{nota}

By Lemma~\ref{gam:pulback:mono}, the map from $\gsk \RES[*] \times \gsk \Gamma[*]$ to $\gsk \RV[*]$ naturally determined by the assignment
\[
([(U, f)], [I]) \efun [(U \times \vrv^{-1}(I), f \times \id)]
\]
is well-defined and is clearly $\gsk \Gamma^{c}[*]$-bilinear. Hence it induces a $\gsk \Gamma^{c}[*]$-linear map
\[
\bb D: \gsk \RES[*] \otimes_{\gsk \Gamma^{c}[*]} \gsk \Gamma[*] \fun \gsk \RV[*],
\]
which is a homomorphism of graded semirings. We shall abbreviate ``$\otimes_{\gsk \Gamma^{c}[*]}$'' as ``$\otimes$'' below. Note that, by the universal property, groupifying $\otimes_{\gsk \Gamma^{c}[*]}$ in the category of $\gsk \Gamma^{c}[*]$-semimodules is the same as taking the corresponding tensor product in the category of $\ggk \Gamma^{c}[*]$-modules.

\begin{prop}\label{red:D:iso}
$\bb D$ is an isomorphism of graded semirings.
\end{prop}
\begin{proof}
Surjectivity of $\bb D$ follows immediately from Lemma~\ref{RV:decom:RES:G}. For injectivity, let $\bm U_i \coloneqq (U_i, f_i)$, $\bm V_j \coloneqq (V_j, g_j)$ be objects in $\RES[*]$ and $I_i$, $J_j$ objects in $\Gamma[*]$ such that $\bb D([\bm U_i] \otimes [I_i])$, $\bb D([\bm V_j] \otimes [J_j])$ are objects in $\gsk \RV[l]$ for all $i$, $j$. Set
\begin{gather*}
M_i = U_i \times \vrv^{-1}(I_i), \quad N_i = V_j \times \vrv^{-1}(J_j),\\
M \coloneqq \biguplus_i M_i, \quad N \coloneqq \biguplus_j N_j.
\end{gather*}
Suppose that there is a definable bijection $f : M \fun N$. We need to show
\[
\sum_i [\bm U_i] \otimes [I_i] = \sum_j [\bm V_j] \otimes [J_j].
\]
By \omin-minimal cell decomposition and \cite[\S 8.2.11]{dries:1998}, without changing the sums, we may assume that each $U_i$ is a disjoint union of finitely many copies of $(\K^+)^i$ and thereby re-index $M_i$ more informatively as
\[
M_{i, m} \coloneqq U_i \times \vrv^{-1}(I_m),
\]
where $I_m$ is an object in $\Gamma[m]$; similarly each $N_j$ is re-indexed as $N_{j, n}$. The respective maximums of the numbers $i+m$, $j+n$ are the $\RV$-dimensions of $M$, $N$ and hence must be equal; it is denoted by $p$. Let $q$ be the largest $m$ such that $i + m = p$ for some $M_{i, m}$ and $q'$ the largest $n$ such that $j + n = p$ for some $N_{j, n}$. It is not hard to see that, by Lemmas~\ref{gk:ortho} and \ref{gam:red:K}, without changing the sums, we may arrange $q = q'$.

We now proceed by induction on the pair $(p, q)$ with respect to the lexicographic ordering. The base case $(0,0)$ is rather trivial. For the inductive step, let $M'_{p-q, q}$ be the union of the subsets of $M_{p-q, q}$ of the form $\{t\} \times \vrv^{-1}(\gamma)$ such that $f(\{t\} \times \vrv^{-1}(\gamma))$ is a subset of $N_{p-q, q}$ of this form too. For all $\gamma \in I_q$, if $U_{p-q} \times \vrv^{-1}(\gamma)$ is not contained in $M'_{p-q, q}$ then $\gamma$ is reducible (in the sense defined in the proof of Lemma~\ref{Gamma:lift}). On the other hand, applying Lemma~\ref{resg:decom} to (the graph of) $f \rest M_{p-q, q}$ and then following the main argument in the proof of Lemma~\ref{Gamma:lift}, we see that $M'_{p-q, q}$ and $f(M'_{p-q, q})$ may be deleted from $M_{p-q, q}$ and $N_{p-q, q}$, respectively, or they may be reduced (in the sense that has just been recalled). Therefore, by compactness, $q$ will decrease unless it is zero, in which case $p$ will decrease, and the inductive hypothesis may be applied.
\end{proof}

\begin{cor}\label{gsk:gamma:inj}
Every nonzero $[\bm U] \in \gsk \RES[*]$ induces an injective semigroup homomorphism:
\[
[\bm U] \otimes - : \gsk \Gamma[*] \fun \gsk \RES[*] \otimes \gsk \Gamma[*].
\]
In particular, there is a canonical injective homomorphism of graded semirings:
\[
\bb L_{\Gamma} : \gsk \Gamma[*] \fun \gsk \RV[*].
\]
\end{cor}
\begin{proof}
Suppose that $[\bm U] \otimes [I] = [\bm U] \otimes [J]$. By Proposition~\ref{red:D:iso}, $U \times \vrv^{-1}(I)$ is definably bijective to $U \times \vrv^{-1}(J)$. By Remark~\ref{rem:Ga:match:impr}, $[I] = [J]$. For the second assertion we may simply take $[\bm U] = \bm 1_{\K}$ in $[\bm U] \otimes -$ and then compose it with the isomorphism $\bb D$.
\end{proof}

Any definable subset of $\Gamma^n$ is definably bijective to a subset of  $|\Gamma|^{n+1}$ and hence we can associate two Euler characteristics $\chi_{\Gamma, g}$ and $\chi_{\Gamma, b}$ with the $\Gamma$-sort, which are induced by those on $|\Gamma|$ (see \cite{kage:fujita:2006} and also~\cite[\S 9]{hrushovski:kazhdan:integration:vf}). They are distinguished by
\[
\chi_{\Gamma, g}((\infty, 1)) = -1 \quad \text{and} \quad \chi_{\Gamma, b}((\infty, 1)) = 0.
\]

Similarly, there is an Euler characteristic $\chi_{\K}$ associated with the $\K$-sort (there is only one). Note that this is not available in algebraically closed valued fields.

We shall denote all of these Euler characteristics simply by $\chi$ if no confusion can arise. Using $\chi$ and the groupification of $\bb D$ (also denoted by $\bb D$), we can obtain various retractions from the Grothendieck ring $\ggk \RV[*]$ to (certain localizations of) the Grothendieck rings $\ggk \RES[*]$, $\ggk \Gamma[*]$.

\begin{lem}\label{gam:euler}
The Euler characteristics induce naturally three graded ring homomorphisms:
\[
\mdl E_{\K} : \ggk \RES[*] \fun \Z[X] \quad \text{and} \quad \mdl E_{\Gamma, g}, \mdl E_{\Gamma, b} : \ggk \Gamma[*] \fun \Z[X].
\]
\end{lem}
\begin{proof}
For $(U, f) \in \RES[k]$ and $I \in \Gamma[k]$, we simply set $\mdl E_{\K, k}([(U, f)]) = \chi(U)$ (see Remark~\ref{rem:RV:weako}) and $\mdl E_{\Gamma, k}([I]) = \chi(I)$. It is routine to check that these maps are well-defined and they induce graded ring homomorphisms $\mdl E_{\K} \coloneqq \sum_k \mdl E_{\K, k} X^k$ and $\mdl E_{\Gamma} \coloneqq \sum_k \mdl E_{\Gamma, k} X^k$ as desired.
\end{proof}

\begin{rem}\label{rem:poin}
Of course $\mdl E_{\K}$ is actually an isomorphism; on the other hand, $\ggk \Gamma^{c}[*]$ is canonically isomorphic to $\ggk \RES[*]$ and $\ggk \Gamma[*]$ is canonically isomorphic to the graded ring
\[
\Z[X, Y^{(2)}] \coloneqq \Z \oplus X\Z[X, Y]/ (Y^2+Y),
\]
where $Y$ represents the class $[\bm H]$ in $\ggk \Gamma[1]$. Thus there are two ``Euler characteristics'' $\chi_{\Gamma[*], g}$ and $\chi_{\Gamma[*], b}$ on $\ggk \Gamma[*]$ given by
\[
\Z[X, Y^{(2)}] \two^{Y \efun -1}_{Y \efun 0} \Z[X].
\]
By Proposition~\ref{red:D:iso} and Lemma~\ref{gam:euler}, there is a graded ring isomorphism
\[
\ggk \RV[*] \to^{\sim} \Z[X, Y^{(2)}] \quad \text{with} \quad \bm 1_{\K} + [\bm J] \efun 1 + 2YX + X.
\]
Setting
\[
 \Z^{(2)} = \Z[X, Y^{(2)}] / (1 + 2YX + X),
\]
we see that there is a canonical ring isomorphism
\[
\bb E_{\Gamma}: \ggk \RV[*] / (\bm 1_{\K} + [\bm J]) \to^{\sim} \Z^{(2)}.
\]
\end{rem}

\begin{prop}\label{prop:eu:retr:k}
There are two ring homomorphisms
\[
\bb E_{\K, g}: \ggk \RV[*] \fun \ggk \RES[*][[\bm A]^{-1}] \quad \text{and} \quad \bb E_{\K, b}: \ggk \RV[*] \fun \ggk \RES[*][[1]^{-1}]
\]
such that
\begin{itemize}
  \item their ranges are precisely the zeroth graded pieces of their respective codomains,
  \item $\bm 1_{\K} + [\bm J]$ vanishes under both of them,
  \item for all $x \in \ggk \RES[k]$, $\bb E_{\K, g} (x) = x [\bm A]^{-k}$ and $\bb E_{\K, b}(x) = x [1]^{-k}$.
\end{itemize}
\end{prop}
\begin{proof}
We first define a homomorphisms $\bb E_{g, n}: \ggk \RV[n] \fun \ggk \RES[n]$ for each $n$ as follows. By Proposition~\ref{red:D:iso}, there is an isomomorphism
\[
\bb D_n : \bigoplus_{i + j = n} \ggk \RES[i] \otimes \ggk \Gamma[j] \to^{\sim} \ggk \RV[n].
\]
By the universal property, there is a group homomorphism
\[
E_{g}^{i, j}: \ggk \RES[i] \otimes \ggk \Gamma[j] \fun \ggk \RES[i+j]
\]
given by
\[
x \otimes y \efun \mdl E_{g, j}(y) x [\bm T]^{j},
\]
where $\mdl E_{g, j}$ is defined with respect to $\chi_{\Gamma, g}$ as in Lemma~\ref{gam:euler}. Let $E_{g, n} = \sum_{i + j = n} E_{g}^{i, j}$ and then $\bb E_{g, n} = E_{g, n} \circ \bb D_n^{-1}$. It is straightforward to check the equality
\[
\bb E_{g, n}(x)\bb E_{g, m}(y) = \bb E_{g, n+m}(xy).
\]

The group homomorphisms $\tau_{m, k} : \ggk \RES[m] \fun \ggk \RES[m+k]$ given by $x \efun x [\bm A]^k$ determine a colimit system and the group homomorphisms
\[
\bb E_{g, \leq n} \coloneqq \sum_{m \leq n} \tau_{m, n-m} \circ \bb E_{g, m} : \ggk \RV[\leq n] \fun \ggk \RES[n]
\]
determine a homomorphism of colimit systems. Hence we have a ring homomorphism:
\[
\colim{n} \bb E_{g, \leq n} : \ggk \RV[*] \fun \colim{\tau_{n, k}} \ggk \RES[n].
\]
For all $n \geq 1$ we have
\[
\bb E_{g, \leq n}(\bm 1_{\K} + [\bm J]) = [\bm A]^n - 2[\bm T][\bm A]^{n-1} - [1] [\bm A]^{n-1} = 0.
\]
This yields the desired homomorphism $\bb E_{\K, g}$ since the colimit in question can be embedded into the zeroth graded piece of $\ggk \RES[*][[\bm A]^{-1}]$.

The construction of $\bb E_{\K, b}$ is completely analogous, where $[\bm A]$ is replaced by $[1]$ and $\chi_{\Gamma, g}$ by $\chi_{\Gamma, b}$.
\end{proof}

The homomorphisms $\bb E_{\K, g}$ and $\bb E_{\K, b}$ should be understood as ``Euler characteristics'' since the zeroth graded pieces of both $\ggk \RES[*][[\bm A]^{-1}]$ and $\ggk \RES[*][[1]^{-1}]$ are canonically isomorphic to $\Z$. They may also be obtained by specializing $\bb E_{\Gamma}$.

\begin{rem}
Alternatively, we may reformulate the $\Gamma$-categories in terms of the value group $|\Gamma_{\infty}|$ instead of the signed value group $\Gamma_{\infty}$ and denote them by $|\Gamma|[k]$, $|\Gamma|[*]$, etc. There is a natural injective homomorphism
\[
\ggk \Gamma[*] \fun \ggk |\Gamma|[*] \otimes \Z{[\tfrac{1}{2}]}
\]
induced by the assignment
\[
[\bm H] \efun \tfrac{1}{2} [(0, \infty)] \in \ggk |\Gamma|[1] \otimes \Z[\tfrac{1}{2}].
\]
This will have the effect of introducing the tensor factor $\otimes \Z[\frac{1}{2}]$ into the various constructions above, which is akin to the scenario described in \cite{Comte:fichou}.
\end{rem}

In the remainder of this section we shall introduce the simplest volume form, namely the constant $\Gamma$-volume form $1$, and reproduce some of the constructions above for the resulting categories; this will be used in the construction of topological zeta functions in \S\ref{sec:zeta}, where we shall work exclusively with power series fields over $\Q$.

Let $U \sub (\RV^{\times})^n \times \Gamma^m$, $V \sub (\RV^{\times})^{n'} \times \Gamma^{m'}$, and $C \sub U \times V$ be subsets (not necessarily definable). For every $((u, \alpha), (v, \beta)) \in C$, the \emph{$\Gamma$-Jacobian} of $C$ at $((u, \alpha), (v, \beta))$, written as $\jcb_{\Gamma} C((u, \alpha), (v, \beta))$, is the element
\[
\Pi (\vrv(v), \alpha) / \Pi (\vrv(u), \beta)  \in \Gamma,
\]
where $\Pi (\gamma_1, \ldots, \gamma_n) = \gamma_1 * \cdots * \gamma_n$; set
\[
\jcb_{|\Gamma|} C((u, \alpha), (v, \beta)) \coloneqq \Sigma (|\vrv|(v), |\alpha|) - \Sigma (|\vrv|(u), |\beta|),
\]
which is equal to $|\jcb_{\Gamma} C((u, \alpha), (v, \beta))| \in |\Gamma|$, where $\Sigma (\gamma_1, \ldots, \gamma_n) = \gamma_1 + \cdots + \gamma_n$.

\begin{defn}[$\vol \RV$- and $\vol \RES$-categories]\label{defn:RV:cat:vol}
The category $\vol \RV[k]$ has the same objects as $\RV[k]$. Let $(U, f)$, $(V, g)$ be objects in $\vol \RV[k]$ and $F : U \fun V$ an $\RV[k]$-morphism. Then $F$ is a \emph{morphism} of $\vol\RV[k]$ if and only if $\jcb_{\Gamma} F^{\rightleftharpoons}(u, v) = 1$ for all $(u, v) \in F^{\rightleftharpoons}$.

The category $\vol\RV^{\db}[k]$ is the full subcategory of $\vol\RV[k]$ such that $(U, f) \in \vol \RV^{\db}[k]$ if and only if (the graph of) $f$ is doubly bounded.

The category $\vol \RES[k]$ is the full subcategory of $\vol \RV[k]$ such that $(U, f) \in \vol \RES[k]$ if and only if $\vrv(U)$ is finite.

Set $\vol\RV[*] = \coprod_k \vol\RV[k]$; similarly for $\vol\RV^{\db}[*]$ and $\vol\RES[*]$.
\end{defn}

\begin{defn}[$\vol\Gamma$-categories]\label{def:Ga:cat:vol}
The category $\vol\Gamma[k]$ has the same objects as $\Gamma[k]$. Let $I$, $J$ be objects in $\vol\Gamma[k]$ and $F : I \fun J$ a $\Gamma[k]$-morphism. Then $F$ is a \emph{morphism} of $\vol\Gamma[k]$ if and only if $\jcb_{\Gamma} F(\alpha, \beta) = 1$ for all $(\alpha, \beta) \in F$.

The category $\vol\Gamma^{\db}[k]$ is the full subcategory of $\vol\Gamma[k]$ such that $I \in \vol\Gamma^{\db}[k]$ if and only if $I$ is doubly bounded.

The category $\vol\Gamma^{c}[k]$ is the full subcategory of $\vol\Gamma[k]$ such that $I \in \vol\Gamma^{c}[k]$ if and only if $I$ is finite.

Set $\vol\Gamma[*] = \coprod_k \vol\Gamma[k]$; similarly for $\vol\Gamma^{\db}[*]$ and $\vol\Gamma^c[*]$.
\end{defn}

Obviously we may identify $\gsk \vol \RES[*]$ as a sub-semiring of $\gsk \vol\RV^{\db}[*]$ and $\gsk \vol \Gamma^{c}[*]$ of $\gsk \vol\Gamma^{\db}[*]$. The identity map $\ob \vol \RES[*] \fun \ob \RES[*]$ induces a surjective semiring homomorphism
\[
\bb F : \gsk \vol \RES[*] \fun \gsk \RES[*];
\]
similarly for $\bb F : \gsk \vol \Gamma[*] \fun \gsk \Gamma[*]$. By Lemma~\ref{Gamma:lift}, the map $\vrv^{-1}$ induces a canonical injective homomorphism
\[
\gsk \vol \Gamma^{c}[*] \fun \gsk \vol \RES[*]
\]
of graded semirings. As above, it also induces a $\gsk \vol\Gamma^{c}[*]$-linear map
\[
\vol \bb D: \gsk \vol\RES[*] \otimes_{\gsk \vol\Gamma^{c}[*]} \gsk \vol\Gamma[*] \fun \gsk \vol\RV[*],
\]
which is a homomorphism of graded semirings. We shall again abbreviate ``$\otimes_{\gsk \vol\Gamma^{c}[*]}$'' as ``$\otimes$'' below.

We need an improved version of Lemma~\ref{RV:decom:RES:G}:

\begin{lem}\label{RV:decom:RES:vol}
Let $U \sub (\RV^{\times})^m$ be a definable subset and $E = \vrv(U)$. Then there are
\begin{itemize}
  \item a definable finite subset $E_0 \sub E$,
  \item finitely many definable products $V_i \times \vrv^{-1}(D_i) \sub \K^{k_i} \times \vrv^{-1}(\Gamma)^{l_i}$ with $k_i + l_i = m$,
  \item definable injections $F_i : V_i \times \vrv^{-1}(D_i) \fun U \mi \vrv^{-1}(E_0)$
\end{itemize}
such that
\begin{itemize}
  \item the subsets $U_i \coloneqq F_i(V_i \times \vrv^{-1}(D_i))$ form a partition of $U \mi \vrv^{-1}(E_0)$,
  \item $\jcb_{\Gamma} F_i(t, u) = 1$ for all $(t, u) \in F_i$.
\end{itemize}
Moreover, if $U$ is doubly bounded then every $D_i$ is doubly bounded as well.
\end{lem}
\begin{proof}
The case $m=1$ may be handled exactly as in the proof of Lemma~\ref{RV:decom:RES:G}. For the case $m>1$, by Lemmas~\ref{gk:ortho} and~\ref{gam:tup:red}, compactness, and a routine induction on $m$, we may delete a definable finite subset $E_0$ from $E$ and thereby assume that, for some $k \leq m$,
\begin{itemize}
  \item $\vrv^{-1}(\pr_{>k}(E)) = \pr_{>k}(U)$,
  \item $E$ is the graph of a definable function $\phi: \pr_{>k}(E) \fun \Gamma^k$,
  \item for all $\gamma \in \pr_{>k}(E)$ there is a $\gamma$-definable function $f_{\gamma} : \vrv^{-1}(\gamma) \fun \vrv^{-1}(\phi(\gamma))$.
\end{itemize}
For each $\gamma = (\gamma_{k+1}, \ldots, \gamma_{m}) \in \pr_{>k}(E)$ let $\gamma' = (\gamma_{k+1}, \ldots, \gamma_{m-1}, \gamma_{m} * \Pi \phi(\gamma))$. By Lemma~\ref{gam:pulback:mono}, we may also assume that there is a $\mgl_{m-k}(\Q)$-transformation between $\pr_{>k}(E)$ and the subset $D \coloneqq \{\gamma' : \gamma \in \pr_{>k}(E)\}$. Thus there is a definable bijection $F : U \fun U' \sub (\RV^{\times})^m$, where $\pr_{\leq k}(U') \sub \K^{k}$ and $\pr_{>k}(U') = \vrv^{-1}(D)$, such that $\jcb_{\Gamma} F(u, u') = 1$ for all $(u, u') \in F$. Now the first assertion follows from Lemma~\ref{resg:decom}.

The second assertion is clear.
\end{proof}

\begin{prop}\label{red:volD:iso}
$\vol \bb D$ is an isomorphism of graded semirings.
\end{prop}
\begin{proof}
Surjectivity of $\vol \bb D$ follows immediately from Lemma~\ref{RV:decom:RES:vol}. The proof of Proposition~\ref{red:D:iso} may be modified, as in the proof of Lemma~\ref{RV:decom:RES:vol}, to show injectivity of $\vol \bb D$. This is straightforward and is left to the reader. Note that here we need the extra clause in the statement of Lemma~\ref{Gamma:lift}.
\end{proof}

\begin{cor}\label{db:vold}
$\vol \bb D$ restricts to an isomorphism of graded semirings
\[
\vol \bb D^{\db}: \gsk \vol\RES[*] \otimes \gsk \vol\Gamma^{\db}[*] \fun \gsk \vol\RV^{\db}[*].
\]
\end{cor}

\begin{rem}\label{rem:volG:embed}
Due to the extra clause in the statement of Lemma~\ref{Gamma:lift}, Corollary~\ref{gsk:gamma:inj} still holds in the current context. There is an analogue of Proposition~\ref{prop:eu:retr:k} too, which will be established in \S\ref{sec:mor:vol}.
\end{rem}

\begin{nota}
The \emph{sign function} $\sgn : \Gamma^n \fun \{+, -\}$ is given by $\sgn(\gamma) = +$ if $\Pi \gamma \in \Gamma^+$ and $\sgn(\gamma) = -$ if $\Pi \gamma \in \Gamma^-$. Recall that $[(\vrv^{-1}(1), \id)]$ is abbreviated as $[\bm T]$, or more suggestively as $[+, \bm T]$. Note that $[(\vrv^{-1}(1), \id)]$ and $[(\vrv^{-1}(-1), \id)]$ are different elements in $\gsk \vol \RV[*]$; we shall denote the latter by $[-, \bm T]$ and write
\[
[\sgn(\alpha),\bm T][\sgn(\beta),\bm T] \eqqcolon [\sgn(\alpha)\sgn(\beta),\bm T^2].
\]
More generally, for any element $[+, \bm U] \coloneqq [\bm U]$ in $\gsk \vol \RV[*]$, there is a unique element $[-, \bm U]$ such that the equality \[
[+, \bm U][-, \bm T] = [-, \bm U][+, \bm T]
\]
is witnessed by the morphism $(\id, -\id)$.
\end{nota}

Let $I!$ be the semiring congruence relation on $\gsk \vol \RES[*]$ generated by the pairs
\[
([(\vrv^{-1}(\gamma), \id)], [\bm T]), \quad \gamma \in \Gamma^{+}(\mdl S).
\]
We set
\[
\sgsk \vol \RES[*] = \gsk \vol \RES[*] / I!.
\]
Let $\RES^{\pm}[k]$ be the category such that it has the same objects as $\RES[k]$ and its morphisms are the $\RES[k]$-morphisms $F$ with
\[
\sgn (\jcb_{\Gamma} F^{\rightleftharpoons} (u, v)) = 1
\]
for all $(u, v) \in F^{\rightleftharpoons}$. Set $\RES^{\pm}[*] = \coprod_k \RES^{\pm}[k]$. It is easy to see that $\sgsk \vol \RES[*]$ is naturally isomorphic to $\gsk \RES^{\pm}[*]$ and there is a commutative diagram
\[
\bfig
  \Square(0,0)|aaaa|/->`->`<-`->/<400>[{\gsk \vol \RES[*]}`{\gsk \RES[*]}`{\sgsk \vol \RES[*]}`{\gsk \RES^{\pm}[*]}; \bb F``\bb F`\sim]
\efig
\]
The semiring $\sgsk \vol \RES[*]$ will be instrumental in \S\ref{sec:zeta}.

\section{Grothendieck homomorphisms}\label{section:gromor}

We shall establish canonical homomorphisms between the Grothendieck semirings of the $\VF$-categories and the $\RV$-categories in this section. We shall recall much of the formalism and quote quite a few proofs in the relevant sections of \cite{Yin:special:trans, Yin:int:acvf, Yin:int:expan:acvf}.

\subsection{Special bijections}
From now on the underlying substructure $\mdl S$ is assumed to be $\VF$-generated. Also note that Convention~\ref{conv:can} will frequently be in effect.

\begin{defn}\label{defn:special:bijection}
Let $A$ be a (regularized) definable subset whose first coordinate is a $\VF$-coordinate. Let $C \sub \RVH(A)$
be an $\RV$-pullback and $\lambda: \pr_{>1}(C \cap A) \fun \VF$ a definable function such that the graph of $\lambda$ is contained in $C$. Let
\[
C^{\sharp} = \bigcup_{x \in \pr_{>1} (C)} \MM_{|\vrv|(\prv_1(x))} \times \{x\} \quad \text{and} \quad \RVH(A)^{\sharp} = C^{\sharp} \uplus (\RVH(A) \mi C).
\]
The \emph{centripetal transformation $\eta : A \fun \RVH(A)^{\sharp}$ with respect to $\lambda$} is defined by
\[
\begin{cases}
  \eta (a, x) = (a - \lambda(x), x), & \text{on } C \cap A,\\
  \eta = \id, & \text{on } A \mi C.
\end{cases}
\]
Note that $\eta$ is injective. The inverse of $\eta$ is naturally called the \emph{centrifugal transformation with respect to $\lambda$}. The
function $\lambda$ is referred to as the \emph{focus} of $\eta$ and the $\RV$-pullback $C$ as the \emph{locus} of $\lambda$ (or $\eta$).

A \emph{special bijection} $T$ on $A$ is an alternating composition of centripetal transformations and regularizations. The \emph{length} of such a special bijection $T$, denoted by $\lh(T)$, is the number of centripetal transformations in $T$. The range of $T$ is sometimes denoted by $A^{\sharp}$.
\end{defn}

For functions between subsets that have only one $\VF$-coordinate, composing with special bijections on the right and inverses of special bijections on the left obviously preserves otop.

\begin{lem}\label{inverse:special:dim:1}
Let $T$ be a special bijection on $A \sub \VF \times \RV^m$ such that $A^{\sharp}$ is an $\RV$-pullback. Then there is a definable function $\epsilon : \prv (A^{\sharp}) \fun \VF$ such that, for every $\RV$-polydisc $\gp = \rv^{-1}(t) \times \set{s} \sub A^{\sharp}$,
\[
(\pvf \circ T^{-1})(\gp) = \rv^{-1}(t) + \epsilon(s).
\]
\end{lem}
\begin{proof}
By induction on the length $\lh(T)$ of $T$, this is reduced to the case $\lh(T) = 1$, which is clear from Definition~\ref{defn:special:bijection}. (See \cite[Lemma~4.1]{Yin:int:acvf} for more details.)
\end{proof}

Note that, in the above lemma, since $\dom(\epsilon) \sub \RV^l$ for some $l$, by
Corollary~\ref{function:rv:to:vf:finite:image}, $\ran(\epsilon)$ is actually finite.

A definable subset $A$ is called a \emph{deformed $\RV$-pullback} if there is a special bijection $T$ on $A$ such that $A^{\sharp}$ is an $\RV$-pullback.

\begin{lem}\label{simplex:with:hole:rvproduct}
Every definable subset $A \sub \VF \times \RV^m$ is a deformed $\RV$-pullback.
\end{lem}
\begin{proof}
By compactness and HNF this is immediately reduced to the situation where $A \sub \VF$ is contained in an $\RV$-disc and is a $\vv$-interval with end-discs $\ga$, $\gb$. This may be further divided into several cases according to whether $\ga$, $\gb$ are open or closed discs and whether the ends of $A$ are open or closed. In each of these cases Lemma~\ref{clo:disc:bary} is applied in much the same way as its counterpart is applied in the proof of \cite[Lemma~4.26]{Yin:QE:ACVF:min}. It is a tedious exercise and is left to the reader.
\end{proof}

Here is an analogue of \cite[Theorem~5.4]{Yin:special:trans} (see also \cite[Theorem~4.25]{Yin:int:expan:acvf}):

\begin{thm}\label{special:term:constant:disc}
Let $F(X) = F(X_1, \ldots, X_n)$ be an $\lan{T}{}{}$-term. Let $u \in \RV^n$ and $R : \rv^{-1}(u) \fun A$ be a special bijection. Set $f = F \circ R^{-1}$. Then there is a special bijection $T$ on $A$ such that $f \circ T^{-1}$ is contractible.
\end{thm}
\begin{proof}
First observe that if the assertion holds for one $\lan{T}{}{}$-term then it holds simultaneously for any finite number of $\lan{T}{}{}$-terms. We do induction on $n$. For the base case $n=1$, by Corollary~\ref{part:rv:cons}, there is a definable finite partition $B_1, \ldots, B_n$ of $\rv^{-1}(u)$ such that, for all $i$, if $\ga \sub B_i$ is an open disc then $\rv \rest F(\ga)$ is constant. By Lemma~\ref{simplex:with:hole:rvproduct}, there is a special bijection $T$ on $A$ such that each $(T \circ R) (B_i)$ is an $\RV$-pullback. Clearly $T$ is as required.

For the inductive step, we may and shall adopt the inductive step in the proof of \cite[Theorem~5.4]{Yin:special:trans} (or \cite[Theorem~4.25]{Yin:int:expan:acvf}), since the construction there only depends on its base case formally.

We may concentrate on a single $\RV$-polydisc $\gp = \rv^{-1}(v) \times \set{(v, r)} \sub A$. Let $\phi(X, Y)$ be a quantifier-free formula that defines the function $(\rv \circ f) \rest \gp$. Let $G_{i}(X)$ enumerate the occurring $\lan{T}{}{}$-terms of $\phi$. For each $a \in \rv^{-1}(v_1)$ let $G_{i,a} = G_{i}(a, X_2, \ldots, X_n)$. By the inductive hypothesis, there is a special bijection $R_{a}$ on $\rv^{-1}(v_2, \ldots, v_n)$ such that every $G_{i,a} \circ R_a^{-1}$ is contractible. Let $U_{k, a}$ enumerate the loci of the components of $R_{a}$ and $\lambda_{k, a}$ the corresponding focus maps. By compactness,
\begin{itemize}
  \item for each $i$ there is a quantifier-free formula $\psi_i$ such that $\psi_i(a)$ defines the contraction of $G_{i,a} \circ R_a^{-1}$,
  \item there is a quantifier-free formula $\theta$ such that $\theta(a)$ determines the sequence $\rv(U_{k, a})$ and the $\VF$-coordinates targeted by $\lambda_{k, a}$.
\end{itemize}
Let $H_{j}(X_1)$ enumerate the occurring $\lan{T}{}{}$-terms of the formulas $\psi_i$, $\theta$. Applying the inductive hypothesis again, we obtain a special bijection $T_1$ on $\rv^{-1}(v_1)$ such that every $H_{j} \circ T_1^{-1}$ is contractible. This means that, for every $\RV$-polydisc $\gq \sub T_1(\rv^{-1}(v_1))$ and all $a_1, a_2 \in T_1^{-1}(\gq)$,
\begin{itemize}
  \item the formulas $\psi_i(a_1)$, $\psi_i(a_2)$ define the same contraction,
  \item the special bijections $R_{a_1}$, $R_{a_2}$ may be naturally glued together to form one special bijection on $\{a_1, a_2\} \times \rv^{-1}(v_2, \ldots, v_n)$.
\end{itemize}
Consequently, $T_1$ and $R_{a}$ naturally induce a special bijection $T$ on $\gp$ such that every $G_{i} \circ T^{-1}$ is contractible. This implies that $f \circ T^{-1}$ is contractible and hence $T$ is as required.
\end{proof}

The proof of \cite[Theorem~4.26]{Yin:int:expan:acvf} works for the following corollary:

\begin{cor}\label{special:bi:term:constant}
Let $A \sub \VF^n$ and $f : A \fun \RV^m$ be a definable function. Then there is a special bijection $T$ on $A$ such that $T(A)$ is an $\RV$-pullback and the function $f \circ T^{-1}$ is contractible.
\end{cor}

\begin{cor}\label{all:subsets:rvproduct}
Every definable subset is a deformed $\RV$-pullback.
\end{cor}

Recall that the \emph{$\RV$-fiber dimension} of a definable subset $A$, denoted by $\dim_{\RV}^{\fib}(A)$, is the number
\[
\max \{\dim_{\RV}(A_a) : a \in \pvf(A)\}.
\]

\begin{defn}[$\VF$-categories]\label{defn:VF:cat}
The objects of the category $\VF[k]$ are the definable subsets of $\VF$-dimension $\leq k$ and $\RV$-fiber dimension $0$ (that is, all the $\RV$-fibers are finite). Any definable bijection between two such objects is a morphism of $\VF[k]$. Set $\VF_* = \bigcup_k \VF[k]$.
\end{defn}

\begin{defn}\label{def:L}
The \emph{$k$th canonical $\RV$-lifting map} $\mathbb{L}_k: \ob \RV[k] \fun \ob \VF[k]$ is given by
\[
\mathbb{L}_k(U,f) = \bigcup \{\rv^{-1}(f(u)) \times \{ u \}: u \in U\}.
\]
The lifting map $\mathbb{L}_{\leq k}: \ob \RV[\leq k] \fun \ob \VF[k]$ is given by
\[
\bigoplus_{i \leq k} \bm U_i \efun \biguplus_{i \leq k} \bb L_i(\bm U_i).
\]
Set $\mathbb{L} = \bigcup_k \mathbb{L}_{\leq k}$.
\end{defn}

\begin{defn}
Let $F$ be an $\RV[k]$-morphism and $F_i : U_i \fun V_i$ a definable finite partition of the induced finite-to-finite correspondence $F^{\rightleftharpoons}$ such that each $F_i$ is a bijection. Suppose that
\[
F_i^{\uparrow} : \rv^{-1}(U_i) \fun \rv^{-1}(V_i)
\]
is a definable bijection that contracts to $F_i$. Then $F^{\uparrow} \coloneqq \bigcup_i F_i^{\uparrow}$ is called a \emph{lift} of $F$. We shall regard $F^{\uparrow}$ as a definable bijection $\bb L (\bm U) \fun \bb L (\bm V)$ that ``contracts'' to $F^{\rightleftharpoons}$.
\end{defn}

\begin{lem}\label{RV:lift}
For every $\RV[k]$-morphism $F : (U, f) \fun (V, g)$ there is a $\VF[k]$-morphism $F^{\uparrow}$ that lifts $F$.
\end{lem}
\begin{proof}
Let $U^* = f(U)$ and $V^* = g(V)$. By compactness, we may assume that $U^*$, $V^*$ are subsets of $(\K^+)^k$ and $F^{\rightleftharpoons} : U^* \fun V^*$ is simply a definable bijection. Then, by \cite[Theorem~A]{Dries:tcon:97}, $F^{\rightleftharpoons}$ is defined by an $\lan{T}{}{}$-formula $\phi$, possibly with parameters in $\K(\mdl S)$. By Theorem~\ref{tcon:qe}, $F^{\rightleftharpoons}$ is piecewise given by $\lan{T}{}{}$-terms. Since all such terms define total continuous functions, by \cite[Proposition~2.20]{DriesLew95} and Lemma~\ref{RV:bou:dim}, there is an $\lan{T}{}{}$-definable subset $W \sub U^*$ such that
\begin{itemize}
  \item $U^* \mi W$ is open and $\dim_{\RV}(W) < k$,
  \item the bijection defined by $\phi$ in the $\lan{T}{}{}$-reduct of $\gC$ contains a lift of $F^{\rightleftharpoons} \rest (U^* \mi W)$.
\end{itemize}
By \omin-minimality, there are an $\lan{T}{}{}$-definable finite partition $W_i$ of $W$ and injective coordinate projections $\pi_i : W_i \fun (\K^+)^{k_i}$ such that $\dim_{\RV}(\pi_i(W_i)) = k_i$. If $k_i = 0$ then $W_i$ is a singleton and the $\RV$-polydisc $\rv^{-1}(W_i)$ contains a definable point, in which case lifting $F^{\rightleftharpoons} \rest W_i$ is easy. Therefore, by compactness, a routine induction on $n = \dim_{\RV}(U^*)$ proves the assertion.
\end{proof}

\begin{cor}\label{L:sur:c}
The lifting map $\bb L_{\leq k}$ induces a surjective homomorphism, which is simply denoted by $\bb L$, between the Grothendieck semigroups
\[
\gsk \RV[\leq k] \fun \gsk \VF[k].
\]
\end{cor}
\begin{proof}
Using Corollary~\ref{all:subsets:rvproduct} and Lemma~\ref{RV:lift} instead of \cite[Corollary~5.6, Theorem~7.6]{Yin:special:trans}, the proof of \cite[Corollary~7.7]{Yin:special:trans} works.
\end{proof}

\begin{lem}\label{simul:special:dim:1}
Let $f : A \fun B$ be a definable bijection, where $A$, $B$ have exactly one $\VF$-coordinate each. Then there exist special bijections $T_A : A \fun A^{\sharp}$ and $T_B : B \fun B^{\sharp}$ such that $A^{\sharp}$, $B^{\sharp}$ are $\RV$-pullbacks and, in
the commutative diagram
\[
\bfig
  \square(0,0)/->`->`->`->/<600,400>[A`A^{\sharp}`B`B^{\sharp};
  T_A`f``T_B]
 \square(600,0)/->`->`->`->/<600,400>[A^{\sharp}`\rv(A^{\sharp})`B^{\sharp} `\rv(B^{\sharp});  \rv`f^{\sharp}`f^{\sharp}_{\downarrow}`\rv]
 \efig
\]
$f^{\sharp}_{\downarrow}$ is bijective and hence $f^{\sharp}$ is a lift of it.
\end{lem}
\begin{proof}
By Corollaries~\ref{special:bi:term:constant}, \ref{all:subsets:rvproduct}, and Lemma~\ref{open:pro}, we may assume that $A$, $B$ are $\RV$-pullbacks, $f$ is contractible and has otop, and there is a special bijection $T_B: B \fun B^{\sharp}$ such that $(T_B \circ f)^{-1}$ is contractible. Let $T_B = T_{B, n} \circ \ldots \circ T_{B, 1}$. It is enough to construct a special bijection $T_A = T_{A, n} \circ \ldots \circ T_{A, 1}$ on $A$ such that, for each $i$, both $\wt{T}_{B, i} \circ f \circ (\wt{T}_{A, i})^{-1}$ and $\wt{T}_{A, i} \circ (T_B \circ f)^{-1}$ are contractible, where
\[
\wt{T}_{B, i} = T_{B, i} \circ \ldots \circ T_{B, 1} \quad \text{and} \quad \wt{T}_{A, i} = T_{A, i} \circ \ldots \circ T_{A, 1}.
\]
Now we may simply use the construction in the proof of \cite[Lemma~5.2]{Yin:int:acvf}, since it only depends on $f$ being contractible and having otop.
\end{proof}

\begin{defn}
Let $A \sub \VF^{n} \times \RV^{m_1}$ and $B \sub \VF^{n} \times \RV^{m_2}$ and $f : A \fun B$ be a bijection. We say that $f$ is \emph{relatively unary} if there is an $i \in [n]$ such that $(\pr_{\widetilde{i}} \circ f)(x) = \pr_{\widetilde{i}}(x)$ for all $x \in A$. In this case we say that $f$ is \emph{unary relative to the $i$th $\VF$-coordinate}. If $f \rest A_a$ is also a special bijection for every $a \in \pr_{\widetilde{i}} (A)$ then we say that $f$ is \emph{special relative to the $i$th $\VF$-coordinate}.
\end{defn}

Obviously the inverse of a relatively unary bijection is a relatively unary bijection.

Let $A \sub \VF^n \times \RV^m$ be a definable subset, $C \sub \RVH(A)$ an $\RV$-pullback, $\lambda$ a focus map with respect to $C$ (and the first $\VF$-coordinate), and $\eta$ the centripetal transformation with respect to $\lambda$. Clearly $\eta$ is unary relative to the first $\VF$-coordinate. It follows that every special bijection $T$ on $A$ is a composition of relatively special bijections. Choose an $i \in [n]$. By Corollary~\ref{all:subsets:rvproduct} and compactness, there is a bijection $T_i$ on $A$, special relative to the $i$th $\VF$-coordinate, such that $T_i(A_a)$ is an $\RV$-pullback for every $a \in \pr_{\widetilde i}(A)$. Let
\[
A_i = \bigcup_{a \in \pr_{\widetilde i}(A)} \{a\} \times (\prv \circ T_i)(A_a) \sub \VF^{n-1} \times \RV^{m_i}.
\]
We write $\wh T_i : A \fun A_i$ for the function naturally induced by $T_i$. For any $j \in [n-1]$, we may repeat the above procedure on $A_i$ with respect to the $j$th $\VF$-coordinate and thereby obtain a subset $A_{j} \sub \VF^{n-2} \times \RV^{m_j}$ and a function $\wh T_{j} : A_i \fun A_{j}$. Continuing thus, we see that, for any permutation $\sigma$ of $[n]$, we can construct a (not necessarily unique) sequence of relatively special bijections $T_{\sigma(1)}, \ldots, T_{\sigma(n)}$ and a corresponding function $\wh T_{\sigma} : A \fun \RV^{l}$. Note that $(\wh T_{\sigma}(A), \pr_{\leq n})$ is an object in $\RV[\leq n]$, where we always shuffle the relevant $\RV$-coordinates in $\wh T_{\sigma}(A)$ to the first $n$ positions. We also have the natural bijection determined by $\wh T_{\sigma}$:
\[
T_{\sigma} : A \fun \mathbb{L}(\wh T_{\sigma}(A), \pr_{\leq n}),
\]
which is a special bijection and may be thought of as the ``composition'' $T_{\sigma(n)} \circ \ldots \circ T_{\sigma(1)}$.

\begin{defn}\label{defn:standard:contraction}
The function $\wh T_{\sigma}$ (or the image $\wh T_{\sigma}(A)$) is called a \emph{standard contraction} of $A$.
\end{defn}

\begin{rem}\label{special:dim:1:RV:iso}
In Lemma~\ref{simul:special:dim:1}, $\rv(A^{\sharp})$ and $\rv(B^{\sharp})$ are respectively standard contractions of $A$ and $B$. Actually $(\rv(A^{\sharp}), \pr_1)$ and $(\rv(B^{\sharp}), \pr_1)$ are $\RV[\leq 1]$-isomorphic. To see this, let $(a, t) \in A^{\sharp}$ and $f^{\sharp}(a, t) = (b, s)$; then, by Lemma~\ref{RV:no:point}, $b$ is $a$-definable and hence $\rv(b)$ is $\rv(a)$-definable; similarly for the other direction.
\end{rem}

\begin{lem}\label{bijection:partitioned:unary}
Let $A \sub \VF^{n} \times \RV^{m_1}$, $B \sub \VF^{n} \times \RV^{m_2}$, and $f : A \fun B$ be a definable bijection. Then there is a partition of $A$ into definable subsets $A_i$ such that each $f \rest A_i$ is a composition of definable relatively unary bijections.
\end{lem}
\begin{proof}
The proof of \cite[Lemma~5.6]{Yin:int:acvf} may be reproduced here with virtually no changes.
\end{proof}

\subsection{Additive invariants}
We begin by a discussion of the issue of $2$-cells, as in \cite[\S4]{Yin:int:acvf}.

\begin{lem}\label{bijection:dim:1:decom:RV}
Let $f : A \fun B$ be a definable bijection between two subsets of $\VF$. Then there is a special bijection $T$ on $A$ such that $T(A)$ is an $\RV$-pullback and, for each $\RV$-polydisc $\gp \sub T(A)$, $f \rest T^{-1}(\gp)$ is $\rv$-linear.
\end{lem}
\begin{proof}
By Lemma~\ref{rv:lin} and compactness, for all but finitely many $a \in A$ there is an $a$-definable $\delta_a \in |\Gamma|$ such that $f \rest \go(a, \delta_a)$ is $\rv$-linear. Without loss of generality, we may assume that, for all $a \in A$, $\delta_a$ exists and is the least element that satisfies this condition. Let $g : A \fun |\Gamma|$ be the definable function given by $a \efun \delta_a$. By Corollary~\ref{special:bi:term:constant}, there is a special bijection $T$ on $A$ such that $T(A)$ is an $\RV$-pullback and, for all $\RV$-polydisc $\gp \sub T(A)$, $(g \circ T^{-1}) \rest \gp$ is constant. So $T$ is as required.
\end{proof}

\begin{lem}\label{bijection:rv:one:one}
Let $A \sub \VF^2$ be a definable subset such that $\ga_1 \coloneqq \pr_1(A)$ and $\ga_2 \coloneqq \pr_2(A)$ are open discs. Let $f : \ga_1 \fun \ga_2$ be a definable bijection that has otop. Suppose that for each $a \in \ga$ there is a $t_a \in \RV$ such that $A_a = \rv^{-1}(t_a) + f(a)$. Then there is a special bijection $T$ on $\ga_1$ such that $T(\ga_1)$ is an $\RV$-pullback and, for each $\RV$-polydisc $\gp \sub T(\ga_1)$, the subset
\[
\{\rv(a - f^{-1}(b)) : a \in T^{-1}(\gp) \text{ and } b \in A_a \}
\]
is a singleton.
\end{lem}
\begin{proof}
For each $a \in \ga_1$, let $\gb_a$ be the smallest closed disc that contains $A_a$. Since $A_a - f(a) = \rv^{-1}(t_a)$, we have $f(a) \in \gb_a$ but $f(a) \notin A_a$ if $t_a \neq \infty$. Hence $a \notin f^{-1}(A_a)$ if $t_a \neq \infty$ and $\{a\} = f^{-1}(A_a)$ if $t_a = \infty$. Since $f^{-1}(A_a)$ is a disc, in either case, the function $\rv(a - X)$ is constant on $f^{-1}(A_a)$. Since the function $h : \ga_1 \fun \RV$ given by $a \efun \rv(a - f^{-1}(A_a))$ is definable, we may apply Corollary~\ref{special:bi:term:constant} as in the proof of Lemma~\ref{bijection:dim:1:decom:RV}.
\end{proof}

\begin{defn}\label{defn:balance}
Let $A$, $\ga_1$, $\ga_2$, and $f$ be as in Lemma~\ref{bijection:rv:one:one}. Suppose that $f$ is actually $\rv$-linear. We say that $f$ is \emph{balanced in $A$} if there are $t_1$, $t_2$ in $\RV$, called the \emph{paradigms} of $f$, such that, for every $a \in \ga_1$,
\[
 A_a = \rv^{-1}(t_2) + f(a) \quad \text{and} \quad f^{-1}(A_a) = a - \rv^{-1}(t_1).
\]
\end{defn}

If one of the paradigms is $\infty$ then the other one must be $\infty$. In this case $A$ is the (graph of the) bijection $f$. Assume $t_1, t_2 \in \RV^{\times}$. Let $\gB_1$, $\gB_2$ be respectively the sets of open subdiscs of $\ga_1$, $\ga_2$ of valuative radii $|\vrv(t_1)|$, $|\vrv(t_2)|$. Then, for every $\gb_1 \in \gB_1$, every $a_1 \in \gb_1$, and every $a_2 \in A_{a_1} \eqqcolon \gb_2$, we have
\[
\gb_2 = \rv^{-1}(t_2) + f(\gb_1) \in \gB_2 \quad \text{and} \quad A_{a_2} = f^{-1}(\gb_2) + \rv^{-1}(t_1) = \gb_1.
\]
This internal symmetry of $A$ is illustrated by the following diagram:
\[
\bfig
  \dtriangle(0,0)|amb|/.``<-/<800,250>[\gb_1`f^{-1}(\gb_2)`\gb_2; \pm \rv^{-1}(t_1)`\times`f^{-1}]
  \ptriangle(800,0)|amb|/->``./<800,250>[\gb_1`f(\gb_1)`\gb_2; f`` \pm \rv^{-1}(t_2)]
 \efig
\]

\begin{defn}\label{def:units}
We say that a subset $A$ is a \emph{$1$-cell} if it is either an open disc contained in a single $\RV$-disc or a point in $\VF$. We say
that $A$ is a \emph{$2$-cell} if
\begin{itemize}
 \item $A$ is a subset of $\VF^2$ contained in a single $\RV$-polydisc and $\pr_1(A)$ is a $1$-cell,
 \item there is a function $\epsilon : \pr_1 (A) \fun \VF$ and a $t \in \RV$ such that, for every $a \in \pr_1(A)$, $A_a =
   \rv^{-1}(t) + \epsilon(a)$,
 \item one of the following three possibilities occurs:
  \begin{itemize}
   \item $\epsilon$ is constant,
   \item $\epsilon$ is injective, has otop, and $\rad(\epsilon(\pr_1(A))) \geq |\vrv(t)|$,
   \item $\epsilon$ is balanced in $A$.
  \end{itemize}
\end{itemize}
The function $\epsilon$ is called the \emph{positioning function} of $A$ and the element $t$ is called the \emph{paradigm} of $A$.

More generally, a subset $A$ with exactly one $\VF$-coordinate is a \emph{$1$-cell} if, for each $t \in \pr_{>1}(A)$, $A_t$ is a $1$-cell in the sense specified above; the parameterized version of the notion of a $2$-cell is formulated in the same way.
\end{defn}

A $2$-cell is definable if all the relevant ingredients are definable. Naturally we shall only be concerned with definable $2$-cells. Notice that Corollary~\ref{all:subsets:rvproduct} implies that for every definable subset $A$ with exactly one $\VF$-coordinate there is a
definable function $\pi: A \fun \RV^l$ such that each fiber $\pi^{-1}(s)$ is a $1$-cell. This should be understood as $1$-cell decomposition and the next lemma $2$-cell decomposition.

\begin{lem}\label{decom:into:2:units}
For every definable subset $A \sub \VF^2$ there is a definable function $\pi: A \fun \RV^m$ such that each fiber $\pi^{-1}(s)$ is a $2$-cell.
\end{lem}
\begin{proof}
The proof of \cite[Lemma~4.8]{Yin:int:acvf} still works if we replace the supporting results \cite[Corollary~2.25, Lemmas~2.7, 4.1, 4.2, 4.4]{Yin:int:acvf} there with, respectively, Corollaries~\ref{all:subsets:rvproduct}, \ref{uni:fun:decom}, Lemmas~\ref{inverse:special:dim:1}, \ref{bijection:dim:1:decom:RV}, \ref{bijection:rv:one:one} above.
\end{proof}

For the next two lemmas, let $12$ and $21$ denote the permutations of $\{1, 2\}$; for their proofs, see \cite[Lemma~5.7, Corollary~5.8]{Yin:int:acvf}.
\begin{lem}\label{2:unit:contracted}
Let $A \sub \VF^2$ be a definable $2$-cell. Then there are standard contractions $\wh T_{12}$ and $\wh R_{21}$ of $A$ such that $(\wh T_{12}(A), \pr_{\leq 2})$ and $(\wh R_{21}(A), \pr_{\leq 2})$ are $\RV[\leq 2]$-isomorphic.
\end{lem}

\begin{lem}\label{subset:partitioned:2:unit:contracted}
Let $A \sub \VF^2 \times \RV^m$ be a definable subset. Then there is a definable injection $f : A \fun \VF^2 \times \RV^l$ such that
\begin{itemize}
  \item $f$ is unary relative to both coordinates,
  \item there are standard contractions $\wh T_{12}$, $\wh R_{21}$ of $f(A)$ such that, in $\gsk \RV[\leq 2]$,
  \[
  [(\wh T_{12}(f(A)), \pr_{\leq 2})] = [(\wh R_{21}(f(A)), \pr_{\leq 2})].
  \]
\end{itemize}
\end{lem}

The key notion for understanding the kernels of the semigroup homomorphisms $\bb L$ is still that of a blowup:

\begin{defn}\label{defn:blowup:coa}
Let $\bm U = (U, f) \in \RV[k]$ with $k > 0$ and, for some $j \leq k$, the coordinate projection $\pr_{\wt j} \rest f(U)$ is finite-to-one. An \emph{elementary blowup} of $\bm U$ is an object $\bm U^{\sharp} = (U^{\sharp}, f^{\sharp})$ such that $U^{\sharp} = U \times \RV^{\circ \circ}$ and, for every $(t, s) \in U^{\sharp}$,
\[
f^{\sharp}_{i}(t, s) = f_{i}(t) \text{ for } i \neq j, \quad f^{\sharp}_{j}(t, s) = s f_{j}(t).
\]
Note that $\bm U^{\sharp}$ is an object in $\RV[\leq k]$ (actually in $\RV[k-1] \amalg \RV[k]$ ) because $f^{\sharp}_{j}(t, \infty) = \infty$.

Let $\bm V = (V, g) \in \RV[k]$, $C \sub V$, and $\bm C = (C, g \rest C) \in \RV[k]$. Let $F : \bm U \fun \bm C$ be an $\RV[k]$-morphism. Then
\[
\bm U^{\sharp} \uplus (V \mi C, g \rest (V \mi C))
\]
is understood as a \emph{blowup of $\bm V$ via $F$}, written as $\bm V^{\sharp}_F$. The subscript $F$ may be dropped in context if there is no danger of confusion. The object $\bm C$ (or the subset $C$) is referred to as the \emph{locus} of the blowup $\bm V^{\sharp}_F$. A \emph{blowup of length $n$} is a composition of $n$ blowups.
\end{defn}

\begin{defn}\label{defn:isp}
Let $\isp[k]$ be the subset of $\ob \RV[\leq k] \times \ob \RV[\leq k]$ that consists of those pairs $(\bm U, \bm V)$ such that there exist isomorphic blowups $\bm U^{\sharp}$, $\bm V^{\sharp}$.

Set $\isp[*] = \bigcup_{k} \isp[k]$.
\end{defn}

At this point we have produced analogues of the auxiliary results as well as the arguments (in Lemmas~\ref{RV:lift}, \ref{bijection:dim:1:decom:RV}, etc.) that the discussion in \cite[\S6]{Yin:int:acvf} formally depends on. The reader is invited to reconstruct the passage that leads to the crucial conclusion \cite[Proposition~6.17]{Yin:int:acvf} (the wording of its proof is improved in \cite[Proposition~5.11]{Yin:int:expan:acvf}) and its corollaries, \textit{mutatis mutandis}; but we shall omit it here.

\begin{prop}\label{kernel:L:dag:coa}
For $\bm U, \bm V \in \RV[\leq k]$,
\[
[\bb L (\bm U)] = [\bb L (\bm V)] \quad \text{if and only if} \quad ([\bm U], [\bm V]) \in \isp.
\]
\end{prop}

This shows that, as a semiring congruence relation on $\gsk \RV[*]$, $\isp$ is generated by the pair
\[
([1], \bm 1_{\K} + [(\RV^{\circ \circ} \mi \{\infty\}, \id)])
\]
and hence its corresponding ideal in the graded ring $\ggk \RV[*]$ is generated by the element $\bm 1_{\K} + [\bm J]$ (see Notation~\ref{nota:RV:short}).

\begin{thm}\label{main:prop:k:vol:dag}
For each $k \geq 0$ there is a canonical isomorphism of Grothendieck semigroups
\[
\int_{+} : \gsk  \VF[k] \fun \gsk  \RV[\leq k] /  \isp
\]
such that
\[
\int_{+} [A] = [\bm U]/  \isp \quad \text{if and only if} \quad  [A] = [\bb L(\bm U)].
\]
Putting these together, we obtain a canonical isomorphism of Grothendieck semirings
\[
\int_{+} : \gsk \VF_* \fun \gsk  \RV[*] /  \isp.
\]
\end{thm}

Recall that, for any definable subset $A \sub \VF^n$, the semimodule of definable functions $A \fun \gsk  \RV[*] /  \isp$ is denoted by $\fn(A, \gsk  \RV[*] /  \isp)$. By Theorem~\ref{main:prop:k:vol:dag}, we have a canonical homomorphism of semimodules:
\[
\int_{+A} : \fn(A, \gsk  \RV[*] /  \isp) \fun \gsk \RV[*] /  \isp.
\]

\begin{prop}\label{semi:fubini}
For all $\bm f \in \fn(A, \gsk  \RV[*] /  \isp)$ and all nonempty subsets $E_1, E_2 \sub [n]$,
\[
\int_{+ a \in \pr_{E_1}(A)} \int_{+ A_a} \bm f = \int_{+ a \in \pr_{E_2}(A)} \int_{+ A_a} \bm f.
\]
\end{prop}

Groupifying $\int_+$ and combining it with Remark~\ref{rem:poin} and Proposition~\ref{prop:eu:retr:k}, we obtain:

\begin{thm}\label{thm:ring}
The Grothendieck semiring isomorphism $\int_+$ naturally induces a ring isomorphism onto $\Z^{(2)}$:
\[
\Xint{\textup{G}}  :  \ggk \VF_* \to \ggk \RV[*] / (\bm 1_{\K} + [\bm J]) \to^{\bb E_{\Gamma}} \Z^{(2)},
\]
and two ring homomorphisms onto $\Z$:
\[
\Xint{\textup{R}}^g, \Xint{\textup{R}}^b: \ggk \VF_* \to \ggk \RV[*] / (\bm 1_{\K} + [\bm J]) \two^{\bb E_{\K, g}}_{\bb E_{\K, b}} \Z.
\]
\end{thm}

\subsection{Integration against the constant $\Gamma$-volume form $1$}\label{sec:mor:vol}

As soon as one considers adding volume forms to definable subsets in $\VF$, even for the simplest one, namely the constant $\Gamma$-volume form $1$, the question of ambient dimension arises and, consequently, one has to take ``essential bijections'' as morphisms.

Let $f : \VF^n \times \RV^m \fun \VF^{n'} \times \RV^{m'}$ be a definable function. For each $(t,s) \in \RV^{m+m'}$ let $f_{t, s}$ be the function $f \cap (\VF^{n + n'} \times \{(t, s)\})$. Note that for some $(t, s)$ we have $\dim_{\VF}(\dom(f_{t, s})) = n$, in other words, $\dom(f_{t, s})$ contains an open polydisc. For such a $(t,s)$ and each $a \in \dom(f_{t, s})$ we define the \emph{$ij$th partial derivative of $f$ at $(a, t)$} to be the $ij$th partial derivative of $f_{t, s}$ at $a$. It follows that every partial derivative of $f$, and hence its Jacobian $\jcb_{\VF} f$ if $n = n'$, are defined almost everywhere.

\begin{defn}[$\vol\VF$-categories]\label{defn:volVF:cat}
A subset $A$ is an object of the category $\vol \VF[k]$ if and only if it is an object of $\VF[k]$ with $\pvf(A) \sub \VF^k$. A \emph{morphism} between two such objects $A$, $A'$ is a definable \emph{essential bijection} $F : A \fun A'$, that is, a bijection that is defined outside of definable subsets of $A$, $A'$ of $\VF$-dimension $< k$, such that $\vv(\jcb_{\VF} F(x)) = 1$ for every $x \in \dom(F)$.

The category $\vol\VF^{\db}[k]$ is the full subcategory of $\vol\VF[k]$ such that $A \in \vol \VF^{\db}[k]$ if and only if $A$ is doubly bounded.

Set $\vol\VF[*] = \coprod_k \vol \VF[k]$ and $\vol\VF^{\db}[*] = \coprod_k \vol \VF^{\db}[k]$.
\end{defn}

\begin{lem}\label{volRV:lift}
Let $F : (U, f) \fun (V, g)$ be a $\vol\RV[k]$-morphism. Then there is a $\vol\VF[k]$-morphism that lifts $F$.
\end{lem}
\begin{proof}
By Lemma~\ref{gk:ortho}, without loss of generality, we may assume that $f = g = \id$ and $\vrv(U)$, $\vrv(V)$ are singletons. By Lemma~\ref{RV:lift}, $F$ can be lifted to a $\VF[k]$-morphism $F^{\uparrow} : \rv^{-1}(U) \fun \rv^{-1}(V)$. By Corollary~\ref{rv:op:comm}, there is a definable subset $U' \sub U$ with $\dim_{\RV}(U') = k' < k$ such that $F^{\uparrow} \rest \rv^{-1}(U \mi U')$ is a $\vol\VF[k]$-morphism. By \omin-minimality, there is a $\vol\RV[k]$-morphism
\[
F' : U' \fun U'' \times \{(1, \ldots, 1)\},
\]
where $U'' \sub (\RV^{\times})^{k'}$; similarly for $V' \coloneqq F(U')$. The induced $\vol\RV[k']$-morphism $F'' : U'' \fun V''$ may be lifted almost everywhere, and so on. If $k' = 1$ then $F''$ may be simply lifted since $\mdl S$ is assumed to be $\VF$-generated. The lemma follows.
\end{proof}

Since all special bijections are $\vol\VF[*]$-morphisms, we may deduce the following analogue of Corollary~\ref{L:sur:c}:

\begin{cor}\label{volL:sur:c}
The lifting map $\bb L_{k}$ induces a surjective homomorphism
\[
\vol \bb L :\gsk \vol \RV[k] \fun \gsk \vol \VF[k].
\]
\end{cor}

Blowups in $\vol \RV[k]$ are constructed as in Definition~\ref{defn:blowup:coa}, except that $\RV^{\circ \circ}$ is replaced by $\RV^{\circ \circ} \mi \{\infty\}$; the corresponding binary relation is still denoted by $\isp[k]$ as in Definition~\ref{defn:isp}; note that we now set $\isp[*] = \coprod_{k} \isp[k]$.

Let $\bm U$, $\bm V$ be two objects in $\vol \RV[k]$ and $\bm U^{\sharp}$, $\bm V^{\sharp}$ two elementary blowups of them. It is easy to deduce that, in $\gsk \vol \RV[k]$, if $[\bm U] = [\bm V]$ then $[\bm U^{\sharp}] = [\bm V^{\sharp}]$. It follows that $\isp[*]$ is a semiring congruence relation on $\gsk \vol \RV[*]$. In fact, the situation in question here may be formally assimilated into \cite[\S 7]{Yin:int:acvf} as a special case and hence, following the discussion there, we obtain the key description of the kernel of $\vol \bb L$:

\begin{prop}\label{kernel:volL:dag:coa}
For $[\bm U], [\bm V] \in \gsk \vol \RV[k]$,
\[
\vol \bb L ([\bm U]) = \vol \bb L ([\bm V]) \quad \text{if and only if} \quad ([\bm U], [\bm V]) \in \isp.
\]
\end{prop}

\begin{thm}\label{main:k:dag}
For each $k \geq 0$ there is a canonical isomorphism of Grothendieck semigroups
\[
\int_{+} : \gsk  \vol \VF[k] \fun \gsk \vol \RV[k] /  \isp
\]
such that
\[
\int_{+} [A] = [\bm U]/  \isp \quad \text{if and only if} \quad  [A] = \vol \bb L ([\bm U]).
\]
Putting these together, we obtain a canonical isomorphism of graded Grothendieck semirings
\[
\int_{+} : \gsk \vol \VF[*] \fun \gsk \vol \RV[*] /  \isp.
\]
\end{thm}

\begin{rem}\label{rem:shr}
The semiring congruence relation $\isp$ is now generated by two pairs
\[
([+, 1], [(\RV^{\circ \circ} \mi \{\infty\}, \id)]) \quad \text{and} \quad  ([-, 1], [(\RV^{\circ \circ} \mi \{\infty\}, \id)]).
\]
In $\gsk \vol \RV[*] /  \isp$, this has the effect of canceling the difference of the sign among elements that are, in a sense, ``disjoint'' from $\gsk \vol\Gamma[*]$; for example, if $U$ and $V$ are two (nonempty) finite definable subsets of $\RV^{\times}$ then
\[
[+, \bm W] [(U, \id)] / \isp = [-, \bm W] [(V, \id)] / \isp.
\]
Thus, in $\gsk \vol \RV[*] /  \isp$, we may set
\[
[\bm A!] \coloneqq [+, \bm T] + [-, \bm T] + [1]
\]
without specifying the sign of $[1]$.

Note that the ideal determined by $\isp$ in the graded ring $\ggk \vol \RV[*]$ is homogeneous, being generated by the corresponding elements $[+, \bm J]$ and $[-, \bm J]$.

The semiring congruence relation on $\gsk \vol \RV[*]$ generated by the pair
\[
([+, 1], [-, 1])
\]
will be denoted by $\ipm$. The semiring $\sgsk \vol \RES [\ast] / \ipm$ is abbreviated as
\[
\dsgsk \vol \RES [\ast],
\]
in which we may speak of the element $[\bm A!]$. Clearly in the Grothendieck \emph{ring} $\dsggk \vol \RES [\ast]$ we have
\[
[+, \bm T] = [-, \bm T] = -[1]
\]
and hence $\dsggk \vol \RES [\ast]$ is actually isomorphic to $\ggk \RES [\ast]$, which in turn is naturally isomorphic to $\Z[X]$ via the assignments $[\bm A] \efun - X$ and $[1] \efun X$, or the other way around.
\end{rem}

\begin{prop}\label{prop:eu:retr:volk}
There are four natural homomorphisms of graded rings
\[
\sbe_{\K, g}^{\pm}, \sbe_{\K, b}^{\pm}: \ggk \vol\RV[*] \fun \dsggk \vol \RES [\ast] \to^{[\bm A!] \efun \pm X} \Z[X]
\]
such that, for the first arrow, they agree on $\ggk \vol\RV^{\db}[*]$ and, in particular, they agree with the natural homomorphism on $\ggk \vol \RES [\ast]$. Moreover, on $\ggk \vol\RV^{\db}[*]$, they factor through the ideal determined by $\isp$.
\end{prop}
\begin{proof}
This is completely analogous to the construction of the homomorphisms $\bb E_{\K, g}$ and $\bb E_{\K, b}$ in Proposition~\ref{prop:eu:retr:k}, but using $\vol \bb D$ instead of $\bb D$.

For the second assertion, we first examine a key example. Let $t \in |\vrv|^{-1}(\gamma)$ be a definable element and $U = \rv(\MM \mi \MM_{\gamma}) \cup \{t\}$. Clearly $[\bm U]/\isp = [1]/\isp$. We have
\[
\sbe_{\K}^{\pm}([\bm U]) = \sbe_{\K}^{\pm}([\bm t]) = \sbe_{\K}^{\pm}([1]).
\]
For the general case, we may go through the same induction as in the proof of Lemma~\ref{hm:factor} below.
\end{proof}

Of course the volume form does not really play any role in $\sbe_{\K, g}^{\pm}$, $\sbe_{\K, b}^{\pm}$ and hence they factor through the forgetful homomorphism $\ggk \vol\RV[*] \fun \ggk \RV[*]$. However, the volume form is needed when we reconstruct these homomorphisms piece by piece through certain zeta functions in \S\ref{sec:zeta}.

\begin{defn}\label{defn:binv}
Let $\beta \in |\Gamma|(\mdl S)$ and $\pi_\beta: \VF \fun \VF / \MM_{\beta}$ be the natural map (see Notation~\ref{nota:tor}). We say that a definable subset $A \sub \VF^n \times \RV^m$ is \emph{$\beta$-invariant} if $A$ is bounded, $\prv(A)$ is doubly bounded, and $A_t$ is a pullback via $\pi_\beta$ for every $t \in \prv(A)$.

Let $\vol \VF^{\diamond}[k]$ be the full subcategory of $\vol \VF[k]$ such that $A \in \vol \VF^{\diamond}[k]$ if and only if $A$ is $\beta$-invariant for some $\beta$. Set $\vol \VF^{\diamond}[\ast] = \coprod_k \vol \VF^{\diamond}[k]$.
\end{defn}

\begin{lem}\label{binv}
For all $A \in \vol \VF^{\diamond}[k]$ there is a $\bm U \in \vol \RV^{\db}[k]$ such that $[A] = \vol \bb L([\bm U])$.
\end{lem}
\begin{proof}
We first remark that Corollary~\ref{all:subsets:rvproduct} does not guarantee that such an object $\bm U$ exists since, in general, special bijections may produce polydiscs around $0$, whose regularizations are not doubly bounded. However, for $\beta$-invariant subsets, this may be remedied by interweaving special bijections with centrifugal transformations at suitable places. To illustrate this, let us consider the simplest nontrivial case $A \sub \VF \times \RV^m$. Let $A'$ be the union of the subsets of $A$ of the form $\MM_{\beta} \times \{t\}$. Since $\beta$ is definable and $\mdl S$ is $\VF$-generated, there is a definable point $b \in \VF$ with $|\vv|(b) = \beta$. Set $u = \rv(b)$. Then there is a centrifugal transformation from $A'$ onto the subset $\rv^{-1}(u) \times \prv(A')$, which is doubly bounded. On the other hand, $A \mi A'$ is $\beta$-invariant and doubly bounded as well. Applying Corollary~\ref{all:subsets:rvproduct} to $A \mi A'$ may produce polydiscs of the form $\MM_{\beta} \times \{t\}$ again, but then we just repeat the above procedure until a doubly bounded $\RV$-pullback is reached.

More generally, when $k>1$, we may concentrate on one $\VF$-coordinate at a time, as done in a standard contraction. This is straightforward and the details are left to the reader.
\end{proof}

\begin{thm}\label{thm:poin}
The Grothendieck semiring isomorphism $\int_+$ naturally induces two homomorphisms of graded rings onto $\Z[X]$:
\[
\Xint{\textup{R}}^{\pm}: \ggk \vol \VF^{\diamond}[\ast] \to \ggk \vol\RV^{\db}[*] / \isp \to^{\sbe_{\K}^{\pm}} \Z[X].
\]
\end{thm}
\begin{proof}
By Theorem~\ref{main:k:dag} and Lemma~\ref{binv}, the first arrow is well-defined. By Proposition~\ref{prop:eu:retr:volk}, the second arrow is well-defined as well.
\end{proof}

\section{Real topological zeta functions}\label{sec:zeta}

Let $\mdl R$ be a polynomial-bounded \omin-minimal field that models $T$. In this section we shall follow the discussion in \cite[\S 8]{hru:loe:lef} closely and work in the power series field $\mdl R \dlbr t^{\Q} \drbr$, with \emph{all} parameters allowed. Recall from Example~\ref{exam:RtQ} that $\Gamma$ may be identified with $\pm e^{\Q}$ and $|\Gamma|$ with $\Q$. We shall also consider $t^q$, $q \in \Q$, as elements in $\RV$.

For each integer $m \geq 1$, let $\vol \RV^{\db}_m[k]$ be the full subcategory of $\vol \RV^{\db}[k]$ such that $(U, f) \in \vol \RV^{\db}_m[k]$ if and only if $\Sigma |\gamma| \in \frac{1}{m} \Z$ for all $\gamma \in \vrv(f(U))$. The categories $\vol \RES_m[k]$, $\vol \Gamma^{\db}_m[k]$, etc., are formulated likewise. Set $\vol \RV^{\db}_m[\ast] = \coprod_k \vol \RV^{\db}_m[k]$, etc. With a little extra bookkeeping, we may construct an isomorphism of graded semirings as in Corollary~\ref{db:vold}:
\[
\vol \bb D^{\db}_m : \gsk \vol\RES_m[*] \otimes \gsk \vol\Gamma^{\db}_m[*] \fun \gsk \vol\RV^{\db}_m[*],
\]
where the tensor product is taken over $\gsk \vol\Gamma^{c}_m[*]$. On the other hand, for any $U \sub (\RV^{\times})^k$, we may take its ``$\frac{1}{m} \Z$-rational'' points:
\[
\Delta_m(U) = \{t \in U : \Sigma |\vrv(t)| \in \tfrac{1}{m} \Z\}.
\]
This naturally determines a surjective homomorphism of graded semirings
\[
\Delta_m : \gsk \vol \RV^{\db}[\ast] \fun \gsk \vol \RV^{\db}_m[\ast].
\]

Let $I$ be an object in $\vol \Gamma^{\db}_m[k]$. Let $\sigma_I : I \fun \Gamma$ be the definable function given by $\gamma \efun \Pi \gamma$ and $|\sigma_I| = |\cdot| \circ \sigma_I : I \fun \frac{1}{m} \Z$. We set
\[
\bm a_m (I) = \sum_{|\beta| \in \dom(|\sigma_I|)} \chi_{\Gamma}(\sigma_I^{-1}(\beta)) \frac{[\sgn(\beta), \bm T^k]}{[\bm A!]^k} \biggl(\frac{[1]}{[\bm A!]} \biggr)^{m |\beta|}
\]
in $\dsgsk \vol \RES [\ast] [[\bm A!]^{-1}]$, where $\chi_{\Gamma}$ is either one of the Euler characteristics $\chi_{\Gamma, g}$ and $\chi_{\Gamma, b}$ (it does not matter which one we choose since $I$ is doubly bounded). It is routine to check that this determines a homomorphism of semirings
\[
\bm a_m : \gsk  \vol \Gamma^{\db}_m [\ast]  \fun \dsgsk \vol \RES [\ast] [[\bm A!]^{-1}].
\]
Let $\bm U = (U, \id)$ be an object in $\vol \RES_m [k]$ such that $\vrv(U)$ is a singleton $\{\gamma\}$. We set
\[
\bm b_m (\bm U) = \frac{[\bm U]}{[\bm A!]^{k}} \biggl( \frac{[1]}{[\bm A!]} \biggr)^{m |\Pi \gamma|}
\]
in $\dsgsk \vol \RES [\ast] [[\bm A!]^{-1}]$. This assignment extends uniquely to a homomorphism of semirings
\[
\bm b_m : \gsk \vol \RES_m [\ast] \fun \dsgsk \vol \RES [\ast] [[\bm A!]^{-1}].
\]
Now, if $U = \vrv^{-1}(\gamma)$ then $[\sgn(\Pi \gamma), \bm T^k] = [\bm U]$ and hence $\bm a_m([\{\gamma\}]) = \bm b_m ([\bm U])$. This means that the tensor product of $\bm b_m$ and $\bm a_m$ over $\gsk \vol\Gamma^{c}_m[*]$ is a well-defined homomorphism:
\[
\bm b_m \otimes \bm a_m : \gsk \vol \RES_m [\ast] \otimes \gsk \vol\Gamma^{\db}_m[\ast] \fun \dsgsk \vol \RES [\ast] [[\bm A!]^{-1}].
\]
Composing it with the inverse of $\vol \bb D^{\db}_m$, we obtain a homomorphism
\[
\bm h_m : \gsk \vol \RV^{\db}_m [\ast] \fun \dsgsk \vol \RES [\ast] [[\bm A!]^{-1}].
\]

\begin{lem}\label{hm:factor}
The homomorphism $\bm h_m$  factors through the semiring congruence relation $\isp$.
\end{lem}
\begin{proof}
It is enough to show that, for all objects $\bm U$, $\bm V$ in $\vol \RV^{\db}_m [1]$ with $[\bm U] / \isp = [\bm V] / \isp =  [1] / \isp$, we have $\bm h_m ([\bm U]) = \bm h_m ([\bm V])$. Let $\bm U^{\sharp}$, $\bm V^{\sharp}$ be blowups of $\bm U$, $\bm V$ such that $[\bm U^{\sharp}] = [\bm V^{\sharp}]$. We may always assume that $\bm U^{\sharp}$ is of the form $(U^{\sharp}, \id)$; similarly for $\bm V^{\sharp}$.

Suppose that $U \sub |\vrv|^{-1}(\frac{p}{m})$ and $\bm U^{\sharp}$ is obtained by blowing up one point in $U$. Consider any $\frac{q}{m} \geq \frac{p}{m}$ and set
\[
U^{\sharp}_{\frac{q}{m}} = \Delta_m(U^{\sharp} \mi |\vrv|^{-1}(\tfrac{q}{m}, \infty)) \uplus \{t^{\frac{q}{m}}\};
\]
note that $U^{\sharp}$ can be recovered from $U^{\sharp}_{\frac{q}{m}}$ by pulling back along $\Delta_m$ and blowing up $t^{\frac{q}{m}}$. We have
\[
\bm a_m([\pm e^{\frac{q}{m}}]) + \bm b_m ([\bm t^{\frac{q}{m}}]) = \frac{[\pm, \bm T]}{[\bm A!]} \biggl(\frac{[1]}{[\bm A!]} \biggr)^{q} + \frac{[1]}{[\bm A!]}\biggl(\frac{[1]}{[\bm A!]} \biggr)^{q} = \bm b_m ([\bm t^{\frac{q-1}{m}}]).
\]
Repeating this calculation, we see that $\bm h_m([\bm U]) = \bm h_m([\bm U^{\sharp}_{\frac{q}{m}}])$. More generally, by induction on the length of the blowup $\bm U^{\sharp}$, this equality holds between any $[\bm U]$ and any suitable sum whose summands are of the form $[\bm U^{\sharp}_{\frac{q}{m}}]$. This implies that, by weak \omin-minimality in the $\RV$-sort, $\bm h_m([\bm U])$ and $\bm h_m([\bm V])$ may be written as the same sum and hence are equal.
\end{proof}

\begin{cor}\label{3vol}
For each integer $m \geq 1$ there is a natural homomorphism of Grothendieck semirings
\[
\bm e_m : \gsk \vol \VF^{\diamond}[*] \to^{\int_+} \gsk \vol \RV^{\db}[*] / \isp \to^{\bm h_m \circ \Delta_m} \dsgsk \vol \RES [\ast] [[\bm A!]^{-1}].
\]
The groupification of $\bm e_m$ yields two ring homomorphisms
\[
\bm e_m^{\pm} : \ggk \vol \VF^{\diamond}[*] \fun \Z[X, X^{-1}].
\]
\end{cor}
\begin{proof}
As in Theorem~\ref{thm:poin}, the first arrow is well-defined. On the other hand, by Lemma~\ref{hm:factor}, the second arrow is well-defined as well. The second assertion follows from Remark~\ref{rem:shr}.
\end{proof}

Henceforth let $f : \mdl R^n \fun \mdl R$ be an $\lan{T}{}{}$-definable non-constant continuous function (in $\mdl R$) sending $0$ to $0$, or more generally a germ at $0$ of such functions. The positive and the negative \emph{Milnor fibers} of $f$ at $0$ are given by
\[
M_+ = B(0, \epsilon) \cap f^{-1}(\delta) \quad \text{and} \quad M_- = B(0, \epsilon) \cap f^{-1}(- \delta),
\]
where $0 < \delta \ll \epsilon \ll 1$ and $B(0, \epsilon)$ is the ball in $\mdl R^n$ centered at $0$ with radius $\epsilon$. By \omin-minimal trivialization (see \cite[\S 9.2.1]{dries:1998}), the (embedded) definable homeomorphism types of $M_+$ and $M_-$ do not depend on the choice of $\delta$ and $\epsilon$ (of course $M_+$ and $M_-$ are not necessarily homeomorphic, definably or not).

By $T$-convexity (see (Ax.~\ref{ax:tcon}) and (Ax.~\ref{ax:match}) of Definition~\ref{defn:tcf}), there is a quasi-$\lan{T}{}{}$-definable continuous function $f^{\uparrow} : \OO^n \fun \OO$, uniquely determined by $f$, that $\res$-contracts to $f$. Let
\begin{align*}
\wt M_+  &= \{ a \in \MM^n : \rv (f^{\uparrow}(a)) = \rv(t)\}, \\
\wt M_-  &= \{ a \in \MM^n : \rv (f^{\uparrow}(a)) = -\rv(t)\}.
\end{align*}
These may be thought of as the Milnor fibers of $f$ with (thickened) formal arcs attached to each point. We shall concentrate on $\wt M_+$ below and omit the subscript, since all results about $\wt M_+$ also hold for $\wt M_-$, \textit{mutatis mutandis}.

\begin{lem}\label{qinv}
$\wt M$ is $q$-invariant for some (and hence for all sufficiently large) $q \in \Q$.
\end{lem}
\begin{proof}
Since $f^{\uparrow}$ is continuous, $\wt M$ is actually clopen. So for every $a \in \wt M$ there is an $a$-definable $\gamma_a \in |\Gamma|$ such that $\go(a, \gamma_a) \sub \wt M$. This naturally gives rise to an $\go$-partition of $\wt M$ and hence the assertion follows from Lemma~\ref{vol:par:bounded}.
\end{proof}

Note that for certain special $f$, such as a polynomial function, there is no need to invoke Lemma~\ref{vol:par:bounded} and we may take $q =1$ in Lemma~\ref{qinv}.

Following \cite{denefloeser:arc}, the motivic zeta functions attached to $f$ are defined to be the generating series
\[
Z^+_{f}(Y) \coloneqq \sum_{m \geq 1} \bm e^+_m([\wt M]) Y^m \quad \text{and} \quad Z^-_{f}(Y) \coloneqq \sum_{m \geq 1} \bm e^-_m([\wt M]) Y^m
\]
in $\Z[X^{-1}]\llbracket Y \rrbracket \simeq \Z[X]\llbracket Y \rrbracket$ (only one of the variables $X$, $X^{-1}$ needs to appear here since $\wt M \sub \OO^n$). Actually these series seem to be closer conceptually to the topological zeta functions defined in \cite{denef:loeser:1992:caract}. We shall write $Z^{\pm}_{f}(Y)$ below when there is no need to distinguish the two zeta functions.

The reader may consult \cite[\S5.1]{DL:tom:seb} or \cite[\S8.4]{hru:loe:lef} for the definition of (formal) rationality of such generating series and other related materials.

\begin{thm}\label{zeta:rat}
The zeta function $Z^{\pm}_{f}(Y)$ is rational and its denominators are products of terms of the form
$1 - (-X)^{a} Y^b$, where $b \geq 1$.
\end{thm}
\begin{proof}
By Corollary~\ref{db:vold}, we may assume that $\int_+ [\wt M]$ is of the form $[\bm U] \otimes [I]$, where $\bm U \in \vol\RES[k]$ and $I \in \vol\Gamma^{\db}[l]$. By Remark~\ref{rem:volG:embed}, $Z^{\pm}_{f}(Y)$ may be written as the Hadamard product of the series
\[
Z^{\pm}(\bm U)(Y) \coloneqq \sum_{m \geq 1} \bm l^{\pm}_m([\bm U]) Y^m \quad \text{and} \quad  Z^{\pm}(I)(Y) \coloneqq \sum_{m \geq 1} \bm l^{\pm}_m([I]) Y^m,
\]
where $\bm l_m^{\pm}$ is defined by the condition $\bm e^{\pm}_m = \bm l^{\pm}_m \circ \int$. By \cite[Propositions~5.1.1, 5.1.2]{DL:tom:seb} (also see \cite[Lemma~8.4.1]{hru:loe:lef}), it is enough to prove the statement for $Z^{\pm}(\bm U)(Y)$ and $Z^{\pm}(I)(Y)$.

By construction, for some positive integers $p$ and $q$, $Z^{\pm}(\bm U)(Y)$ may be understood as the geometric series
\[
\frac{[\bm U]}{[\bm A!]^{k}} \sum_{n \geq 1} \biggl( \frac{[1]}{[\bm A!]} \biggr)^{n q} Y^{n p}
\]
and hence the desired rationality of $Z^{\pm}(\bm U)(Y)$ follows from standard geometric series summation.

Next, recall that the definable function $\sigma_I : I \fun \Gamma$ is given by $\gamma \efun \Pi \gamma$. By \cite[Proposition~4.2.10]{dries:1998}, we may assume that $\sigma_I(I)$ is an open interval and the function on $\sigma_I(I)$ given by $\beta \efun \chi_{\Gamma}(\sigma_I^{-1}(\beta))$ is constant, say, $d \in \Z$. In the simplest case $|\sigma_I(I)| = (0, 1)$, $Z^{\pm}(I)(Y)$ may be understood as the series
\[
d \biggl(\frac{[\bm T]}{[\bm A!]}\biggr)^l \sum_{m > 1} \sum_{0 < n < m} \biggl( \frac{[1]}{[\bm A!]} \biggr)^{n} Y^{m},
\]
which is rational of the desired form. The case that the endpoints of $|\sigma_I(I)|$ are integers follows immediately from this and linearity. More generally, we may assume that $|\sigma_I(I)|$ is of the form $(0, \frac{q}{p})$ and consider the points in $\frac{p}{m} \Z \cap (0, q)$; the desired rationality follows.
\end{proof}

Continuing the proof above, a straightforward computation shows that
\[
\lim_{T \limplies \infty} Z^{\pm}(\bm U)(Y) = - \frac{[\bm U]}{[\bm A!]^{k}} \quad \text{and} \quad \lim_{T \limplies \infty} Z^{\pm}(I)(Y) = - d \biggl(\frac{[\bm T]}{[\bm A!]}\biggr)^l.
\]
Consequently, by \cite[Lemma~8.4.1]{hru:loe:lef} and Theorem~\ref{thm:poin}, we have:

\begin{cor}\label{limitcor}
The zeta function $Z^{\pm}_f(Y)$ attains a limit $e^{\pm}_f \in \Z$ as $Y \limplies \infty$ and we have the equality
\[
e^{\pm}_f X^n = - \Xint{\textup{R}}^{\pm} [\wt M].
\]
\end{cor}

\begin{rem}
More generally, the construction of this section may be carried out with respect to a germ $f$ of non-constant continuous functions on a definable subset $V$ in $\mdl R$. However, it may seem that, for such an arbitrary $V$, the associated zeta function $Z^{\pm}_{V,f}(Y)$ is often rather trivial since $\bm e_m$ ignores those objects in $\vol \VF^{\diamond}[\ast]$ whose dimension is less than the ambient dimension. The proper way to remedy this is to introduce volume forms, that is, we equip each object $A$ in $\vol \VF^{\diamond}[\ast]$ with various definable functions $\omega : A \fun \RV^{\times}$. For the formalism to work out, we shall need a reduced cross-section $\Gamma \fun \RV^{\times}$ or a reduced standard part map $\RV^{\times} \fun \K^{\times}$, as in \cite{Yin:int:expan:acvf}. However, for definable subsets in $\mdl R$, there is no need to appeal to this level of generality. The reason is that volume forms over $\mdl R$ are nullified through the forgetful homomorphism into $\gsk \vol \VF^{\diamond}[*]$ and we may still use $\bm e_m$ in the construction, as long as which volume form is being used is understood (see \cite[Remark~8.2.2]{hru:loe:lef} for further explanation).
\end{rem}


\begin{thebibliography}{10}

\bibitem{Comte:fichou}
Georges Comte and Goulwen Fichou, \emph{Grothendieck ring of semialgebraic
  formulas and motivic real {M}ilnor fibres},  (2012), arXiv:1111.3181.

\bibitem{DL:tom:seb}
Jan Denef and Fran\c{c}ois Loeser, \emph{Motivic exponential integrals and a
  motivic {T}hom-{S}ebastiani theorem}, Duke Mathematical Journal \textbf{99}
  (1999), 285--309.

\bibitem{denefloeser:arc}
\bysame, \emph{Geometry on arc spaces of algebraic varieties}, European
  Congress of Mathematics, Progress in Mathematics, vol. 201, Birkh\"{a}user
  Basel, 2001, arXiv:math/0006050, pp.~327--348.

\bibitem{denef:loeser:1992:caract}
Jan Denef and Fran{\c{c}}ois Loeser, \emph{Caract\'eristiques
  d'{E}uler-{P}oincar\'{e}, fonctions z\^{e}ta locales et modifications
  analytiques}, Journal of the American Mathematical Society \textbf{5} (1992),
  no.~4, 705--720.

\bibitem{Dries:tcon:97}
Lou van~den Dries, \emph{{$T$}-convexity and tame extensions {II}}, Journal of
  Symbolic Logic \textbf{62} (1997), no.~1, 14--34.

\bibitem{dries:1998}
\bysame, \emph{Tame topology and \omin-minimal structures}, LMS Lecture Note
  Series, vol. 248, Cambridge University Press, Cambridge, UK, 1998.

\bibitem{DriesLew95}
Lou van~den Dries and Adam~H. Lewenberg, \emph{{$T$}-convexity and tame
  extensions}, Journal of Symbolic Logic \textbf{60} (1995), no.~1, 74--102.

\bibitem{DMM94}
Lou van~den Dries, Angus Macintyre, and David Marker, \emph{The elementary
  theory of restricted analytic fields with exponentiation}, Annals of
  Mathematics \textbf{140} (1994), no.~1, 183--205.

\bibitem{DrMi96}
Lou van~den Dries and Chris Miller, \emph{Geometric categories and
  \omin-minimal structures}, Duke Mathematical Journal \textbf{84} (1996),
  no.~2, 497--540.

\bibitem{DriesSpei:2000}
Lou van~den Dries and Patrick Speissegger, \emph{The field of reals with
  multisummable series and the exponential function}, Proc. London Math. Soc.
  \textbf{81} (2000), no.~3, 513--565.

\bibitem{holly:can:1995}
Jan Holly, \emph{Canonical forms for definable subsets of algebraically closed
  and real closed valued fields}, Journal of Symbolic Logic \textbf{60} (1995),
  no.~3, 843--860.

\bibitem{hrushovski:kazhdan:integration:vf}
Ehud Hrushovski and David Kazhdan, \emph{Integration in valued fields},
  Algebraic geometry and number theory, Progr. Math., vol. 253, Birkh\"{a}user,
  Boston, MA, 2006, math.AG/0510133, pp.~261--405.

\bibitem{hru:loe:lef}
Ehud Hrushovski and Fran\c{c}ois Loeser, \emph{Monodromy and the {L}efschetz
  fixed point formula}, arXiv:1111.1954, 2011.

\bibitem{kage:fujita:2006}
M.~Kageyama and M.~Fujita, \emph{Grothendieck rings of $o$-minimal expansions
  of ordered abelian groups}, Journal of Algebra \textbf{299} (2006), 8--20,
  arXiv:math/0505341v1.

\bibitem{mac:mar:ste:weako}
Dugald Macpherson, David Marker, and Charles Steinhorn, \emph{Weakly
  \omin-minimal structures and real closed fields}, Transactions of the
  American Mathematical Society \textbf{352} (2000), no.~12, 5435--5483.

\bibitem{jana:omin:res}
Jana Ma\v{r}\'{i}kov\'{a}, \emph{\omin-minimal residue fields of \omin-minimal
  fields}, Annals of Pure and Applied Logic \textbf{162} (2011), no.~6,
  457--464.

\bibitem{mccrory:paru:virtual:poin}
Clint McCrory and Adam Parusi\'{n}ski, \emph{Virtual {B}etti numbers of real
  algebraic varieties}, C. R. Math. Acad. Sci. Paris \textbf{Ser. I 336}
  (2003), no.~9, 763--768, arXiv:math/0210374.

\bibitem{pet:star:otop}
Ya'acov Peterzil and Sergei Starchenko, \emph{Computing \omin-minimal
  topological invariants using differential topology}, Transactions of the AMS
  \textbf{359} (2007), no.~3, 1375--1401.

\bibitem{yin:hk:part:3}
Yimu Yin, \emph{{F}ourier transform of the additive group in algebraically
  closed valued fields}, arXiv:0903.1097, submitted, 2009.

\bibitem{Yin:QE:ACVF:min}
\bysame, \emph{Quantifier elimination and minimality conditions in
  algebraically closed valued fields}, arXiv:1006.1393v1, 2009.

\bibitem{Yin:special:trans}
\bysame, \emph{Special transformations in algebraically closed valued fields},
  Annals of Pure and Applied Logic \textbf{161} (2010), no.~12, 1541--1564,
  arXiv:1006.2467.

\bibitem{Yin:int:acvf}
\bysame, \emph{Integration in algebraically closed valued fields}, Annals of
  Pure and Applied Logic \textbf{162} (2011), no.~5, 384--408,
  arXiv:0809.0473v2.

\bibitem{Yin:int:expan:acvf}
\bysame, \emph{Integration in algebraically closed valued fields with
  sections}, Annals of Pure and Applied Logic \textbf{164} (2013), no.~1,
  1--29, arXiv:1204.5979v2.

\end{thebibliography}
\end{document}